\documentclass[reqno,10pt,centertags]{amsart}
\usepackage{amsmath,amsthm,amscd,amssymb,latexsym,upref,stmaryrd}

\usepackage{srcltx}
\usepackage{graphicx}
\usepackage{subfigure}
\usepackage{psfrag}
\usepackage{rotating}

\newcommand{\R}{{\mathbb R}}
\newcommand{\N}{{\mathbb N}}

\newcommand{\C}{{\mathbb C}}

\newcommand{\bbC}{{\mathbb{C}}}

\newcommand{\bbH}{{\mathbb{H}}}

\newcommand{\bbM}{{\mathbb{M}}}

\newcommand{\bbR}{{\mathbb{R}}}

\newcommand{\bv}{{\mathbf{v}}}
\newcommand{\bw}{{\mathbf{w}}}

\newcommand{\cC}{{\mathcal C}}
\newcommand{\cD}{{\mathcal D}}
\newcommand{\cE}{{\mathcal E}}
\newcommand{\cF}{{\mathcal F}}
\newcommand{\cG}{{\mathcal G}}
\newcommand{\cH}{{\mathcal H}}

\newcommand{\cK}{{\mathcal K}}
\newcommand{\cL}{{\mathcal L}}
\newcommand{\cM}{{\mathcal M}}
\newcommand{\cN}{{\mathcal N}}

\newcommand{\cU}{{\mathcal U}}

\newcommand{\cW}{{\mathcal W}}

\newcommand{\cY}{{\mathcal Y}}

\newcommand{\il}{i\lambda^{1/2}}

\renewcommand{\a}{\alpha}

\renewcommand{\d}{\delta}

\renewcommand{\k}{\varkappa}
\renewcommand{\l}{\lambda}

\renewcommand{\o}{\omega}

\newcommand{\s}{\sigma}

\newcommand{\x}{\xi}
\newcommand{\z}{\zeta}

\newcommand{\G}{\Gamma}

\DeclareMathOperator{\meas}{meas}

\DeclareMathOperator{\dist}{dist}
\DeclareMathOperator{\supp}{supp}
\DeclareMathOperator{\rank}{rank}

\DeclareMathOperator{\dom}{dom}

\DeclareMathOperator{\diag}{diag}
\DeclareMathOperator{\adj}{adj}

\DeclareMathOperator{\Span}{span}

\DeclareMathOperator{\sing}{sing}

\renewcommand{\Re}{\text{\rm Re}}
\renewcommand{\Im}{\text{\rm Im}}

\newcommand{\loc}{\text{\rm{loc}}}

\newcommand{\beq}{\begin{equation}}
\newcommand{\enq}{\end{equation}}

\newcommand{\lb}{\label}
\newcommand{\f}{\frac}

\newcommand{\wti}{\widetilde}

\newcommand{\bi}{\bibitem}

\renewcommand{\ge}{\geqslant}
\renewcommand{\le}{\leqslant}
\newcommand{\vect}[1]{\begin{pmatrix}#1\end{pmatrix}}

\numberwithin{equation}{section}

\renewcommand{\P}{{\mathcal P}}

\allowdisplaybreaks \numberwithin{equation}{section}

\newtheorem{theorem}{Theorem}[section]

\newtheorem{proposition}[theorem]{Proposition}
\newtheorem{lemma}[theorem]{Lemma}
\newtheorem{corollary}[theorem]{Corollary}
\newtheorem{definition}[theorem]{Definition}

\newtheorem{hypothesis}[theorem]{Hypothesis}
\theoremstyle{remark}
\newtheorem{remark}[theorem]{Remark}
\newtheorem{remarks}[theorem]{Remarks}

\newtheorem{example}[theorem]{Example}

\begin{document}

\numberwithin{equation}{section}
\allowdisplaybreaks

\title[Eigenvalues and boundary value problems]{Finding eigenvalues of 
holomorphic Fredholm operator pencils using boundary value problems and contour 
integrals}

  \author[W.-J.\ Beyn,  Y.\ Latushkin and J. Rottmann-Matthes]{Wolf-J\"urgen Beyn, Yuri Latushkin and Jens Rottmann-Matthes \hspace{80pt}}

\address{Department of Mathematics, Bielefeld University, 
P.O. Box 100131, D-33501, Bielefeld, Germany}
\email{beyn@math.uni-bielefeld.de}
\urladdr{http://www.math.uni-bielefeld.de/~beyn/AG\_Numerik}

\address{Department of Mathematics,
University of Missouri, Columbia, MO 65211, USA}
\email{latushkiny@missouri.edu}
\urladdr{http://www.math.missouri.edu/personnel/faculty/latushkiny.html}

\address{Department of Mathematics, Bielefeld University, 
P.O. Box 100131, D-33501, Bielefeld, Germany}
\email{jrottman@math.uni-bielefeld.de}
\urladdr{http://www.math.uni-bielefeld.de/~jrottman/}

\thanks{Supported by DFG CRC 701 'Spectral Structures and Topological Methods in
Mathematics' at Bielefeld University and partially by the US National Science
Foundation under Grants  DMS-0754705, DMS-1067929, by the Research Board and Research Council of the University of Missouri.}

\date{\today}

\hspace*{-3mm} 
\begin{abstract} 
  Investigating the stability of nonlinear waves often leads to linear or
  nonlinear eigenvalue problems for differential operators on unbounded
  domains.

  In this paper we propose to detect and approximate the point spectra
  of such operators (and the associated eigenfunctions) via contour
  integrals of solutions to resolvent equations.
  The approach is based on Keldysh' theorem and extends a recent method
  for matrices depending analytically on the eigenvalue parameter.
  We show that errors are well-controlled under very general assumptions
  when the resolvent equations are solved via boundary value problems on 
  finite domains. Two applications are presented: an analytical study
  of Schr\"odinger operators on the real line as well as 
  on bounded intervals and a numerical study of the FitzHugh-Nagumo
  system.

  We also relate the contour method to the well-known Evans function and
  show that our approach provides an alternative to evaluating and
  computing its zeroes. 
\end{abstract}

\maketitle

{\scriptsize{\tableofcontents}}
\normalsize


\section{Introduction}  \lb{s:Intro}

Studies of the analytic Evans function (\cite{AGJ90,PW92}) have
found numerous applications to the stability problem for nonlinear waves.
The function has proved to be an invaluable tool  for locating  point
spectra of differential operators that are defined on the real line or,
more generally, on cylindrical domains. In particular, analyzing the 
behavior of
the Evans function near zero or near infinity provides one major step in
 studying point spectra rigorously. 
We refer to  \cite{Sa02}, \cite{KK04}, \cite{GLM07},\cite{OS10}, \cite{DL}
for a variety of applications and extensions of the concept
to other settings. 

There have been quite a few approaches to computing the Evans function 
numerically and then determine from its zeros the unknown eigenvalues,
see \cite{Br01}, \cite{BDG02}, \cite{BZ02}, \cite{HZ06}, \cite{HSZ06},
\cite{MN08}, \cite{LMT10}, \cite{SE10}. 
The standard definition of the Evans function (see \cite{Sa02})
involves a determinant of vectors which depend analytically on the eigenvalue parameter and which are determined as initial values of
 exponentially decaying solutions on both semi-axes. This leads
to the problem of integrating stiff ODEs while keeping the analyticity
with respect to the parameter. For higher dimensions the evaluation
of a determinant may also lead to instabilities unless a proper scaling
is employed, see for instance \cite[Thm.4.15]{GLZ} for a discussion of the appropriate scaling of the Evans function. The techniques proposed to solve these problems utilize
exterior products \cite{Br01},\cite{BZ02},\cite{AB02}, \cite{BDG02}
or solve Kato's matrix differential equation \cite[Sec.II.4.2]{K80},
see \cite{HSZ06}. 

In this paper we pursue an alternative road that avoids the
 intermediate stage of computing the Evans function. Rather we propose to solve the
 original analytical eigenvalue problem in the whole space via well-posed boundary value problems on finite domains while preserving analyticity. 
For the latter problem we propose a numerical method using contour
integrals of solutions to resolvent equations. 
The approach is based on Keldysh' theorem and extends
a recent method  \cite{B12} for nonlinear eigenvalue problems with matrices.
Our goal is to determine all eigenvalues inside a given contour
$\Gamma \subset \C$ and to guarantee that, at each stage
of approximation, approximate equations are as  well-conditioned as the 
original eigenvalue problem. The method applies
 to general nonlinear eigenvalue problems 
\begin{equation} \label{noneig}
F(\lambda) v =0, \quad \lambda \in \Omega,
\end{equation}
where $F(\l)$ are Fredholm operators of index $0$ that depend analytically on
$\l \in \Omega$.
For certain right-hand sides $v$ and functionals $w$ we require to evaluate
the integrals
\begin{equation} \label{contougen}
\frac{1}{2 \pi i}\int_{\G}\langle w,  F(\l)^{-1}v \rangle \,d\l \quad \text{and} \quad
\frac{1}{2 \pi i}\int_{\G} \l \, \langle w, F(\l)^{-1}v\rangle \, d\l,
\end{equation}
where $\Gamma$ lies in the resolvent set.
The idea first appeared for the matrix case in the papers \cite{A08},\cite{A09}
where it was used in connection with the Smith normal form. In \cite{B12}
we arrived independently at similar expressions (without applying the functional
$w$) by using the theorem of Keldysh, see \cite[Ch.1]{MM}. It turns out that the
Keldysh approach allows to locate eigenvalues precisely, to handle multiplicities and to simultaneously compute
(generalized) eigenvectors, see Section \ref{s:multipleev}. 
Moreover, Keldysh' Theorem works for abstract
Fredholm operators \cite{MM} and thus allows to generalize the whole
approach as we will show in Section \ref{s:abst}. Note that 
contour integrals are normally used for computing the winding number
and thus the number of zeros inside the contour  (see \cite{Br01}, \cite{BDG02}, \cite{BZ02}
in case of the Evans function), whereas our method allows to locate
all zeros inside the contour and, in addition, to obtain approximate eigenfunctions (Theorems \ref{simpleeigen}, \ref{multieigen} and
equation \eqref{vapprox}).

In Section \ref{s:evfun} we discuss suitable normalizations of
the Evans function and relate our contour method to such a
normalized Evans function (Theorems \ref{zerobehavior}, \ref{formulaE}).
 Then we apply the method
to first order differential systems with 
$\l$-dependent matrices. We continue this in Section \ref{s:converge} and show that isolated eigenvalues are well approximated when the resolvent equations are solved on bounded intervals
with suitable boundary conditions (Theorems \ref{thm:EConv}, \ref{thm:EvalsConv}). Here we follow \cite{BR07} where
it is shown that the functional analysis of discretization methods
\cite{V76}, \cite{V80} applies to this situation. In Section \ref{s:appl}
we apply the theory to Schr\"odinger operators on the real line and express
the various quantities of our approach in terms of Jost solutions and Levinson's Theorem (Theorems \ref{thm:singformjost}, \ref{EntoE}).
 Section \ref{sec:num} concludes with several numerical
experiments for the FitzHugh-Nagumo system which demonstrate
the robustness of the contour method. In particular, we show how computational
errors depend on the length of the finite interval, the number
of quadrature points used for \eqref{contougen}, and on the rank test
used to determine the number of eigenvalues inside the contour.


\section{Abstract results for the contour method}\label{s:abst}
\subsection{The abstract setting}
\label{s:abstset}
We consider nonlinear eigenvalue problems as in  \eqref{noneig}
where $F:\Omega\to\cL(\cH,\cK)$ is a holomorphic function on a
 domain $\Omega\subseteq\bbC$ with values in the space $\cL(\cH,\cK)$ of bounded linear operators 
 from some complex Banach space $\cH$ into another Banach space $\cK$. We will assume that the operators $F(\lambda)$ are Fredholm of index $0$ for all 
$\lambda \in \Omega$ (this will be written as $F\in \bbH(\Omega,\cF(\cH,\cK))$). As usual, cf.\ \cite[Ch.I]{MM}, we define the resolvent set and the spectrum of the operator pencil $F$ by
\[\rho(F)= \{\l\in \Omega: F(\l)\,\mathrm{invertible}  \}, \qquad 
\sigma(F)= \Omega \setminus \rho(F).
\]
Throughout, we impose the following assumptions.
\begin{hypothesis}\label{hyp0} 
Assume that $F\in \bbH(\Omega,\cF(\cH,\cK))$ and $\rho(F)\neq\emptyset$.
\end{hypothesis}
 Under these assumptions $\sigma(F)$ is a discrete subset of $\Omega$ and  the operator valued function $F^{-1}(\cdot)$ is meromorphic, see \cite[Thm.1.3.1]{MM}. In the terminology of \cite[Ch.IX]{GGK} the function  $F^{-1}(\cdot)$ is finitely meromorphic.
In general, we follow the setting in \cite[Ch.I]{MM}. In particular, we use
dual spaces $\cH'$, $\cK'$ and denote the dual pairing by elements $w\in \cH'$,
$u \in \cK'$ in two equivalent ways
\begin{equation}\label{fdual} w^\top v=\langle w, v \rangle, \quad v\in \cH, \qquad
u^\top v=\langle u, v \rangle, \quad v\in \cK.
\end{equation}
Thus, the dual space is the space of linear (versus complex conjugate linear) functionals.  As noted in \cite[Sec.1.1]{MM} it is important to avoid the complex adjoint 
space and complex conjugate functionals (even if $\cH$ and $\cK$ are Hilbert 
spaces),
since this will turn holomorphic functions into antiholomorphic ones. We prefer
the notion $w^\top$ over $w^*$ and $w'$ because it is in accordance with matrix-vector 
notation. Correspondingly, we denote the dual of an operator 
$A\in \cL(\cH,\cK)$ by $A^\top\in \cL(\cK',\cH')$. It is defined by
\[ \langle A^\top w, v\rangle =(A^\top w)^\top  v = w^\top (Av) = \langle w, Av\rangle,
\quad w\in \cK',v\in \cH.
\] 
 Any two elements $v\in\cH$ and $w\in \cK'$ define an operator
$v w^\top\in \cL(\cK,\cH) $ as follows
\begin{equation} \label{dyadic}
 (v w^\top) u   :=v \,(w^\top u) = v\,\langle w, u \rangle, \qquad u\in \cK.
\end{equation}
Because of this definition we omit brackets and simply write $v w^\top u$.
Note that the operator $v w^\top$ is of rank one if $v\neq0,w\neq0$. We
denote by $\mathcal{N}(\cdot)$ the null-space and by
$\mathcal{R}(\cdot)$ the range of a linear operator.

Let $\G$ be a smooth contour in $\Omega$ surrounding a bounded subdomain 
$\mathrm{int}(\G)=\Omega_0\subseteq \Omega$. Since  $\Omega_0\cup \G$ is
compact and eigenvalues are isolated \cite[Thm.1.3.1]{MM} there are at
most finitely many eigenvalues  $\l_1,\l_2,\dots,\l_\k\in\Omega_0$,
 and, in addition, they are of finite multiplicity. 
Our goal is to determine these eigenvalues and good approximations of the
eigenvectors by computing contour integrals of the type \eqref{contougen}

In the following choose $m$ linearly independent functionals $\widehat{w}_j\in \cH'$, 
$j=1,\dots,m$,
and $\ell$ linearly independent vectors $\widehat{v}_k\in \cK$, $k=1,\dots,\ell$.
Let us assume that $F(\l)$ is invertible for all $\l\in\Gamma$. Then we can solve 
the following equations for unknown vectors $y_k(\l)\in \cH$:
\begin{equation}\label{yeq}
F(\l)y_k(\l)=\widehat{v}_k,\quad \l\in\G,\; k=1,\dots,\ell.
\end{equation}
We introduce an $(m\times\ell)$-matrix valued function, $E(\cdot)$, on $\G$ as
 follows:
\begin{equation}\lb{defE}
E(\l)=\Big(\langle \widehat{w}_j, y_k(\l) \rangle \Big)_{j,k=1}^{m,\ell},
\quad\l\in\G.
\end{equation}

In addition, we introduce the following $(m\times\ell)$ matrices:
\beq\lb{defB01}D_0=\frac{1}{2\pi i}\int_\G E(\l)\,d\l,\quad
D_1=\frac{1}{2\pi i}\int_\G \l E(\l)\,d\l.\enq
\begin{remark}\lb{remB01}
 We note that equation \eqref{yeq} is the only information used to obtain 
$E(\l)$ and thus the matrices $D_0$ and $D_1$. So, replacing \eqref{yeq}
 by an approximate equation with good stability properties, the 
matrices $D_0,D_1$ can be computed using approximation arguments.
In our applications, equation \eqref{yeq} is a differential equation on the 
line, while the respective approximation is a differential equation on a 
finite segment with appropriate boundary conditions. We will discuss the
errors of this approximation in Sections \ref{s:approxop} and \ref{s:converge} below. 
In addition, there are errors caused by
approximating the contour integrals \eqref{defB01} by a quadrature
rule, e.g. the trapezoid sum. This error depends on the number of
quadrature points and on the distance of eigenvalues to the contour.
 For a detailed  analysis 
in case of analytic contours $\Gamma$ we refer to \cite{B12}. In any rate, we may assume in what follows that the matrices $E(\lambda)$, $D_0$, $D_1$ are known to us.
\hfill$\Diamond$\end{remark}
\begin{remark} \lb{remB02}
At first sight it seems unneccesarily general to have different
dimensions $m$ and $\ell$  that may also differ from $\k$. However,
it is important for practical computations. First of all, $\k$ is generally
unknown and thus a quantity to be determined. Our approach will only require
$m,\ell\ge \k$ which can be achieved by increasing $m$ and $\ell$. Second,
the number $\ell$ determines the number of equations \eqref{yeq} to
be solved and hence the numerical effort. Therefore we want to keep
it as small as possible. Finally, $m$ can be very large since it gives
the number of functionals that we can evaluate on the solutions of
\eqref{yeq}. In principle we can assume to know $y_j(\l)$ exactly
and hence all values $w^\top y_j(\l), w\in \cK'$. In this case 
we may think of $E(\l)$ being a matrix with infinitely long columns,
see Example \ref{ex:compactf}. 
\hfill$\Diamond$\end{remark}

\subsection{Simple eigenvalues inside the contour}
\lb{s:simpleev}
In this subsection, in addition to Hypothesis \ref{hyp0}, we assume that $F$ has only simple eigenvalues inside
the contour $\Gamma$. That is, for some $\k\ge0$,
\begin{equation} \label{evint}
 \sigma(F)\cap \Omega_0=\{\l_{1},\ldots \l_{\k} \},
\end{equation}
$\dim \cN(F(\l_n))=1$, and there are eigenvectors $v_n\in \cH$ of the operator $F(\l_n)$, the eigenvectors $w_n\in \cK'$ of the dual operator $F(\l_n)'$ such that, cf.\ \cite[Def.1.7.1]{MM},
\begin{equation} \lb{simplevw}
F(\l_k)v_n=0, \quad w_n^\top F(\l_n)=0, \quad  w_n^\top F'(\l_n) v_n \neq 0,
\quad n=1,\ldots,\k.
\end{equation}
In the following it will be convenient to normalize $v_n,w_n$ such that
\begin{equation} \label{normalize}
w_n^\top F'(\l_n)v_n =1, \quad n=1,\ldots,\k.
\end{equation}

We recall the well-known Keldysh formula, see \cite[Thm.1.6.5]{MM}:
\beq\lb{KF}F(\l)^{-1}=\sum_{n=1}^\k \frac{1}{\l-\l_n}v_n w_n^\top +H(\l),\;
\l\in\Omega_0\setminus\{\l_1,\dots,\l_\k\},\enq
where the normalization \eqref{normalize} is assumed and
$H(\cdot)$ is a holomorphic function on a neighborhood $\cU$ of $\Omega_0$ with values in 
$\cL(\cK,\cH)$. We also use the notation from \eqref{dyadic}.

Using \eqref{KF} in \eqref{defE} yields for $j=1,\ldots,m$, $k=1,\ldots,\ell$
the following formula for the entries of the matrix $E(\l)$:
\begin{equation} \label{eexpress}
E_{jk}(\l)=\sum_{n=1}^{\varkappa}\frac{1}{\l-\l_n} \langle \widehat{w}_j,v_n\rangle
\langle w_n,\widehat{v}_k \rangle + \langle \widehat{w}_j,H(\l)\widehat{v}_k(\l)\rangle.
\end{equation}
We write this in matrix form by introducing 
the rectangular matrices $G_l\in \C^{m,\k}$ and $G_r\in \C^{\ell,\k}$,
\begin{equation}\label{defGrl}
 G_l=\Big(\langle \widehat{w}_j, v_n\rangle \Big)_{j=1, n=1}^{m,\k}, \quad
G_r=\Big(\langle w_n, \widehat{v}_k\rangle\Big)_{k=1, n=1}^{\ell,\k},
\end{equation}
the diagonal matrix $\Lambda=\diag\{\l_1,\dots,\l_\k\}$, and the
matrix function
\[  H_0(\l) = \left(\langle \widehat{w}_j,H(\l)\widehat{v}_k(\l)\rangle\right)_{j=1,k=1}^{m,\ell}.\]
 Then \eqref{eexpress} has the matrix form
\begin{equation} \label{Ematrep}
E(\l)= G_l(\l I_{\varkappa}-\Lambda)^{-1} G_r^{\top}+ H_0(\l), \quad 
\l \in \Omega,
\end{equation}
where $H_0 \in \bbH(\cU,\C^{m,\ell})$.
Using  Cauchy's Theorem,
 we evaluate the matrix $D_0$ from \eqref{defB01} as follows:
\begin{equation} \label{FB0} 
D_0=\frac{1}{2 \pi i}\int_\G G_l(\l I_{\varkappa}-\Lambda)^{-1} G_r^{\top} d\l
=G_l G_r^\top.
\end{equation}
A similar calculation yields
\beq\lb{FB1} D_1=G_l \Lambda G_r^\top.\enq
Our standing assumption in the following will be that both matrices
$G_l$ and $G_r$  have rank $\varkappa$. As we  noted in Remark \ref{remB02} 
this requires to have
$\ell\ge \k$ and $ m \ge \k$.
Equation \eqref{FB0}  then implies the following formula for 
the number of eigenvalues  $\l_k$ enclosed by $\G$:
\beq\lb{vkrank}\k=\rank D_0.\enq
Thus, cf.\ Remark \ref{remB01}, one can compute the number $\k$ from a rank test
of $D_0$ and this should be sufficiently robust to approximation arguments. 
 
Next, we show how $D_1$ can be used to evaluate the actual location of the 
eigenvalues $\l_n$ enclosed by $\G$. And this procedure should be
robust to approximation arguments as well.

Let $\s_1,\dots,\s_\k$ denote the  nonzero singular values of the matrix
$D_0$, and introduce the diagonal $(\k\times\k)$ matrix
$\Sigma_0=\diag\{\s_1,\dots,\s_\k\}$. We  use the short form of the singular value
decomposition of $D_0$ (e.g. \cite[\S 3.2]{A05})
\begin{equation}\lb{singval}
D_0=V_0\Sigma_0W_0^\ast,\; V_0\in \C^{m,\k},\; V_0^\ast V_0=I_{\k}, \;
 W_0\in \C^{\ell,\k}, \;W_0^\ast W_0=I_{\k}.
\end{equation}
Note that here one uses adjoint matrices 
$W_0^*=\overline{W_0}^\top$
and $V_0^*=\overline{V_0}^\top$.

Due to equation  \eqref{FB0}, we have
\beq\lb{B0V} D_0=V_0\Sigma_0W_0^\ast=G_lG_r^\top.\enq
Since the columns of the matrices $G_l$ and $V_0$ span the same subspace, 
there is a nonsingular $(\k\times\k)$ matrix $S$ such that $V_0S=G_l$.
From this we obtain
\[\Sigma_0 W_0^* =S G_r^\top \quad \text{and} \quad  G_r^\top =S^{-1}\Sigma_0 W_0^*.
\]
Using this in equation \eqref{FB1} we find
\[ D_1= V_0S \Lambda S^{-1} \Sigma_0 W_0^*.
\]
 Multiplying the last equation by $V_0^\ast$ from the left and by 
$W_0\Sigma_0^{-1}$ from the right, yields
\beq\lb{eqfL}
S\Lambda S^{-1}=V_0^\ast D_1W_0\Sigma_0^{-1} =: D,\enq
which, in turn, implies the desired formula for the  $\l_k$'s:
\beq\lb{lkf}
\{\l_1,\dots,\l_\k\}=\s(D).
\enq
We stress again that the spectrum of the matrix $D$ in  
\eqref{eqfL}, by \eqref{defE}, \eqref{defB01}, can be evaluated using only the
data from \eqref{yeq}, and thus can be obtained using approximation arguments.
We summarize our result in the following theorem.
\begin{theorem} \label{simpleeigen}
Suppose that the operator pencil $F\in \bbH(\Omega,\cF(\cH,\cK))$ has
 only simple eigenvalues 
$\lambda_1,\ldots,\lambda_{\varkappa}$ inside a simple closed contour $\Gamma$ in
$\Omega$ and no eigenvalues on the contour. Given a set of linearly independent
functionals  $\widehat{w}_j\in \cH'$, $j=1,\dots,m$, and vectors
 $\widehat{v}_k\in \cK$, $k=1,\dots,\ell$, with $m,\ell\ge \varkappa$,
let the matrices $D_0,D_1\in \mathbb{C}^{m \times \ell}$ be determined by 
\eqref{yeq},\eqref{defE},\eqref{defB01} and assume that the matrices
$G_l,G_r$ from \eqref{defGrl} have maximum rank.
Then the $\varkappa\times \varkappa$-matrix 
$D=V_0^\ast D_1W_0\Sigma_0^{-1}$, where $V_0,\Sigma_0,W_0$ are
given by the singular value decomposition \eqref{singval} of $D_0$, has
$\lambda_1,\ldots,\lambda_{\varkappa}$ as simple eigenvalues.
\end{theorem}

As soon as $\Lambda$ is determined, one can view \eqref{eqfL}, that is, the equation
\beq\lb{eqfL2} DS=S\Lambda,\enq
as an equation for the corresponding to $\l_n$ eigenvectors of the matrix $D$, which are the columns of the matrix $S$.
With $S$ known we can compute $G_l=V_0S$, i.e. the values
$\langle \widehat{w}_j,v_n\rangle$ are available. From this we determine good 
approximations of eigenvectors $v_n,n=1,\ldots,\k$, as follows:\\
{\bf Step 1:} Select  vectors $\widehat{u}_k\in \cH,k=1,\ldots,m$, that are
biorthogonal to the $\widehat{w}_j$'s, 
\begin{equation} \lb{biorth}
\langle \widehat{w}_j, \widehat{u}_k \rangle = \delta_{jk}, \quad j,k=1,\ldots,m.
\end{equation}
{\bf Step 2:}
Determine coefficients 
$\beta_{k,n},k=1,\ldots,m$, which minimize
\begin{equation} \label{minfunc}
  \Phi(\beta)=\sum_{j=1}^m | \langle \widehat{w}_j, v_n - 
\sum_{k=1}^m \beta_{k,n}\widehat{u}_k \rangle|^2.
\end{equation}
Due to \eqref{biorth} the minimum is attained at
 $\beta_{k,n}=\langle \widehat{w}_k,v_n \rangle = (G_l)_{k,n}$.\\
{\bf Step 3:}
Determine approximate eigenvectors of the operator $F(\l_n)$ from
\begin{equation} \label{vapprox}
v_{n}^{\mathrm{approx}}=\sum_{k=1}^{m} (G_l)_{k,n} \widehat{u}_k, \quad n=1,\ldots,\k.
\end{equation}
If $\cH$ is a Hilbert space then we can identify $\widehat{w}_j=\widehat{u}_j, j=1,\ldots,m$, so that \eqref{biorth} requires these vectors to form an orthonormal
system. 
Introducing the subspace $X_m = \mathrm{span}\{\widehat{u}_1,\ldots, \widehat{u}_m \}$ we find that  $\Phi(\beta)$ in \eqref{minfunc} agrees 
with $\|v_n-\sum_{k=1}^m \beta_{k,n}\widehat{u}_k\|^2$ up to a constant.
Hence $v_n^{\mathrm{approx}}$ is the best approximation of $v_n$ in the
subspace $X_m$.
If the functions $\widehat{u}_j$ are not orthonormal, one can find the
best approximation of $v_n$ in $X_m$ by first solving for $\beta_{j,n}$ the linear
system
\[ \sum_{j=1}^{m} \langle \widehat{u}_k,\widehat{u}_j \rangle_{\mathcal{H}} 
\beta_{j,n} = \langle \widehat{w}_k, v_n \rangle_{\mathcal{H}},
\quad k=1,\ldots,m, \, n=1,\dots,\k,
\]
and then setting
\begin{equation} \label{vapprox2}
v_{n}^{\mathrm{approx}}=\sum_{j=1}^{m} \beta_{j,n} \widehat{u}_j, \quad n=1,\ldots,\k.
\end{equation}

\subsection{Multiple eigenvalues inside the contour}
\label{s:multipleev}
In this section we show that multiple eigenvalues do not produce serious 
problems
with the contour method. In Theorem \ref{multieigen} we prove that the 
matrix $D$ from \eqref{eqfL}  
 inherits the multiplicity structure of the original
nonlinear problem. This is analogous to the behavior of the Evans function 
(see \cite{AGJ90}), see Section \ref{ss:defev} for more details. 
Let us first recall the definition of multiplicity of eigenvalues for  nonlinear pencils and the associated notion of chains of eigenvectors
(see \cite[Sec.I.1.6]{MM}, \cite{V76}).
\begin{definition} \label{CGE}
Let $F\in \bbH(\Omega,\cF(\cH,\cK))$ and $\lambda_0 \in \sigma(F)$.

(i)\,  A tuple $(v_0,\ldots,v_{n-1})\in \cH^n, n\ge 1$, is called a
chain of generalized eigenvectors (CGE) of $F$ at $\l_0$ if the polynomial
$v(\l)=\sum_{j=0}^{n-1} (\l-\l_0)^j v_j$ satisfies
\[ (Fv)^{(j)}(\l_0)=0, \quad j=0,\ldots,n-1.
\]
The order of the chain is the index $r_0$ $(\ge n)$ satisfying
\[ (Fv)^{(j)}(\l_0)=0, \quad j=0,\ldots,r_0-1, \quad
(Fv)^{(r_0)}(\l_0)\neq0 .
\]
The rank $r(v_0)$ of a vector $v_0\in \cN(F(\l)), v_0 \neq 0$,
is the maximum order of CGEs starting at $v_0$.

(ii)\, A canonical system of generalized eigenvectors (CSGE)
of $F$ at $\l_0$ is a  system of vectors 
\[ v_{j,p}\in \cH, \quad j=0,\ldots,\mu_p-1, \; p=1,\ldots,q, \; q\ge 1,
\]
with the following properties:
\begin{enumerate}
\item $v_{0,1},\ldots,v_{0,q}$ form a basis of $\cN(F(\l_0))$,
\item the tuple $(v_{0,p},\ldots,v_{\mu_p-1,p})$ is a CGE of $F$ at $\l_0$
for $p=1,\ldots,q$,
\item for $p=1,\ldots,q$ the indices $\mu_p$ satisfy \\
$\mu_p= \max\{r(v_0): v_0 \in \cN(F(\l_0)) \setminus
\mathrm{span}\{v_{0,\nu}:1\le \nu <p \}\}$. 
\end{enumerate}

(iii)\, The numbers $\mu_p$, $p=1,\dots,q$, are called the {\em partial
multiplicities}, where $\mu_1+\dots+\mu_q$ is the {\em
algebraic multiplicity} and $q=\dim \cN(F(\l_0))$  is the {\em geometric multiplicity}. The subspace 
$\mathrm{span}\{v_{j,p}:j=0,\ldots,\mu_p-1, p=1,\ldots,q \}$ is called the {\em root subspace} of the eigenvalue $\l_0$.
\end{definition}
Such a CSGE always exists and, as with usual Jordan chains for matrices, condition (ii)(3) guarantees that chains of 
highest order are taken first, so that
$\mu_1 \ge \mu_2 \ge \ldots \ge \mu_p$ holds.  

Next we state the general formula of Keldysh (\cite[Thm.1.6.5]{MM})
for a finite number of eigenvalues inside a given contour (cf. \cite[Cor.2.8]{B12} for this generalization).

\begin{theorem} \label{generalKeldysh}
Let $F\in \bbH(\Omega,\cF(\cH,\cK))$ and 
let $\Gamma\subset \rho(F)$ be a simple closed contour with bounded interior
$\Omega_0$  in $\Omega$.
Let $\Omega_0 \cap \s(F)= \{\l_1,\ldots,\l_{\varkappa}\}$ and consider for each
 $\l_n$, $n=1,\ldots,\varkappa$, a CSGE denoted by
\[ v_{j,p}^n \in \cH, \quad j=0,\ldots,\mu_{n,p}-1, \quad p=1,\ldots, q_n, \quad
n=1,\ldots,\varkappa. \]
Then there exist corresponding CSGEs
\[ w_{j,p}^n\in \cK', \quad j=0,\ldots,\mu_{n,p}-1, \quad p=1,\ldots, q_n, \quad
n=1,\ldots,\varkappa, \]
for $F^{\top}\in \bbH(\Omega,\cF(\cK',\cH'))$, an open set $\cU$ with 
$\Omega_0\cup\Gamma \subset \cU \subset \Omega$ and a function 
$H\in \bbH(\cU,\cF(\cK,\cH))$ such that
\begin{equation}
\label{Keldyshformula}
F(\l)^{-1}= 
\sum_{n=1}^{\varkappa} \sum_{p=1}^{ q_n} \sum_{j=1}^{\mu_{n,p}} (\l-\l_n)^{-j}
\sum_{\nu=0}^{\mu_{n,p}-j} v_{\nu,p}^{n} w_{\mu_{n,p}-j-\nu}^{n \top} + H(\l),
\, \l\in \cU\setminus \s(F).
\end{equation}
\end{theorem}
\begin{remark}
The dual CSGEs are uniquely determined by an orthogonality condition
which generalizes \eqref{simplevw}, see \cite[Thm.1.6.5]{MM}. We omit
these conditions here since they will not be used it in the sequel. \hfill$\Diamond$
\end{remark}
With Keldysh' formula \eqref{Keldyshformula} and Cauchy's formula  
we repeat the calculations that lead to \eqref{FB0},
\eqref{FB1}.
The result is
\begin{equation} \label{FB2}
D_0 = G_l G_r^{\top}, \quad D_1= G_l \Lambda G_r^{\top},
\end{equation}
where the matrices $G_l,G_r$ are of size $m\times\varkappa_0$
and $\ell \times \varkappa_0$, respectively with
\begin{equation} \label{kappazero}
\varkappa_0 = \sum_{n=1}^{\varkappa} \sum_{p=1}^{ q_n} \mu_{n,p}.
\end{equation}
More explicitly, we have
\begin{equation} \label{Glform}
(G_l)_{j,(n,p,\nu)} = \langle \widehat{w}_j , v_{\nu,p}^n \rangle,
 \begin{array}{ll}  j=1,\ldots,m, \\
 \nu=0,\ldots,\mu_{n,p}-1,\;  p=1,\ldots, q_n,
\; n=1,\ldots,\varkappa.
\end{array}
\end{equation}
\begin{equation} \label{Grform}
(G_r)_{j,(n,p,\nu)} = \langle w_{\mu_{n,p}-\nu-1} , \widehat{v}_j \rangle,
 \begin{array}{ll}  j=1,\ldots,\ell, \\
 \nu=0,\ldots,\mu_{n,p}-1,\;  p=1,\ldots, q_n,
\; n=1,\ldots,\varkappa,
\end{array}
\end{equation}
where one may think of the triples $(n,p,\nu)$ being ordered
lexicographically.
Moreover, it turns out that $\Lambda$ is of Jordan normal form
\begin{eqnarray} \label{jnf1}
\Lambda = \begin{pmatrix} J_1 & & \\
& \ddots & \\
& & J_{\varkappa} 
\end{pmatrix},
J_n =\begin{pmatrix} J_{n,1} & & \\
& \ddots & \\
& & J_{n, q_n}
\end{pmatrix}, \\ \label{jnf2}
J_{n,p} =\begin{pmatrix} \l_n & 1 & & \\
& \ddots & \ddots & \\
& & \l_n & 1 \\
& & & \l_n
\end{pmatrix}  \in \C^{\mu_{n,p}\times \mu_{n,p}}.
\end{eqnarray}
This immediately leads to the following generalization
of Theorem \ref{simpleeigen}.

\begin{theorem} \label{multieigen}
Let $F\in \bbH(\Omega,\cF(\cH,\cK))$ and 
let $\Gamma\subset \rho(F)$ be a simple closed contour with bounded interior
$\Omega_0$  in $\Omega$.
Let $\Omega_0 \cap \s(F)= \{\l_1,\ldots,\l_{\varkappa}\}$ and denote 
the CSGEs associated with $\l_n$ by
\[ v_{j,p}^n, \quad j=0,\ldots,\mu_{n,p}-1, \quad p=1,\ldots, q_n, \quad
n=1,\ldots,\varkappa. \]
Consider linearly independent elements
 $\widehat{w}_j\in \cH'$, $j=1,\dots,m$, and $\widehat{v}_k\in \cH$, 
$k=1,\dots,\ell$, such that $m,\ell\ge \varkappa_0$, see \eqref{kappazero}.
Determine the matrices $D_0,D_1\in \mathbb{C}^{m \times \ell}$  by 
\eqref{yeq},\eqref{defE},\eqref{defB01} and assume that the matrices
$G_l,G_r$ from \eqref{Glform}, \eqref{Grform} have maximum rank.
Then the $\varkappa_0\times \varkappa_0$-matrix 
$D=V_0^\ast D_1W_0\Sigma_0^{-1}$, where $V_0,\Sigma_0,W_0$ are
given by the singular value decomposition \eqref{singval} of $D_0$, has
Jordan normal form \eqref{jnf1}, \eqref{jnf2} which coincides with 
the multiplicity structure of the spectrum of the nonlinear operator
inside $\Gamma$.
\end{theorem}
\begin{remarks} 
\label{rem:jordanrobust}
(a) It is well known that the Jordan normal form is not robust to perturbations
and, therefore, not a suitable object for numerical computations.
Therefore, it seems questionable, whether one should compute approximate 
CSGEs by applying
formula \eqref{vapprox} to the columns of $G_l$ in \eqref{Glform}. 
Nevertheless, Theorem \ref{multieigen} has some significance.
In our case perturbations of the matrices 
$D_0,D_1,V_0,W_0,\Sigma_0$ may be caused
by approximate solutions of the operator equations \eqref{yeq},
by quadrature errors for the contour integrals \eqref{defB01}, or
by errors of the  singular value decomposition \eqref{singval}( see
Section \ref{s:approxop} for more details).
This leads to a perturbed matrix $D$ with  well known spectral properties,
e.g. the invariant subspaces of $D$ belonging to the eigenvalues
$\l_n,n=1,\ldots,\varkappa$, keep their dimension and perturb with the
same order, see \cite{StS90},
and the error of a single eigenvalue $\l_n$ grows at most with the power 
$1/\mu_{1,n}$ of the perturbation, where $\mu_{1,n}$ is the maximum rank
of eigenvectors belonging to $\l_n$ (cf. \cite{K80}).
\\
(b) There are several reasons that may cause a rank defect for 
the matrices $G_l,G_r$ in  \eqref{defGrl} and \eqref{Glform},\eqref{Grform},
respectively. First, it is possible that $F$ has  more eigenvalues
inside the contour than the dimension of the space $\cH$. For example,
this is typical for characteristic equations of ordinary
delay equations. However, this cannot occur with infinite dimensional
spaces $\cH,\cK$ which is our main concern here.
Second, there may be a vector $\sum_{n=1}^{\varkappa}\alpha_n v_n$ that is
annihilated by all test functionals $\widehat{w}_j$ or a functional 
$\sum_{j=1}^{m}\beta_j w_j$ that annihilates all test functions
$\widehat{v}_k$ (similarly for CSGEs). If the eigenfunctions $v_n$
and $w_n$ are linearly independent and the data $\widehat{w}_j,\widehat{v}_k$
are chosen at random we consider this case to be nongeneric.
However, it is possible for nonlinear eigenvalue problems that eigenvectors
belonging to different eigenvalues are linearly dependent 
(see \cite[Sec.4,5]{B12} for such an example). The contour method can be
extended to handle all these degenerate cases by evaluating higher
order moments 
\[\frac{1}{2 \pi i}\int_{\G}\l^{\nu} E(\l) d\l,\quad \nu=0,\ldots 2K-1.\]
It is shown in \cite[Lem.5.1]{B12} that it suffices to take
$K= \sum_{n=1}^{\varkappa} \mu_{n,1}$ under the conditions of Theorem \ref{multieigen}. \hfill$\Diamond$
\end{remarks}

 \subsection{Approximation of operators}
\label{s:approxop}
We consider a sequence of approximate operators 
$F_N\in \bbH(\Omega, \cF(\cH_N,\cK_N)), N\in \N$, and study the errors
for the  linear system \eqref{yeq} and 
the spectrum $\s(F)$ when $F$ is replaced by $F_N$.
 We will
use the framework of discrete approximations (see \cite{V76} and \cite{V80}
for an English reference) which has the advantage that $\cH_N,\cK_N$
are general Banach spaces (not necessarily subspaces of $\cH,\cK$) connected
to $\cH,\cK$ only via a set of linear operators (not necessarily projections)
\[ p_N : \cH \mapsto \cH_N, \quad q_N: \cK \mapsto \cK_N, \quad N\in \N.
\]
In Section \ref{s:converge} we will apply the theory to boundary value problems
on the infinite line when approximated by two-point boundary value problems
on a bounded interval.
In order to assist readers unfamiliar with the theory we impose conditions
slightly stronger than necessary, and  we try to avoid as many
notions as possible from \cite{V76}. Our assumptions are as follows:
\begin{enumerate}
\item[(D1)]
There exist Banach spaces $\cH_N,\cK_N, N \in \N$ and linear bounded mappings
$p_N\in \cL(\cH,\cH_N),q_N\in \cL(\cK,\cK_N)$ with the property
\[ \lim_{N\rightarrow \infty} \|p_N v \|_{\cH_N}= \|v\|_{\cH},\; v \in \cH,\quad
\lim_{N\rightarrow \infty} \|q_N v \|_{\cK_N}= \|v\|_{\cK}, v \in \cK.
\]
\item[(D2)] Given $F\in \bbH(\Omega,\cF(\cH,\cK))$ with  $\rho(F) \neq \emptyset$ and
$F_N\in \bbH(\Omega, \cF(\cH_N,\cK_N)), N\in \N$ with
$\sup_{N\in \N} \sup_{\l \in \cC} \| F_N(\l)\| < \infty$ for every compact
set $\cC \subset \Omega$.
\item[(D3)] $F_N(\l)$ converges regularly to $F(\l)$ for all $\l \in \Omega$
in the following sense:
\begin{enumerate}
\item $\lim_{N\rightarrow \infty}\| F_N(\l)p_Nv - q_N F(\l)v\| =0, \quad v \in \cH$,
\item for any subsequence  $v_N \in \cH_N, N\in \N'\subset \N$ with
 $\|v_N\|_{\cH_N}, N\in \N'$ bounded
and  $\lim_{\N' \ni N\rightarrow \infty}\|F_N(\l)v_N - q_N y \|_{\cK_N}=0$ for
some $y \in \cK$, there exists a subsequence $\N'' \subset \N'$ and
a $v\in \cH$ such that \\ $\lim_{\N'' \ni N \rightarrow \infty} \|v_N - p_N v\|_{\cH_N}
=0$.
\end{enumerate}
\end{enumerate}
The term $\| F_N(\l)p_Nv - q_N F(\l)v\|$ is called the {\it consistency error}
since it measures the defect in the noncommuting diagram.

\begin{equation}
\begin{CD} \cH@>F(\l)>> \cK\\
 @Vp_N VV @ VV q_N V\\
 \cH_N@>F_N(\lambda)>> \cK_N
\end{CD}\end{equation}
The following therorem summarizes convergence results from
\cite[\S 3(3), \S 4(33)]{V76}.
\begin{theorem} \label{discreteconverge}
Under the assumptions (D1)-(D3) the following assertions hold:

(i)\, For any compact set $\cC \subset \rho(F)$ and any 
$\widehat{v}\in \cK$
there exists an $N_0\in \N$ such that the linear equation 
$F_N(\l)y_N=q_N\widehat{v}$ has a unique solution $y_N=y_N(\l)$ for 
$N\ge N_0,\l\in \cC$, and the following estimate holds
\begin{equation} \label{estimatesolution}
\sup_{\l \in \cC} \|y_N(\l)- p_N y(\l) \|_{\cH_N} \le C 
\sup_{\l \in \cC}\| F_N(\l)p_N y(\l) - q_N F(\l) y(\l) \|,
\end{equation}
where $y(\l)\in \cH$ is the unique solution to $F(\l)y=\widehat{v}$. 

(ii)\, For any $\l_0 \in \s(F)$ there exists $N_0\in \N$ and a sequence
$\l_N \in  \s(F_N), N\ge N_0$, such that $\l_N \rightarrow \l_0$
as $N \rightarrow \infty$. For any sequence $\l_N\in \s(F_N)$ with this convergence property,
and associated eigenvectors $v_N^0\in \cN(F_N(\l_N)), \|v_N^0\|_{\cH_N}=1$,
the following estimates hold:
\begin{eqnarray} \label{estvalue}
| \l_N - \l_0| &\le C \varepsilon_N^{\frac{1}{r_0}} \\
\inf_{v_0\in \cN(F(\l_0))} \|v_N^0 - p_N v_0\|_{\cH_N}  &
 \le C \varepsilon_N^{\frac{1}{r_0}},
\end{eqnarray}
where $r_0=\mu_1$ is the maximum rank of eigenvectors that belong to $\l_0$.
The quantity $\varepsilon_N$ is a consistency error defined by
 \[ \varepsilon_N= \max_{|\l -\l_0|\le \delta}\, \max_{v\in \cM}
\| F_N(\l)p_N v - q_N F(\l)v \|, 
\]
where $\cM=\mathrm{span}\{v_{j,p}:j=0,\ldots,\mu_p-1,p=1,\ldots,q \}$
is the root space associated with $\l_0$ 
(see Definition \ref{CGE}) and $\delta>0$ is chosen sufficiently small.
\end{theorem}
\begin{remark} We note that our assumptions (D2) and (D3) imply regular convergence
in the sense of \cite[\S2(17)]{V76}. Moreover, it is easy to see that
the convergence result in \cite[\S3(3)]{V76} holds uniformly 
in $\l\in \cC$. Of course, for the latter result 
continuity of $\l \mapsto F(\l)$ is sufficient.  \hfill$\Diamond$
\end{remark}

\section{Relation to the Evans function}\lb{s:evfun}
In this section we relate the abstract approach from section \ref{s:abst}
to the study of zeroes of the Evans function. Recall from the abstract setting
formula \eqref{Ematrep}, which we rewrite as
\begin{equation} \label{Erepsing}
E(\l) = \frac{1}{\cE_z(\l)} H_1(\l) + H_0(\l), \quad \l \in \Omega \setminus
\{\l_1,\ldots,\l_{\varkappa}\},
\end{equation}
where we have introduced 
\[ \cE_z(\l) = \Pi_{j=1}^{\varkappa}(\l-\l_j), \quad
H_1(\l) = G_l \diag\left(\Pi_{\nu=1,\nu\neq j}^{\varkappa}(\l-\l_{\nu}),
j=1,\ldots,\varkappa\right) G_r^{\top}.\]
Note that  $\cE_z(\l)$ may be viewed as an abstract version of the Evans function
since its zeroes are exactly the eigenvalues of $F(\l)$. Moreover,
$H_0,H_1$ are holomorphic matrix functions such that $H_1$ degenerates to
a rank one matrix at every simple eigenvalue. 

In the following we make the above relation more explicit when the classical
definition of the  Evans function is used, that is as the determinant of a 
matrix with holomorphic columns determined from appropriate subspaces.
In the first step we set up a normalized Evans function that is 
independent of a perturbative situation. 

\subsection{Definition of a normalized Evans function}
\lb{ss:defev}
In general the Evans function may be considered as a determinant that
measures the coalescence of two families of subspaces depending
holomorphically on a parameter. We define a normalized
version of the function that is unique up to a sign.
To begin, we introduce the class of matrices 
\begin{equation} \label{defuni}
\bbM^{d,k} = \{ P \in \C^{d,k}: \det(P^\top P)=1 \}
\end{equation}
and recall the following elementary fact: Let $U\subset\C^d$ be a
subspace of dimension $k$. Then two matrices $P_1,P_2\in\C^{d,k}$ of
rank $k$ satisfy $\mathcal{R}(P_1)=\mathcal{R}(P_2)=U$  if and only if there is an invertible matrix $R\in\C^{k,k}$ such that $P_2=P_1R$. 
\begin{definition}\label{defEQ}
We call two matrices $P_1,P_2\in\C^{d,k}$ equivalent and write $P_1\thickapprox P_2$ provided $P_2=P_1R$ for some $R\in\C^{k,k}$
such that $\det(R)=-1$ or $\det(R)=1$. We will use notation $[P]=\{P_1\in\C^{d,k}: P_1\thickapprox P\}$ for the equivalence class.
\end{definition}
\begin{remark}\label{P12}
If $P_1,P_2\in\bbM^{d,k}$ and 
$\mathcal{R}(P_1)=\mathcal{R}(P_2)$ then $P_1\thickapprox P_2$ holds automatically since
$1=\det(P_2^\top P_2)=\det(P_1^\top P_1)\det(R)^2=\det(R)^2$.
\hfill$\Diamond$\end{remark}

\begin{lemma}\label{prLem}
For any rank $k$ projection $\Pi$ in $\C^d$ there exist $P,\Phi \in \C^{d,k}$ such that
\begin{equation} \label{normalrep1} \Pi = P\Phi^\top,
\quad \Phi^\top P=I_{k}.
\end{equation}
Moreover, if the inequality $\det(P^\top P)\neq0$ holds for some $P$ from \eqref{normalrep1}, then it holds for any such $P$, and there exist $P_0,\Phi_0 \in \C^{d,k}$ such that
\begin{equation} \label{normalrep2} \Pi = P_0\Phi_0^\top,
\quad P_0\in \bbM^{d,k},\quad \Phi_0^\top P_0=I_{k}.
\end{equation}
\end{lemma}
\begin{proof}
  Using basis vectors in $\mathcal{R}(\Pi)$ and
  $\mathcal{R}(\Pi^\top)$ as columns, we obtain rank $k$ matrices
  $P,\widetilde{\Phi}\in\C^{d,k}$
  with $\mathcal{R}(\Pi)=\mathcal{R}(P)$ and
  $\mathcal{R}(\Pi^\top)=\mathcal{R}(\widetilde{\Phi})$, respectively. 
  Standard linear algebra shows 
  $\mathcal{N}(\widetilde{\Phi}^\top)=
  \mathcal{R}(\widetilde{\Phi})^\bot= \mathcal{R}(\Pi^\top)^\bot=
  \mathcal{N}(\Pi)$, where $^\bot$ means orthogonal with respect to the
  duality pairing (see \eqref{fdual} in Section \ref{s:abst}). 

  First we obtain invertibility of $\widetilde{\Phi}^\top P\in
  \mathbb{C}^{k,k}$: For $c\in \mathbb{C}^k$ the assumption
  $\widetilde{\Phi}^\top P c = 0$ implies
  $Pc\in \mathcal{R}(P)\cap \mathcal{N}(\widetilde{\Phi}^\top)^\bot=
  \mathcal{R}(\Pi)\cap \mathcal{N}(\Pi)=\{0\}$ and hence $c=0$ since
  $P$ has full rank.
  Second, we claim that
  \begin{equation} \label{repp}
    \Pi = P (\widetilde{\Phi}^\top P)^{-1} \widetilde{\Phi}^\top, \quad
    P,\widetilde{\Phi} \in \C^{d,k}.
  \end{equation} 
  Indeed, the matrix $\widetilde{\Pi}=P
  (\widetilde{\Phi}^\top P)^{-1} \widetilde{\Phi}^\top$ is a projection
  with $\mathcal{R}(\widetilde{\Pi})=\mathcal{R}(P)=\mathcal{R}(\Pi)$
  and $\mathcal{N}(\widetilde{\Pi})=\mathcal{N}(\widetilde{\Phi}^\top)
  =\mathcal{N}(\Pi)$, yielding $\Pi=\widetilde{\Pi}$ and finishing the
  proof of \eqref{repp}. Letting 
\begin{equation}\label{normP1}
\Phi=\widetilde{\Phi}\big(P^\top\widetilde{\Phi}\big)^{-1},
\end{equation} we arrive at the representation \eqref{normalrep1}.

If $P,\widetilde{P}$ are two matrices as in \eqref{normalrep1}, then
$P=\widetilde{P}R$ for some invertible $R\in\C^{k,k}$ due to
$\mathcal{R}(P)=\mathcal{R}(\widetilde{P})=\mathcal{R}(\Pi)$. Now $\det(P^\top P)=\det(\widetilde{P}^\top \widetilde{P})\det(R)^2$ proves the second assertion in the lemma. Finally, if
 $\det(P^\top P)\neq 0$ then we can replace $P$ and $\widetilde{\Phi}$ in \eqref{repp} by
\begin{equation} \label{normP2}
 P_0=P \diag(\det(P^\top P)^{-1/2},1,\ldots,1)\in \bbM^{d,k}
 \,\text{ and } \Phi_0=\widetilde{\Phi} (P_0^\top\widetilde{\Phi})^{-1}.
\end{equation}
Since $P_0\in\bbM^{d,k}$, we arrive at \eqref{normalrep2}.
\hfill\end{proof}

Next consider two subspaces of complementary dimension. We may write
them as images of projections that both have the representation \eqref{normalrep1}, or, perhaps, even the normalized representation
\eqref{normalrep2}.
\begin{definition} \lb{defsub}
Let $U,V \subset \C^d$ be subspaces of dimension $k$ and $d-k$,
respectively. Pick any $P\in\C^{d,k}$ of rank $k$ such that
$\mathcal{R}(P)=U$ and $Q\in\C^{d,d-k}$ of rank $d-k$ such that
$\mathcal{R}(Q)=V$.
Then  we call
\begin{equation}\label{detset1}
\cD([P],[Q])= \{\det(P_1|Q_1): P_1\in[P], Q_1\in[Q]\}
\end{equation}
the determinant set of the equivalent classes $([P],[Q])$ from Definition \ref{defEQ}.
In addition,  assume that $P \in \bbM^{d,k},Q\in \bbM^{d,d-k}$. Then we call
\begin{equation} \lb{detset2}
\cD(U,V) = \{ \det(P|Q) : P \in \bbM^{d,k},Q\in \bbM^{d,d-k},
\mathcal{R}(P)=U, \mathcal{R}(Q)=V \}
\end{equation}
the determinant set of the subspaces $U$ and $V$.
\end{definition}
\begin{remark}\lb{remUNIQ}
We stress that the set $\cD([P],[Q])$ is defined with no additional assumptions on the subspaces $U,V$. Although this set does not depend on the choice of the representatives $P,Q$ in the equivalence classes $[P],[Q]$, it does depend on the choice of the equivalence classes. On the other hand,
the set $\cD(U,V)$ is defined under an additional assumption on the subspaces $U,V$, but is uniquely determined by these subspaces.
\hfill$\Diamond$\end{remark} 

Our key observation is the following lemma which shows that the class
\eqref{defuni} and the definitions \eqref{detset1}, \eqref{detset2} have some significance.
\begin{lemma} \lb{detchar}
Let $U,V \subset \C^d$ be subspaces of dimension $k$ and $d-k$,
respectively. Pick any $P\in\C^{d,k}$ of rank $k$ such that
$\mathcal{R}(P)=U$ and $Q\in\C^{d,d-k}$ of rank $d-k$ such that
$\mathcal{R}(Q)=V$.
Then there exists $z \in \C$ such that $\cD([P],[Q])=\{z,-z\}$. The subspaces 
$U$ and $V$ are complementary if and only if $z \neq 0$. In addition, assume that there are $P_0,Q_0$ picked as above and such that $\det(P_0^\top P_0)\det(Q_0^\top Q_0)\neq0$. Then there exists $z \in \C$ such that $\cD(U,V)=\{z,-z\}$.
\end{lemma}
\begin{proof} First take  $P_1 \in [P]$, $Q_1\in [Q]$
  with $\mathcal{R}(P_1)=U, \mathcal{R}(Q_1)=V$  and define  $z=\det(P_1|Q_1)$. Then take another pair
$P_2,Q_2$ of this type and use Definition \ref{defEQ} to write  $P_2=P_1R$,
$Q_2=Q_1S$ for some matrices $R\in \C^{k,k}$, $S \in \C^{d-k,d-k}$ such that $\det(R),\det(S)\in\{-1,1\}$.
Therefore we conclude
\begin{equation*} 
\det(P_2|Q_2)= \det\left( (P_1|Q_1) 
\begin{pmatrix} R & 0 \\ 0 & S \end{pmatrix} \right)= z \det(R)\det(S).
\end{equation*} This proves $\cD([P],[Q])\subset \{z,-z\}$. In fact we have equality
since $-z$ is attained by taking $R=\diag(-1,1,\ldots,1), S=I_{d-k}$.
The second assertion of the lemma is obvious.
 By normalizing $P_0,Q_0$ as in \eqref{normP2}, we may assume that $P_0 \in \bbM^{d,k}$,
$Q_0\in \bbM^{d,d-k}$, and the third assertion holds by Remark \ref{P12}.
\hfill\end{proof}
\begin{definition} \label{holsub}
A family of subspaces $U(\l)\subset \C^d$, $\l\in \Omega \subset \C$, 
is called {\it holomorphic in $\Omega$} if there exists a holomorphic
family of projections $\Pi \in \bbH(\Omega,\C^{d,d})$ such that
$\mathcal{R}(\Pi(\l))=U(\l)$ for all $\l\in \Omega$.
\end{definition}
Clearly, by the connectedness of $\Omega$ the projections must be of 
constant rank and hence the subspaces are of constant dimension.

The following lemma generalizes the representation \eqref{normalrep1} and the normalized representation \eqref{normalrep2}
to holomorphic families. It is essential that the domain is simply connected.
\begin{lemma} \label{holpro} 
Let $\Omega$ be a simply connected domain
in $\C$. For any holomorphic family of projections $\Pi\in \bbH(\Omega, \C^{d,d})$ of rank $k$
there exist functions $P,\Phi \in \bbH(\Omega,\C^{d,k})$
such that 
\begin{equation} \lb{holrep1}
\Pi(\l)= P(\l)\Phi(\l)^\top, \, \Phi(\l)^\top P(\l) = I_k, 
\quad \text{for all} \; \l\in \Omega.
\end{equation}
Moreover, for any $P \in \bbH(\Omega,\C^{d,k})$ such that
$\mathcal{R}(P(\l))=\mathcal{R}(\Pi(\l))$ for all $\l\in\Omega$, the set
\begin{equation}\label{defLP}
\Lambda=\{\l\in\Omega: \det(P(\l)^\top P(\l))=0\}
\end{equation}
is discrete, may be empty, and is independent of the choice of $P$.  Finally, there exist functions $P_0,\Phi_0 \in \bbH(\Omega\setminus\Lambda, \C^{d,k})$
such that 
\begin{equation} \lb{holrep2}
\Pi(\l)= P_0(\l)\Phi_0(\l)^\top,  \Phi_0(\l)^\top P_0(\l) = I_k, P_0(\l)\in \bbM^{d,k} \text{ for all } \l\in \Omega\setminus\Lambda.
\end{equation} 
\end{lemma}
\begin{proof}
By a result from \cite[Sec.II.4.2]{K80} there exist functions
$P, \widetilde{\Phi} \in \bbH(\Omega,\C^{d,k})$ such that
$\mathcal{R}(\Pi(\l))=\mathcal{R}(P(\l))$,
$\mathcal{R}(\Pi(\l)^\top)=\mathcal{R}(\widetilde{\Phi}(\l))$ for all $\l\in \Omega$. This step uses simple
connectedness. In the next step, as in the proof of Lemma \ref{prLem}, 
we normalize $\widetilde{\Phi}(\l)$ as in  \eqref{normP1}, and keep holomorphy. This proves \eqref{holrep1}.
The set of zeros of the holomorphic function $\det(P(\cdot)^\top P(\cdot))$
is discrete. If $P_1, P_2 \in \bbH(\Omega,\C^{d,k})$ are such that
$\mathcal{R}(P_{1}(\l))=\mathcal{R}(P_{2}(\l))=\mathcal{R}(\Pi(\l))$ for all $\l\in\Omega$ then there is a nonsingular 
$R(\l)\in\C^{k,k}$ such that $R_2(\l)=P_1(\l)R(\l)$; hence, $\det(P_1(\cdot)^\top P_1(\cdot))$ and $\det(P_2(\cdot)^\top P_2(\cdot))$ have the same zeros.
If $\l\in\Omega\setminus\Lambda$ then $\det(P(\l)^\top P(\l))\neq0$ and
we normalize $P(\l)$ and $\Phi(\l)$
as in \eqref{normP2}, and keep holomorphy. 
Note that the square root has (up to a sign) a
unique analytic continuation in any simply connected domain that does not
contain $0$.
\hfill\end{proof}
 
The following theorem shows how one can assign to two families of
holomorphic subspaces the {\em usual} 
Evans function which is equal to zero if and only if the subspaces are not complementary. This function is
holomorphic on all of $\Omega$ see \eqref{condevans1}.  It depends, up to a sign, on the equivalence classes from Definition \ref{defEQ} for the matrices formed by the basis vectors of the subspaces.
In addition, one can assign to the two families of subspaces an Evans function that is unique, up to a sign, but is holomorphic
on a smaller set $\Omega\setminus\Lambda(U,V)$, see \eqref{condevans2}.  We will call it the 
{\em normalized} Evans function. We refer to Remark \ref{remUNIQ} regarding the sets $\cD([P(\l)],[Q(\l)])$ and $\cD(U(\l),V(\l))$ used in \eqref{condevans1}, \eqref{condevans2} below.
\begin{theorem} \lb{evexist}
Let $\Omega \in \C$ be a simply connected domain and let 
$U(\l),V(\l),\l\in \Omega$, be two holomorphic families of subspaces of
dimensions $k$ and $d-k$, respectively. Pick any $P\in\bbH(\Omega,
\C^{d,k})$ such that $\mathcal{R}(P(\l))=U(\l)$, $\rank (P(\l))=k$, and
$Q\in\bbH(\Omega, \C^{d,d-k})$ such that $\mathcal{R}(Q(\l))=V(\l)$, $\rank (Q(\l))=d-k$ for all $\l\in\Omega$. Then the following assertions hold.

$(i)$\, There exists a holomorphic
function $\cE:\Omega \rightarrow \C$ such that
\begin{equation} \lb{condevans1}
\cE(\l) \in \cD([P(\l)],[Q(\l)]) \quad \text{for all} \quad \l \in \Omega.
\end{equation}

$(ii)$\, $\cE(\lambda_0)=0$ if and only if $U(\lambda_0)\cap V(\l_0)\neq\{0\}$.

$(iii)$\, If $\cE_1 \in \bbH(\Omega,\C)$ is any function satisfying \eqref{condevans1}
then either $\cE_1=\cE$ or $\cE_1=-\cE$.

$(iv)$\, Moreover,  the set
\begin{equation}\label{defLPl}
\Lambda(U,V)=\{\l\in\Omega: \det(P(\l)^\top P(\l))\det(Q(\l)^\top Q(\l))=0\}
\end{equation}
is discrete, may be empty, and is independent of the choice of $P,Q$ picked as  above.

$(v)$\, Finally, there exists a holomorphic
function $\cE_0:\Omega\setminus\Lambda(U,V) \rightarrow \C$ such that
\begin{equation} \lb{condevans2}
\cE_0(\l) \in \cD(U(\l),V(\l)) \quad \text{for all} \quad \l \in \Omega\setminus\Lambda(U,V),
\end{equation}
and properties $(i), (ii)$ hold for $\cE_0$.
\end{theorem}
\begin{proof} By Lemma \ref{holpro} we have a  representation
for  the projections $\Pi_U,\Pi_V \in \bbH(\Omega,\C^{d,d})$ associated 
with $U(\l)$, $V(\l)$ as follows
\begin{equation} \lb{reploc}
\Pi_U(\l)=P(\l)\Phi(\l)^\top, \quad \Pi_V(\l)=Q(\l)\Psi(\l)^\top, \quad \l \in \Omega,
\end{equation}
where  $P,\Phi\in \bbH(\Omega,\C^{d,k})$, $Q,\Psi \in \bbH(\Omega,\C^{d,d-k})$ and
for all $\l \in \Omega$
\begin{equation} \lb{normloc1}
\begin{aligned} \Phi^\top(\l)P(\l)= I_k,  \quad
   \Psi^\top(\l)Q(\l)= I_{d-k}.
\end{aligned}
\end{equation}
Defining 
\begin{equation} \label{defevans1}
\cE(\l) = \det(P(\l)|Q(\l)), \quad\l \in \Omega,
\end{equation}
 and using
Lemma \ref{detchar} proves assertions $(i)$ and  $(ii)$.
Now let $\cE_1$ be any function with the properties of $\cE$.
 Then the quotient
${\cE_1}/{\cE}$ is  holomorphic on  
$\Omega \setminus \cN(U,V)$ where the zero set is given by 
\begin{equation} \label{zeroset}
 \cN(U,V) = \{ \l \in \Omega: \cD([P(\l)],[Q(\l)])= \{0\} \}=
\{\l\in \Omega: \cE(\l)=0  \}.
\end{equation}
Since the quotient assumes only values $\pm1$ there and $\cN(U,V)$ contains
only isolated points it is a constant yielding $(iii)$. Assertion $(iv)$ follows from Lemma \ref{holpro}, see \eqref{holrep1}. Also, by Lemma \ref{holpro}, see \eqref{holrep2}, we can choose
functions $P_0,\Phi_0\in \bbH(\Omega\setminus\Lambda(U,V),\C^{d,k})$ and  $Q_0,\Psi_0 \in \bbH(\Omega\setminus\Lambda(U,V),\C^{d,d-k})$ such that, in addition to \eqref{reploc}, \eqref{normloc1}, the following normalization holds:
\begin{equation} \lb{normloc2}
 P_0(\l)\in \bbM^{d,k}, \quad
 Q_0(\l) \in \bbM^{d,d-k}, \quad \l\in\Omega\setminus\Lambda(U,V).
\end{equation}
Defining 
\begin{equation} \label{defevans2}
\cE_0(\l) = \det(P_0(\l)|Q_0(\l)), \quad\l \in \Omega\setminus\Lambda(U,V),
\end{equation}
and using Lemma \ref{detchar} proves assertion $(v)$.
\hfill\end{proof}

\begin{remark}
1. An alternative proof of the theorem can be obtained  directly 
from Lemma \ref{detchar} without using Lemma \ref{holpro} and \cite[Sec.II.4.2]{K80}.
One first defines $\cE$, respectively, $\cE_0$ locally via \eqref{defevans1}, respectively, \eqref{defevans2} using a respective local
normalization that is always possible (without simple connectedness).
Then one shows that one can continue holomorphically along any curve in 
$\Omega$, respectively, $\Omega\setminus\Lambda(U,V)$ (since the set $\cN(U,V)$ defined in \eqref{zeroset} is isolated) and the sign of the continuation is determined from
 Lemma \ref{detchar}. Since the domain is simply connected, the result follows from the Monodromy Theorem, see, e.g., \cite[Cor.IX.3.9]{C78}.\\
2. According to Theorem \ref{evexist} the normalized Evans function is unique
up to a sign on simply connected domains. The assumption of simple connectedness
is crucial for the proof. 
The main ingredient which makes the Evans function unique up to a sign
was to allow only matrices with $\det(P^\top P) =1$ in Definition \ref{defsub}.
When homotopy arguments are applied to the Evans functions between different
points in $\Omega$, e.g. between $0$ and $\infty$, it remains to be checked
whether the values are from the same leaf that belongs to one of the signs. \\
3. Normalization \eqref{normP2} shows that the normalized Evans function $\cE_0$ can not be continued to $\Lambda(U,V)$ since $\cE_0(\l)\sim(\l-\l_0)^{-m/2}$ as $\l\to\l_0$ at any point $\l_0\in\Lambda(U,V)$ which is a zero of order $m$ of the function $\det(P(\cdot)^\top P(\cdot))\det(Q(\cdot)^\top Q(\cdot))$ holomorphic in $\Omega$.
\hfill$\Diamond$
\end{remark}

In the next step we project vectors in $\C^d$  onto their components in the 
subspaces $U(\l)$ and $V(\l)$. Although this cannot work at the zeroes of
the Evans function we insist that the main part of the projection stays
holomorphic at the singularity.
Recall the adjugate $\adj(A)$ of a square matrix $A \in \C^{d,d}$ given by
\begin{equation} \lb{adjdef}
 (\adj(A))_{ij}  = (-1)^{i+j} \det\left(A_{\ell,m}\right)_{\ell=1,\ldots,j-1,j+1,\ldots,d}^{m=1,\ldots,i-1,i+1,\ldots,d},
\end{equation} 
(see e.g. \cite[Sec.4.4]{S88}). It  satisfies the identities $A \adj(A)= \adj(A) A = \det(A) I_d$ and 
$\det(\adj(A))= \det(A)^{d-1}$. More importantly, by definition the adjugate
 preserves holomorphy, i.e.
 if $A$ is in $ \bbH(\Omega,\C^{d,d})$ then so is $\adj(A)$.
\begin{theorem} \label{projecthol}
Let the assumptions of Theorem \ref{evexist} hold and let 
$\cE\in \bbH(\Omega,\C)$, respectively, $\cE_0\in \bbH(\Omega\setminus\Lambda(U,V),\C)$ be one of the two Evans functions, respectively, normalized Evans functions,  determined there.
Then there exist  matrix valued functions $\cY_U, \cY_V \in \bbH(\Omega,\C^{d,d})$
with the following properties:
\begin{equation} \lb{addproj}
\cE(\l) I_d = \cY_U(\l)+\cY_V(\l) \quad \text{for all}\quad  \l \in \Omega ,
\end{equation}

\begin{equation} \lb{UVprop}
\begin{array}{cc}
  \mathcal{R}(\cY_U(\l))  =  U(\l), & \mathcal{N}(\cY_U(\l))= V(\l),  \\
  \mathcal{R}(\cY_V(\l))  = V(\l), & \mathcal{N}(\cY_V(\l))= U(\l) 
\end{array} 
\quad \text{if} \quad \cE(\l) \neq 0,
\end{equation}
\begin{equation} \lb{UVsingular}
\begin{array}{cc}
  \mathcal{R}(\cY_U(\l)) \subset U(\l), & \mathcal{N}(\cY_U(\l)) \supset V(\l),  \\
  \mathcal{R}(\cY_V(\l))  \subset V(\l), & \quad \mathcal{N}(\cY_V(\l))\supset U(\l) 
\end{array} \quad \text{if} \quad \cE(\l)=0. 
\end{equation}
Conversely, the functions $\cY_U, \cY_V \in \bbH(\Omega,\C^{d,d})$ are uniquely
determined by properties \eqref{addproj} and \eqref{UVprop}. 
Similarly, for the normalized Evans function $\cE_0$ there exist
$\cY_U^0(\l), \cY_V^0(\l)\in
\bbH(\Omega\setminus\Lambda(U,V),\C^{d,d})$ with the same properties.
\end{theorem}
\begin{remark} Note that $\cY_U$ and $\cY_V$ behave almost like projections since
\eqref{addproj}, \eqref{UVsingular} imply 
\begin{equation} \lb{projlike}
\cY_U(\l)\cY_U(\l)= \cE(\l)\cY_U(\l), \quad
\cY_V(\l)\cY_V(\l)= \cE(\l)\cY_V(\l), \quad \l \in \Omega.
\end{equation}
 \hfill$\Diamond$
\end{remark}
\begin{proof}For the projections $\Pi_U, \Pi_V \in\bbH(\Omega,\C^{d,d})$ associated with the subspaces $U(\l),V(\l)$, we take the representation \eqref{holrep1} from Lemma \ref{holpro} and partition
the adjugate as follows
\begin{equation} \lb{partadj}
\adj\left(P(\l)|Q(\l)\right) = 
\begin{pmatrix} R(\l)^\top \\ S(\l)^\top \end{pmatrix}, \quad \l \in \Omega,
\end{equation} 
where $R\in {\mathbb H}(\Omega,\C^{d,k})$, $S \in \bbH(\Omega,\C^{d,d-k})$.
With these settings we define
\begin{equation} \lb{defquasiproj}
\cY_U(\l)= P(\l) R(\l)^\top, \quad \cY_V(\l)= Q(\l)S(\l)^\top \quad \text{for}
\quad \l \in \Omega.
\end{equation}
Conditions \eqref{addproj}-\eqref{UVsingular} hold for $\l\in \Omega$ as can
be seen from the identities 
\begin{equation} \label{adad}
\begin{pmatrix} P(\l) & Q(\l) \end{pmatrix}
            \begin{pmatrix} R(\l)^\top \\ S(\l)^\top \end{pmatrix} =
\cE(\l) I_d = \begin{pmatrix} R(\l)^\top \\ S(\l)^\top \end{pmatrix}
            \begin{pmatrix} P(\l) & Q(\l) \end{pmatrix}.
\end{equation}
Note that the first equality implies \eqref{addproj} while the second implies
\eqref{UVsingular}. In case $\cE(\l)\neq 0$ the matrices $R(\l),S(\l)$
are of full rank so that \eqref{UVprop} follows.

Now consider functions  $\tilde{\cY}_U, \tilde{\cY}_V \in \bbH(\Omega,\C^{d,d})$ 
that satisfy
\eqref{addproj},\eqref{UVprop} in $\Omega$. If $\cE(\l)\neq0$ then
 from the general form \eqref{repp}  and \eqref{UVprop} we find representations 
$\tilde{\cY}_U(\l)=P(\l)\tilde{R}(\l)^\top$,  $\tilde{\cY}_V(\l)=Q(\l)\tilde{S}(\l)^\top$ for some $\tilde{R}(\l)\in \C^{d,k},\tilde{S}(\l)\in \C^{d,d-k}$.
Invoking  \eqref{addproj} leads to
\[
\begin{pmatrix}P(\l) &Q(\l) \end{pmatrix} 
\begin{pmatrix} \tilde{R}(\l)^\top \\ \tilde{S}(\l)^\top \end{pmatrix} 
= \cE(\l) I_d,
\]
and hence $\tilde{R}(\l),\tilde{S}(\l)$ and $R(\l),S(\l)$  must agree due to \eqref{adad}.
Thus $\tilde{\cY}_U=\cY_U$, $\tilde{\cY}_V=\cY_V$ holds in 
$\Omega \setminus \cN(U,V)$ and hence in all of $\Omega$ due to  holomorphy. 

The last assertion in the theorem follows as above by taking the representation \eqref{holrep2} instead of \eqref{holrep1}. \hfill\end{proof}

Next we consider the behavior of the matrices $\cY_U$ and $\cY_V$ near zeroes of
the Evans function $\cE$. Let us first note, that the multiplicity
of a value $\l_0$ as a zero of the  Evans function
 corresponds exactly to the algebraic multiplicity of the eigenvalue $\l_0$
 of the matrix pencil $\cY(\l)=(P(\l)\big| Q(\l))$ from \eqref{defevans1} defined in terms of root functions (see \cite[Prop.1.8.5]{MM}).
For simplicity we consider only the behavior near simple eigenvalues and
show that the matrices $\cY_U,\cY_V$ degenerate to rank one matrices.

\begin{theorem} \label{zerobehavior}
Let the assumptions of Theorem \ref{projecthol} hold and let $\lambda_0\in \Omega$ be a simple zero of the Evans function $\cE(\l)$ defined as in 
\eqref{defevans1}. Then there are
nontrivial vectors $v_0,w_0$ such that 
\[\mathcal{N}(P(\l_0)|Q(\l_0))=\mathrm{span} \{v_0\}, \quad  \mathcal{N}((P(\l_0)|Q(\l_0))^\top)=
 \mathrm{span} \{w_0\}.
\]  
Morover, with $v_0=\begin{pmatrix}v_{0,1} \\v_{0,2} \end{pmatrix}$ we have the formulas 
\begin{equation} \label{batsing}
\cY_U(\l_0)=\cE'(\l_0)P(\l_0)v_{0,1} w_0^\top=-\cE'(\l_0)Q(\l_0)v_{0,2} w_0^\top
=-\cY_V(\l_0)
\end{equation}
and, if $v_0^\top v_0=1$,
\begin{equation} \label{evansderive}
\cE'(\l_0)= v_0^\top \adj(P(\l_0)|Q(\l_0))(P'(\l_0)|Q'(\l_0)) v_0.
\end{equation}
Finally, if $\lambda\in\Omega\setminus\Lambda(U,V)$ then the Evans function $\cE$ in assertions above can be replaced by the normalized Evans function $\cE_0$. 
\end{theorem}
\begin{proof} From \cite[Prop.1.8.5]{MM} we have that the kernel of
$(P(\l_0)|Q(\l_0))$ and of its transpose are one-dimensional. Applying Keldysh's
Theorem to the matrix pencil $\cY(\l)=(P(\l)|Q(\l))$ shows that for some 
$h_1\in \bbH(\Omega_1,\C^{d,d})$, $\Omega_1$ is some neighborhood of $\l_0$,
\begin{equation} \label{expand}
 \cY(\l)^{-1} = \frac{1}{\l-\l_0} v_0 w_0^\top + h_1(\l), \quad \l\in \Omega_1.
\end{equation}
Here we have normalized $w_0$ such that
 \[1=w_0^\top \cY'(\l_0)v_0 =w_0^{\top}P'(\l_0) v_{0,1}+w_0^{\top}Q'(\l_0) v_{0,2}.\]
Comparing the singular parts of \eqref{expand} and
\[\cY(\l)^{-1}=\frac{1}{\cE(\l)}\begin{pmatrix} R(\l)^{\top} \\
                                 S(\l)^{\top} \end{pmatrix}
\]
leads to 
\begin{equation} \label{reponeovere}
 \frac{1}{\cE'_0(\l_0)}\begin{pmatrix} R(\l_0)^{\top} \\
                                 S(\l_0)^{\top} \end{pmatrix}
= v_0w_0^{\top} = \begin{pmatrix} v_{0,1}w_0^{\top} \\ v_{0,2}w_0^{\top} 
                  \end{pmatrix}.
\end{equation}
With this and $0=\cY(\l_0)v_0=P(\l_0)v_{0,1} +Q(\l_0)v_{0,2}$ we  arrive at
\begin{align*}
 \cY_U(\l_0)&=P(\l_0)R(\l_0)^{\top}=\cE'_0(\l_0)P(\l_0)v_{0,1}w_0^{\top}\\
&=-\cE'_0(\l_0)Q(\l_0)v_{0,2}w_0^{\top} = -\cY_V(\l_0).
\end{align*}
Finally, normalizing $v_0^{\top}v_0=1$ we obtain \eqref{evansderive} from
\eqref{reponeovere} and the normalizing condition. Note that 
\eqref{evansderive} may also be derived directly by differentiating
$\adj(\cY(\l))\cY(\l)=\cE_0(\l) I_{k}$ at $\l=\l_0$ and then multiplying by 
$v_0^{\top}$ from the left and by $v_0$ from the right.
\hfill\end{proof}

\subsection{Application to general first order differential operators}
\lb{sec3.2}
In this section we apply the previous results to general first order
differential operators with matrices that depend holomorphically
 on the eigenvalue parameter.
We consider a first order $(d\times d)$ matrix differential equation,
\beq\lb{dfens} y'=A(\l,x)y, \quad x\in\bbR, \quad \l\in\Omega,\enq
and the respective pencil of first order differential operators
\begin{align}\label{1stOrderOP}
  F(\l)y=-y'+A(\l,x)y, \quad x\in\bbR,\;\l\in\Omega,
\end{align}
and impose the following assumptions.
\begin{hypothesis}\label{hyp1}
(i)\, Assume that the mapping
\begin{equation}\label{AminBhol}
\Omega\ni\l \mapsto A(\l,\cdot) -B(\cdot)\in L^{\infty}(\R,\C^{d,d}) \quad \text{is holomorphic}
\end{equation}
for some  matrix valued function $B(\cdot)$ such that either $B(\cdot)\in L^1(\R,\C^{d,d})$ or $B(\cdot)$ is
bounded continuous with $\lim_{x\rightarrow \pm \infty}B(x)=0$. 

(ii)\, Assume that the differential  equation \eqref{dfens} has for all
$\lambda\in\Omega$ exponential 
dichotomy on $\bbR_+$ with projections
$P_+(\l,x)$, $x\ge 0$, of rank $k$ and on $\R_-$ with projections
$P_-(\l,x)$, $x\le 0$, of the same rank $k$, i.e. for every
$\lambda\in\Omega$ there exist $C,\alpha>0$ so that
the usual dichotomy estimates in \eqref{dichotomy} below hold.
 
 (iii)\, Assume that  \eqref{dfens} has exponential dichotomy on $\R$
 for some $\l\in\Omega$.
\end{hypothesis}

 Hypothesis \ref{hyp1}(i) yields that
the differential operators $F(\lambda)$ 
are defined on the $\l$-independent domain, cf.\ \cite[Lem.2.8]{DL}, \cite[Ch.3,4]{CL}, given by
\begin{equation} \label{domainF}
\cH=\big\{y\in L^2(\bbR,\bbC^d) : y\in AC_\loc(\bbR,\bbC^d),\; -y'+B(\cdot)y\in L^2(\bbR,\bbC^d)\big\}.
\end{equation}
Moreover, $F(\l)$ is a bounded operator from $\cH$ into the space
\begin{equation} \label{rangeF}
\cK=L^2(\R,\C^d),
\end{equation}
when $\cH$ is equipped with the graph norm
\[ \|y\|_{\cH}^2 = \|y\|_{L^2}^2 +\|-y'+B(\cdot) y \|_{L^2}^2.
\]
  If $B(\cdot)$ is bounded, then $\cH=H^1(\R,\C^d)$, the Sobolev space. The 
completeness of $\cH$ follows from Lemma \ref{lem:A1} saying that $\cH$ is
embedded in the space of continuous  functions vanishing at $\pm\infty$ 
,see also \cite[Lem.2.8]{DL}.

 We note in passing that the differential operator $F(\lambda)$, considered 
as an unbounded operator in $L^2(\R,\C^d)$ with the domain $\cH$, generates a strongly continuous semigroup $\{T^t\}_{t\ge0}$, called the evolution semigroup, see \cite{CL}, defined by the formula $(T^ty)(x)=S(x,x-t,\l)y(x-t)$, $x\in\R$, $t\ge0$. Here and below we denote by $S(x,\xi,\l)$, $x,\xi\in\R$, the propagator (solution operator) of the differential equation 
\eqref{dfens}.

\begin{remarks}\label{rem:afterh1}
(a) Hypothesis \ref{hyp1} (ii) holds if and only if
 the operator $F(\l)$ is Fredholm and its 
Fredholm index is equal to zero for all $\l\in\Omega$. This is Palmer's
Theorem, see \cite{P84,P88}, and also \cite{BAG}, \cite[Thm.3.2]{Sa02},
\cite[Thm.2.6]{SaS01} and \cite{LT,LP,SaS08} for its discussions and
generalizations. In particular, the operator $F(\l)$ is invertible if and only if \eqref{dfens} has an exponential dichotomy on the entire line $\R$.

(b)\, Hypothesis \ref{hyp1} (iii) implies that $\rho(F)\neq\emptyset$, 
cf.\ \cite{BL,SaS08}, \cite[Thm.3.17]{CL}, and thus Hypothesis \ref{hyp1} yields Hypothesis \ref{hyp0} for pencil \eqref{1stOrderOP}.

(c)\ Since the coefficient of the differential equation \eqref{dfens} is holomorphic, the dichotomy projections  $P_\pm(\cdot, x)$, $x\in\R_\pm$, are holomorphic, see, e.g., \cite{BL}, \cite[Lem.A1]{DL}, \cite[Thm.1]{SaS00} and further references therein. 
 \hfill$\Diamond$
\end{remarks}

\begin{example} (Perturbations) \, A typical case where Hypotheses \ref{hyp1} are met
occurs with operators of perturbation form 
\[ [F(\l)y](x)= -y'(x) + (A_0(\l,x)+B(x))y(x), \quad x \in \R,
\]
if the unperturbed operator $F_0(\l)y=-y'+A_0(\l,x)y$
has an exponential dichotomy on $\bbR$ for all $\l\in\Omega$ and with projections
of rank $k$. Then the assumption $B\in L^1(\R,\C^{d,d})$ guarantees that the
exponential dichotomies on half lines hold for the perturbed operator (see, e.g., \cite[Prop.4.1]{Cop}, \cite[Lem.2.13]{GLM07}).
This case applies to the Schr\"odinger equation in Section \ref{s:appl} below. \hfill$\Diamond$
\end{example}
A more specific situation occurs when one linearizes a one-dimensional PDE about a traveling front and re-writes the respective eigenvalue problem as the first order ODE system \eqref{dfens}. Then its coefficient  can be assumed to stabilize at $\pm \infty$.
\begin{example}\lb{front} (Traveling Fronts)\, Consider the case of
  piecewise constant matrices 
\begin{equation}\label{APC}
A_\text{pc}(\l,x)=A_+(\l)\; \text{ for } \; x\ge0\; \text{ and } \;
A_\text{pc}(\l,x)=A_-(\l)\;\text{ for }\; x<0,\end{equation}
where the matrices $A_\pm(\lambda)$ satisfy the following properties: 

 (a) $A_\pm(\lambda)$ analytically depend on $\l\in\Omega$, 

 (b)  $A_\pm(\lambda)$ have no purely imaginary eigenvalues, 

 (c)  $\rank P_{A_+(\l)}=\rank P_{A_-(\l)}$ for the Riesz projections $P_{A_\pm(\l)}$ corresponding to 
 
 \quad\; the part of the spectrum of $A_\pm(\l)$ located in the left half plane.

Under these conditions the (unperturbed) differential equation
$y'=A_\text{pc}(\l,x)y$ has exponential dichotomy on $\bbR_\pm$ and is
Fredholm of index zero by Palmer's result cited in Remark \ref{rem:afterh1}(a) above. This corresponds to fact that $\l$ does not belong to the essential spectrum of 
the underlying differential operator that appears when one linearizes the
PDE about the front, cf.\ \cite{Sa02}. Again, as in the previous example
the perturbed differential equation $y'=(A_\text{pc}(\l,x)+B(x))y$
inherits the dichotomies on half lines if either $B\in L^1(\R,\C^{d,d})$
or $B(\cdot)$ is bounded continuous with $\lim_{x\rightarrow \pm
\infty}B(x)=0$. In both cases we are in the system class described
in Hypothesis \ref{hyp1}. We cite \cite{BL,LT,SaS08} for further references.
 \hfill$\Diamond$\end{example}

As described in Section \ref{s:abst}, we consider an inhomogenous equation
\begin{equation} \label{inhom}
F(\l)y(\l)= \widehat{v} \in L^2.
\end{equation}
As above, let $S(x,\x,\l),x,\xi\in \R,\l\in \Omega$ denote the solution operator of 
$F(\l)$ (the propagator of the differential equation \eqref{dfens}), and let $P_{\pm}(\l,x)$ be the dichotomy projections for \eqref{dfens} on $\R_+,\R_-$
from Hypothesis \ref{hyp1}(ii).
Then we have for some $C,\alpha >0$,
\begin{equation} \label{dichotomy}
\begin{aligned} 
S(x,\x,\l)P_{\pm}(\l,\x)=& P_{\pm}(\l,x)S(x,\x,\l) \quad \text{for} \quad
x,\xi \in \R_{\pm}, \\
\|S(x,\x,\l)P_{\pm}(\l,\x)\| \le& C e^{-\alpha(x-\x)}, \quad \text{for}
\quad \x\le x, \; x,\x \in \R_{\pm},\\
\|S(x,\x,\l)(I-P_{\pm}(\l,\x))\| \le& C e^{-\alpha(\x-x)}, \quad \text{for}
\quad x\le \x, \; x,\x \in \R_{\pm}.
\end{aligned}
\end{equation}
Let $\cE$ be the Evans function from Theorem \ref{evexist} $(i)$
with respect to the equivalence classes $[P(\l)], [Q(\l)]$ for any
choice of holomorphic functions $P,Q$ such that
$\mathcal{R}(P(\l))=\mathcal{R}(P_+(\l,0))$
and $\mathcal{R}(Q(\l))=\mathcal{N}(P_-(\l,0))$, i.e.
\begin{equation} \label{evlambda1}
\cE(\l) \in \cD([P(\l)],[Q(\l)]), \quad \l \in \Omega.
\end{equation}
Alternatively, let $\cE_0$ be the normalized Evans function from Theorem \ref{evexist} $(v)$
with respect to the subspaces $U(\l)=\mathcal{R}(P_+(\l,0))$ and
$V(\l)=\mathcal{N}(P_-(\l,0))$ in $\C^d$, i.e.
\begin{equation} \label{evlambda2}
\cE_0(\l) \in \cD(U(\l),V(\l)), \quad \l \in \Omega\setminus\Lambda(U,V).
\end{equation}
We solve \eqref{inhom} for $\l\in \rho(F)$ in a standard way by using
Green's operators
\begin{equation} \label{greenform}
\begin{aligned}
\big(\cG_+(\l) \hat{v}\big)(x) = & \int_0^{\infty} G_+(x,\x,\l)\hat{v}(\x) d\x, \quad
x \ge 0, \\
\big(\cG_-(\l) \hat{v}\big)(x) = & \int_{-\infty}^{0} G_-(x,\x,\l)\hat{v}(\x) d\x, \quad
x \le 0,
\end{aligned}
\end{equation}
acting on functions $\hat{v}:\R\to\C^d$, with the kernels given by
\begin{equation} \label{greenkernels}
\begin{aligned}
G_+(x,\x,\l)=& \left\{ \begin{array}{cc}
                      S(x,\x,\l)P_+(\l,\x) & 0 \le \x \le x, \\
                      S(x,\x,\l)(P_+(\l,\x)-I) & 0\le x < \x,
                     \end{array}
              \right. \\
G_-(x,\x,\l)=& \left\{ \begin{array}{cc}
                      S(x,\x,\l)P_-(\l,\x) & \x \le x \le 0, \\
                      S(x,\x,\l)(P_-(\l,\x)-I) & x < \xi \le 0.
                     \end{array}
              \right.
\end{aligned}
\end{equation}
Due to the exponential dichotomies the operators $\cG_+,\cG_-$ have
uniform bounds in all spaces $L^p(\R,\C^d)$, $1 \le p \le \infty$, see, e.g., \cite{Cop} or \cite[Sec.4.2]{CL}.
The following piecewise defined function gives the general solution of 
\eqref{inhom} on both half lines (we write $\hat{v}_+=\hat{v}_{|\R_+}$, $\hat{v}_-=\hat{v}_{|\R_-}$
for short),
\begin{equation} \label{solrep}
y(\l,x) =  \begin{cases}
                S(x,0,\l) \eta_+(\l) + \big(\cG_+(\l)\hat{v}_+\big)(x), & x\ge 0, \\
                -S(x,0,\l) \eta_-(\l) + \big(\cG_-(\l)\hat{v}_-\big)(x), & x < 0, \\
                 \end{cases}
\end{equation}
where $\eta_+(\l) \in U(\l)$ and $\eta_-(\l)\in V(\l)$ are arbitrary.
The function defined in \eqref{solrep} is a solution of the differential
equation \eqref{dfens} if the left and right limits of $y(\l)$ at zero coincide. Note
that if $y_{\pm}\in AC_{\loc}(\R_{\pm},\C^d)$ then the function
$y$ defined by
\[
y(x)= \left\{ \begin{array}{cc}
                              y_+(x),& x \ge 0, \\ y_-(x), & x< 0,
                 \end{array} \right.
\]
 is in $AC_{\loc}(\R,\C^d)$ if and only if $y_+(0)=y_-(0)$.
Therefore  we have to find vectors $\eta_+(\l)\in U(\l)$
and $\eta_-(\l)\in V(\l)$ such that
\[ \eta_+(\l)+\eta_-(\l) = \big(\cG_-(\l)\hat{v}_-\big)(0)-\big(\cG_+(\l)\hat{v}_+\big)(0)
=: [\hat{v}]_0.
\]
 By Theorem \ref{projecthol} the sought for vectors are
given by
\begin{equation} \label{etasolve}
\eta_+(\l)= \frac{1}{\cE(\l)} \cY_U(\l) [\hat{v}]_0, \quad
\eta_-(\l)= \frac{1}{\cE(\l)} \cY_V(\l) [\hat{v}]_0.
\end{equation}
Inserting this into \eqref{solrep} gives the solution formula
\begin{equation} \label{solrep2}
y(\l,x) =  \left\{ \begin{array}{cc}
                \frac{1}{\cE_0(\l)}S(x,0,\l) \cY_U(\l) [\hat{v}]_0  + \big(\cG_+(\l)\hat{v}_+\big)(x), & x\ge 0,\\
                -\frac{1}{\cE_0(\l)}S(x,0,\l)\cY_V(\l) [\hat{v}]_0  + \big(\cG_-(\l)\hat{v}_-\big)(x), & x < 0. \\
                 \end{array}
         \right.
\end{equation}
The same formulas hold with $\cE$ replaced by $\cE_0$ provided $\lambda\in\Omega\setminus\Lambda(U,V)$.
It is convenient to  introduce the operators $\cG(\l)$ acting on functions $v:\R\to\C^d$ by
\begin{equation} \label{newrep}
 \big(\cG(\l)v\big)(x) = \left\{ \begin{array}{cc}
                            ( \cG_+(\l) v_{|\R_+})(x), & x \ge 0\\
                             ( \cG_-(\l) v_{|\R_-})(x), & x <0, 
                            \end{array}
                   \right.
\end{equation}
and the matrix valued function $G$ by
\begin{equation} \label{jumpop}
 G(\l,x) = \left\{ \begin{array}{cc}
                            S(x,0,\l)\cY_U(\l), & x \ge 0, \\
                              - S(x,0,\l) \cY_V(\l), & x <0,
                           \end{array}
                   \right.
\end{equation} 
so that if $v_0\in\C^d$ is a given vector then $G(\l,\cdot)v_0:\R\to\C^d$.

Using these notions in \eqref{solrep2} and inserting them into \eqref{defE}
finally leads to the expression \eqref{repEODE} below for 
$E_{jk}(\l),j=1,\ldots,m$, $k=1,\ldots,\ell$. To formulate the result,
 recall that the operators $F(\l)$ act from $\cH$ 
 into $\cK$ , see \eqref{1stOrderOP}, \eqref{domainF},\eqref{rangeF}.
Also, note that $\cH\subset L^2(\R,\C^d)\subset \cH'$ such that a function $w\in L^2(\R,\C^d)$ defines on $\cH$ a linear functional by $\langle w,v\rangle = \langle w,v\rangle_\R$; here and below for $w,v\in L^2(\R,\C^d)$ we denote $\langle w,v\rangle_\R=
\int_{-\infty}^\infty w(x)^\top v(x)\,dx$. Following Section \ref{s:abst}, we now choose linearly independent functions $\hat{v}_k\in L^2(\R,\C^d)$, $k=1,\dots,\ell$, and linearly independent functions $\hat{w}_j\in L^2(\R,\C^d) $, $j=1,\dots, m$, viewed as elements of $\cH'$. Thus, the discussion above can be recorded as follows.
\begin{theorem} \label{formulaE}
Assume Hypotheses \ref{hyp1}, and let all eigenvalues of the operator pencil $F$ inside $\Omega_0$
be simple.  Then the matrix $E(\l)\in \C^{m,\ell}$ from the contour method
\eqref{defE},\eqref{defB01} satisfies the following formula
\begin{equation} \label{repEODE}
E_{jk}(\l) = \frac{1}{\cE(\l)}\langle \widehat{w}_j, G(\l,\cdot) [\hat{v}_k]_0
\rangle_\R + \langle \widehat{w}_j, \cG(\l)\hat{v}_k \rangle_\R,\quad\l\in\Omega, 
\end{equation}
where the vector $[\hat{v}_k]_0= \big(\cG_-(\l)\hat{v}_k\big)(0-)-\big(\cG_+(\l)\hat{v}_k\big)(0+)\in\C^d$ denotes
 a jump quantity at $x=0$, and the operators $\cG,\cG_{\pm}$ and the matrix valued function  $G$ are  defined in \eqref{greenform}, \eqref{newrep} and \eqref{jumpop}. For $\l\in\Omega\setminus\Lambda(U,V)$ formula \eqref{repEODE} holds with $\cE$ replaced by the normalized Evans function $\cE_0$.
\end{theorem}
 Formula \eqref{repEODE} shows how to express the abstract
terms in \eqref{Erepsing} through integral kernels and 
the Evans function (or the normalized Evans function) for first order systems.

To conclude this subsection, we summarize our results for the operator pencil \eqref{1stOrderOP}. In particular,  we apply Theorem \ref{zerobehavior} to recover the singular part of 
the function $E_{jk}(\cdot)$ near a simple eigenvalue $\l_n, n=1,\ldots,\varkappa$. Recall that  $P(\lambda), Q(\lambda)\in\bbH(\Omega,\C^{d,k})$ are chosen as in \eqref{reploc}  with $\Pi_U(\lambda)=P_+(\lambda,0)$ and 
$\Pi_V(\lambda)=I-P_-(\lambda,0)$, and the Evans function is defined by $\cE(\lambda)=\det\big(P(\lambda)\big| Q(\lambda)\big)$.
\begin{theorem}\label{thmSum} Assume Hypotheses \ref{hyp1}. The following assertions are equivalent.

(i)\, $\lambda_0$ is a simple eigenvalue of the operator pencil $F$ \eqref{1stOrderOP};

(ii)\, $\lambda_0$ is a simple root of the Evans function $\cE$;

(iii)\, $\dim \mathcal{N}\big(P(\lambda_0)\big| Q(\lambda_0)\big)=1$;

(iv)\, $\dim\big(\mathcal{R}(P_+(\lambda_0,0))\cap \mathcal{N}(P_-(\lambda_0,0))\big)=1$;

(v)\, There exists a unique up to a scalar multiple exponentially decaying at $\pm\infty$ 

\qquad solution $v$ of \eqref{dfens};

(vi)\, There exists a unique up to a scalar multiple exponentially decaying at $\pm\infty$ 

\qquad solution $w$ of the adjoint 
 to \eqref{dfens} equation $(z^\top)'=-z^\top A(\lambda,x)^\top$.

\noindent Moreover, if $v(0)$ denotes the initial value of the
exponentially decaying solution in $(v)$, then $v(0)\in
\mathcal{R}(P_+(\lambda_0,0))\cap \mathcal{N}(P_-(\lambda_0,0))$ if and only if
\begin{equation}\label{PQPP}
v(0)=P(\lambda_0)v_{0,1}=-Q(\lambda_0)v_{0,2}
\end{equation}
for a vector $v_0=\begin{pmatrix}v_{0,1} \\v_{0,2} \end{pmatrix}$
  from $\mathcal{N}\big(P(\lambda_0)\big| Q(\lambda_0)\big)$. 

In addition, assume that all eigenvalues $\l_n$, $n=1,\ldots,\varkappa$, in $\Omega_0$ are simple, 
let $v_n, w_n^\top$  be the solutions described in assertions $(v), (vi)$  above for each $\lambda_n$, and normalized as indicated in \eqref{normalize}, 
let $\{\widehat{v}_k\}_{k=1}^\ell$, $\{\widehat{w}_j\}_{j=1}^m$ be linearly independent functions in $L^2(\R,\C^d)$ chosen as indicated in Section \ref{s:abst}. Then the singular part of the function $E_{jk}(\cdot)$ defined in \eqref{defE} in the framework of the abstract Keldysh theorem \eqref{eexpress} is given by the formula
\begin{equation}\label{cKf}
E^{\mathrm{sing}}_{jk}(\lambda)=\sum_{n=1}^\varkappa\frac{1}{\l-\l_n}
\langle\widehat{w}_j,v_n\rangle_\R\,\langle w_n,\widehat{v}_k\rangle_\R,
\quad |\l-\l_n| \ll 1, \, n=1,\dots,\varkappa.
\end{equation}
\end{theorem}
\begin{proof}
  The three assertions, $\cE(\l_0)=0$, $\mathcal{R}(P_+(\lambda_0,0))\cap
\mathcal{N}(P_-(\lambda_0,0))\neq\{0\}$, and $\mathcal{N}(F(\l_0))\neq\{0\}$, are equivalent by Theorem
\ref{evexist} and by the dichotomy assumptions in Hypothesis \ref{hyp1}. 
Moreover, $\mathcal{R}(P_+(\lambda_0,0))\cap \mathcal{N}(P_-(\lambda_0,0))$ and
$\mathcal{N}(F(\l_0))$ are isomorphic via the map $v(0)\mapsto
v(\cdot)=S(\cdot,0,\l_0)v(0)$. Thus to see the equivalence of the first
four items in the theorem it suffices to show that the subspaces
$\mathcal{R}(P_+(\lambda_0,0))\cap \mathcal{N}(P_-(\lambda_0,0))$ and $\mathcal{N}\big(P(\lambda_0)\big| Q(\lambda_0)\big)$
are isomorphic as indicated in \eqref{PQPP}.
Let $v_0\in \mathcal{N}(P(\lambda_0)\big| Q(\l_0))$. Since
$\mathcal{R}(P(\lambda_0))=\mathcal{R}(P_+(\l_0,0))$ and
$\mathcal{R}(Q(\lambda_0))=\mathcal{N}(P_-(\lambda_0,0))$, we have
$P(\lambda_0)v_{0,1}\in \mathcal{R}(P_+(\l_0,0))$ and $Q(\lambda_0) v_{0,2}\in
\mathcal{N}(P_-(\lambda_0,0))$, and by the choice of $v_0$ we have
$P(\lambda_0)v_{0,1}=-Q(\lambda_n) v_{0,2}$. Then  $v(0)$ from
\eqref{PQPP} belongs to $\mathcal{R}(P_+(\lambda_0,0))\cap
\mathcal{N}(P_-(\lambda_0,0))$. Conversely, if $v(0)\in
\mathcal{R}(P_+(\lambda_0,0))\cap \mathcal{N}(P_-(\lambda_0,0))$
then $v(0)=P(\l_0)v_{0,1}$ and $v(0)=-Q(\l_0)v_{0,2}$ for some
$v_{0,1}\in\C^k$, $v_{0,2}\in\C^{d-k}$. Letting
$v_0=\begin{pmatrix}v_{0,1} \\v_{0,2} \end{pmatrix}$ yields $v_0\in
  \mathcal{N}\big(P(\lambda_0)\big| Q(\lambda_0)\big)$, as required. To
  begin the proof of  $(vi)$, we remark that $\dim
  \mathcal{N}\big(P(\lambda_0)\big| Q(\lambda_0)\big)$ is equal to $\dim
  \mathcal{N}\big((P(\lambda_0)\big| Q(\lambda_0))^\top\big)$;  also, the following identities hold:
\begin{align*}
  \mathcal{N}\big((P(\lambda_0)\big|& Q(\lambda_0))^\top\big)=
  \mathcal{R}\big((P(\lambda_0)\big| Q(\lambda_0))\big)^\bot=
  \mathcal{R}\big(P(\lambda_0)\big)^\bot\cap \mathcal{R} \big(Q(\lambda_0)\big)^\bot\\&=
  \mathcal{R}(P_+(\lambda_0,0))^\bot\cap \mathcal{N}(P_-(\l_0,0))^\bot=
  \mathcal{N}(P_+(\lambda_0,0)^\top)\cap \mathcal{R}(P_-(\l_0,0)^\top)\\&=
  \mathcal{R}(I-P_+(\lambda_0,0)^\top)\cap \mathcal{N}(I-P_-(\l_0,0)^\top).
\end{align*}
Since \eqref{dfens} has the exponential dichotomy $P_\pm(\l_0,0)$ if and
only if the adjoint equation has the exponential dichotomy
$I-P_\pm(\l_0,0)^\top$ (see, e.g., \cite[Lem.4.5]{BAG},
\cite[Lem.2.4]{DL}, \cite[Rem.3.4]{Sa02}), it follows that the subspaces
$\mathcal{N}\big((P(\lambda_0)\big| Q(\lambda_0))^\top\big)$ and
$\mathcal{N}(F(\l_0)^\top)$ are isomorphic via the map $w(0)\mapsto w(\cdot)^\top$, where $w$ is the exponentially decaying solution of the adjoint equation.

It remains to show \eqref{cKf}. For each $n=1,\dots,\k$
 let $v_n(0)$ be the vector from \eqref{PQPP}, let $v_n(x)=S(x,0,\lambda_n)v_n(0)$ be the decaying on both $\R_+$ and $\R_-$ solution of \eqref{dfens}, and let 
 $w_n(0)$ span $\mathcal{N}((P(\lambda_n)\big| Q(\l_n))^\top)$. We obtain from \eqref{batsing} that \[\cY_U(\lambda_n)=\cE'(\lambda_n)v_n(0)w_n(0)^\top\, \text{ and }\,
\cY_V(\lambda_n)=-\cE'(\lambda_n)v_n(0)w_n(0)^\top,\] and from \eqref{jumpop} that $G(\lambda_n,x)=\cE'(\lambda_n)v_n(x)w_n(0)^\top$.
 Using \eqref{repEODE} and 
$\cE(\l)=\cE'(\l_n)(\l-\l_n) +\mathcal{O}(|\l-\l_n|^2)$ 
yields the singular part
\begin{equation} \label{Erepsing2}
 E^{\mathrm{sing}}_{jk}(\l)=\frac{1}{\l-\l_n}
 \langle \widehat{w}_j,v_n\rangle w_n(0)^{\top}[\hat{v}_k]_0,\quad |\lambda-\l_n| \ll 1.
\end{equation}
Note that 
\begin{align*}
 w_n(0)^{\top}[\hat{v}_k]_0=& w_n^{\top}(0)
\left( \int_{-\infty}^0 S(0,\x,\l)P_-(\l,\x) \hat{v}_k(\x)d\x \right. \\
- & \left. \int_{0}^{\infty}S(0,\xi,\l)(P_{+}(\l,\x)-I) \hat{v}_k(\xi) d\x \right).
\end{align*}
Next we observe that 
\[ w_n^{\top}(\xi)=\left\{ \begin{array}{cc}
                     w_n^{\top}(0)S(0,\x,\l)P_-(\l,\x), & \x \le 0, \\
                     w_n^{\top}(0)S(0,\x,\l)(I-P_+(\l,\x)), & \x > 0
                    \end{array} \right. 
\]
 solves the adjoint equation of \eqref{dfens}, is continuous at $0$, and
decays exponentially in both directions. Using this  in \eqref{Erepsing2}
finally leads us back to the singular part determined by the abstract
Keldysh theorem in  \eqref{eexpress}, where the term $\langle w_n, \hat{v}_k\rangle$ is understood as the integral $\langle w_n, \hat{v}_k\rangle_\R$.
\hfill\end{proof}

%
%


\section{Convergence of eigenvalues for finite boundary value problems}
\label{s:converge}
In this section we provide error estimates of the eigenvalues obtained by the contour
method  in Section \ref{s:abst} when the boundary value problems \eqref{inhom} are
solved approximately on a bounded interval. In the first step we analyze the error of the boundary
value solutions themselves, and in the second step we discuss the implications for
the contour method.
\subsection{Estimates of boundary value solutions}
\label{bvestimates} Using the  setting and notation from 
Section \ref{sec3.2}
we consider the all-line boundary value problem
\[F(\lambda)y(\lambda)=-y'(\lambda,\cdot)+A(\lambda,\cdot)y(\lambda,\cdot)=\widehat{v}\in
L^2(\mathbb{R})\]
 for various values of $\lambda$ and $\widehat{v}$. Our main assumption is the following.
\begin{hypothesis}\label{hyp2} Let
  Hypotheses \ref{hyp1} hold and
  assume that the dichotomy exponent $\alpha>0$ in \eqref{dichotomy} 
  is uniform  for all $\lambda\in\Omega$ and that the dichotomy
  projections $P_\pm(\lambda,x)$ given in Hypothesis \ref{hyp1} (ii)
  are asymptotically constant, that is, 
  $\lim_{x\to\pm\infty}P_\pm(\lambda,x)=P_{\pm}(\lambda)$.
\end{hypothesis}
We approximate \eqref{inhom} by a sequence
of boundary value problems on finite intervals 
$J_N=[x_-^N,x_+^N]$, $N\in \mathbb{N}$, with $-x_-^N,x_+^N\to\infty$ as $N\to\infty$:
\begin{equation}\label{eq:inhomapprox}
  F_N(\lambda)y:=\begin{pmatrix}-y'+A(\lambda,\cdot)y\\
    R_-(\lambda)y(x_-^N)+R_+(\lambda)y(x_+^N) 
  \end{pmatrix}
  =\begin{pmatrix}\widehat{v}_{|J_N}\\0
  \end{pmatrix}\in L^2(J_N)\times \mathbb{C}^d,
\end{equation}
where $R_-,R_+\in \mathbb{H}(\Omega,\mathbb{C}^{d,d})$ are given matrix valued functions. 
We allow 
boundary conditions that are nonlinear in the eigenvalue parameter
in order to cover so-called projection boundary conditions 
 which lead to fast convergence of solutions of \eqref{eq:inhomapprox}
 as $-x_-^N,x_+^N \rightarrow
\infty$, see \eqref{eq:projbc} below.  

We will show how \eqref{eq:inhomapprox} fits into the framework of Section \ref{s:approxop} and apply Theorem \ref{discreteconverge} to obtain 
error estimates. Our approach is largely based on \cite{BR07} where
the case of smooth coefficients and $\dom F(\lambda)=H^1(\R,\C^d)$, the Sobolev space, was analyzed.
For any interval $J \subseteq \R$ introduce the  Banach space
\begin{equation} \label{eq:HJdef}
\mathcal{H}_J=\big\{y\in L^2(J,\mathbb{C}^d):y\in AC_{\mathrm{loc}}(J,\mathbb{C}^d),
-y'+By\in L^2(J,\mathbb{C}^d)\big\}
\end{equation}
with norm
$\|y\|_{\mathcal{H}_J}^2=\|y\|_{L^2(J)}^2+\|-y'+By\|_{L^2(J)}^2$.
Note that $\mathcal{H}_{\R}$ agrees with $\mathcal{H}$ from 
\eqref{domainF}.  Using Lemma \ref{lem:A2} it is easy to see that the space $\cH_J$ is complete in this norm (for later reference we collect
further properties of  $\mathcal{H}_J$ in the Appendix).

The spaces $\mathcal{H}_N$ and $\mathcal{K}_N$ from Section
\ref{s:approxop} are defined by
$\mathcal{H}_N=\mathcal{H}_{J_N}$ and
$\mathcal{K}_N=L^2(J_N,\mathbb{C}^d)\times
\mathbb{C}^d$
with  norms
$\|y\|_{\mathcal{H}_N}$
and $\|(v,r)\|_{\mathcal{K}_N}^2=\|v\|_{L^2}^2+|r|^2$. The spaces 
$\mathcal{H}$, $\mathcal{K}$ are taken as in \eqref{domainF}, \eqref{rangeF}
and mapped into $\mathcal{H}_N$, $\mathcal{K}_N$ by
\begin{equation}\label{defpq}
p_Ny=y|_{J_N} \quad \text{and} \quad 
q_Nv=(v|_{J_N},0).
\end{equation}
Obviously these mappings are linear, bounded uniformly in $N$ and
satisfy condition
(D1) from Section \ref{s:approxop}. 

 By Hypothesis \ref{hyp1} (i),(ii) we have  $F\in \mathbb{H}(\Omega,\mathcal{F}(\mathcal{H},\mathcal{K}))$ since $F(\l)$ is Fredholm for each $\l$ in $\Omega$
and the map $\Omega\ni \lambda \mapsto
A(\lambda,\cdot)-B(\cdot)\in L^\infty(\mathbb{R},\mathbb{C}^{d,d})$
is holomorphic. Moreover,
$\rho(F)\ne\emptyset$  by Hypothesis \ref{hyp1}(iii), and (D2)
follows from the next lemma.
\begin{lemma}\label{lem:D2}
  Under the above assumptions on $A,R_-,R_+$ the operators $F_N$ are
 in
 $ \mathbb{H}(\Omega,\mathcal{F}(\mathcal{H}_N,\mathcal{K}_N))$ 
  and $\sup_{N\in \mathbb{N}}\sup_{\lambda\in
  \cC}\|F_N(\lambda)\|<\infty$ for every compact set $\mathcal{C}\subset
  \Omega$.
\end{lemma}
\begin{proof}
  Let $y\in \mathcal{H}_N$ and $\lambda \in \Omega$. Then, by \eqref{AminBhol}
and Lemma \ref{lem:A2},
  \begin{multline*}
    \|F_N(\lambda)y\|_{\mathcal{K}_N}^2=
    \|-y'+A(\lambda,\cdot)y\|_{L^2(J_N)}^2+
    |R_-(\lambda)y(x_-^N)+R_+(\lambda)y(x_+^N)|^2\\
    \le 2 \big(
    \|y\|_{\mathcal{H}_N}^2+\|(A(\lambda,\cdot)-B(\cdot))y\|_{L^2(J_N)}^2 \big)
    + 2 \big( \|R_-(\lambda)\|^2+\|R_+(\lambda)\|^2
    \big)\|y\|_{L^\infty}^2\\
    \le c \big( \|A(\lambda,\cdot)-B(\cdot)\|_{L^\infty}^2 +
    \|R_-(\lambda)\|^2+\|R_+(\lambda)\|^2\big)\|y\|_{\mathcal{H}_N}^2.
  \end{multline*}
   From the holomorphy
  of $A,R_-,R_+$ we obtain uniform bounds for
  $\|F_N(\lambda)\|$ on compact sets $\mathcal{C}\subset \Omega$ as well
  as holomorphy of $\lambda\mapsto F_N(\lambda)$ for all $N\in
  \mathbb{N}$. Finally, the Fredholm property is
  a well-known fact for finite boundary value problems.
\hfill\end{proof}

For the application of Theorem \ref{discreteconverge} it remains to
verify (D3). Let $V_-^s(\lambda)$ be a basis of
the range of $P_{-}(\lambda)$ and let $V_+^u(\lambda)$ be a basis
of the kernel of $P_{+}(\lambda)$.

\begin{proposition}\label{prop:D3}
  Let Hypothesis \ref{hyp2} hold and 
  assume that matrices $R_\pm(\l)$ for all $\lambda\in\Omega$ satisfy
  \begin{equation}
    \det\big((R_-(\lambda)V_-^s(\lambda) \big| R_+(\lambda)
    V_+^u(\lambda))\big) \ne 0.
    \label{eq:detcond}
  \end{equation}
  Then $F_N(\lambda)$ converges regularly to $F(\lambda)$ for all
  $\lambda\in\Omega$.
\end{proposition}
\begin{proof}
  Let $y\in \mathcal{H}$, $\lambda\in\Omega$. Then
  \begin{equation}
    \|F_N(\lambda)p_Ny-q_NF(\lambda)y\|_{\mathcal{K}_N}^2\le 
    2\bigl( \|R_-(\lambda)\|^2+\|R_+(\lambda)\|^2
    \bigr)(|y(x_-^N)|^{2}+|y(x_+^N)|^2),
  \end{equation}
  where the right-hand side converges to $0$ as $N\to\infty$ by Lemma
  \ref{lem:A1}. This proves part (a) of (D3).

  Let $\lambda\in\Omega$ be fixed. Consider a subsequence
  $y_N\in \mathcal{H}_N$, $N\in \N'$, with bounded 
$\|y_N\|_{\mathcal{H}_N}$ and assume there is $v\in
  \mathcal{K}$ with
  $\lim_{N\to\infty}\|F_N(\lambda)y_N-q_Nv\|_{\mathcal{K}_N}=0$.
  Set $(v_N,r_N):=F_N(\lambda)y_N\in \mathcal{K}_N$ and note that $y_N$ can be
  written similarly to \eqref{solrep}:
  \begin{equation} \label{eq:ynrep}
    y_N(x)=
    \begin{cases}
      G_+(x,0,\lambda)y_N(0)-G_+(x,x_+^N,\lambda)y_N(x_+^N)
      +[\mathcal{G}_+^N(\lambda)v_N](x),&x\ge 0,\\
      -G_-(x,0,\lambda)y_N(0)+G_-(x,x_-^N,\lambda)y_N(x_-^N)
      +[\mathcal{G}_-^N(\lambda)v_N](x),&x\le 0,
    \end{cases}
  \end{equation}
  where $G_\pm$ are defined in \eqref{greenkernels} and, cf.\ \eqref{greenform},
  \begin{align*}
    \big(\mathcal{G}_+^N(\lambda)v_N\big)(x)&
    =\int_0^{x_+^N}G_+(x,\xi,\lambda)v_N(\xi)\,d\xi,\quad x_{+}^N\ge
    x\ge 0,\\
    \big(\mathcal{G}_-^N(\lambda)v_N\big)(x)&
    =\int_{x_-^N}^0 G_-(x,\xi,\lambda)v_N(\xi)\,d\xi,\quad x_-^N\le x\le 0.
  \end{align*}
  By Lemma \ref{lem:A2} the boundedness of
  $\|y_N\|_{\mathcal{H}_N}$ implies  boundedness and thus compactness of the sequence $(y_N(0))_{N\in \N'}\subset\C^d$, i.e.\
  $\lim_{\N''\ni
  N\to\infty}y_N(0)=y_0$ for some subsequence $\mathbb{N}''\subset
  \mathbb{N}'$ and some $y_0\in \mathbb{C}^d$.
  Let 
  \begin{equation}\label{eq:yallline}
    y(x)=
    \begin{cases}
      y_+(x)=G_+(x,0,\lambda)y_0
      +[\mathcal{G}_+(\lambda)v_{|\mathbb{R}_+}](x),&x\ge 0,\\
      y_-(x)=-G_-(x,0,\lambda)y_0
      +[\mathcal{G}_-(\lambda)v_{|\mathbb{R}_-}](x),&x< 0.
    \end{cases}
  \end{equation}
  Now we follow verbatim the proof of \cite[Thm.2.1]{BR07} until
  \cite[(2.11)]{BR07} to conclude
  \begin{equation}\label{eq:L2convyn}
    \lim_{\mathbb{N}''\ni N\to\infty}\|y_N-y_{|J_N} \|_{L^2(J_N)}^2= 0.
  \end{equation}
  Note that this step is crucial. It uses  the determinant condition 
  \eqref{eq:detcond} as well as the representation \eqref{eq:ynrep} and the
  exponential dichotomies. It remains to prove 
  $\|y_N-p_Ny\|_{\mathcal{H}_N}\to 0$ as $\mathbb{N}''\ni N\to\infty$,
  for which the arguments in \cite{BR07} do no longer apply.

  By construction $y_{\pm}\in AC_\loc(\mathbb{R}_\pm,\mathbb{C}^d)\cap
  L^2(\mathbb{R}_\pm,\mathbb{C}^d)$ and
  \begin{equation}
    -y_{\pm}'+A(\lambda,\cdot)y_{\pm}=v_{|\R_{\pm}}
  \end{equation}
  holds in $L^2(\mathbb{R}_{\pm},\mathbb{C}^d)$. Therefore, the function
  defined by
  \begin{equation} \label{eq:zdef}
    z(x) = 
    \begin{cases} -y'_+ +B(x) y_+(x),& x\ge 0, \\
      -y'_- +B(x) y_-(x),& x< 0, 
    \end{cases}
  \end{equation}
  is in $L^2(\R,\mathbb{C}^d)$ and satisfies $z= v+(B(\cdot)-A(\lambda,\cdot))y$.
  Using this we obtain
  \begin{multline}\label{eq:Hconvyn}
    \|-y_N'+B(\cdot)y_N- z_{|J_N}\|_{L^2(J_N)}
    = \|-y_N'+B(\cdot)y_N -(v+(B(\cdot)-A(\lambda,\cdot))y)_{|J_N}\|_{L^2(J_N)} \\
    \le  \|F_N(\lambda)y_N-q_Nv\|_{\mathcal{K}_N}+
    \|A(\lambda,\cdot)-B(\cdot)\|_{L^\infty}
    \|y_N-y_{|J_N}\|_{L^2(J_N)},
  \end{multline}
  where the right-hand side converges to zero as
  $\mathbb{N}''\ni N\to \infty$ by \eqref{eq:L2convyn} and our assumption.
  Without loss of
  generality we may assume $J_N\supset [-1,1]$ for all $N\in \mathbb{N}''$. 
Repeating
the estimate \eqref{eq:Hconvyn} with $[-1,1]$ instead of $J_N$ shows that
  $({y_N}|_{[-1,1]})_{N\in \mathbb{N}''}$ is a Cauchy sequence in
  $\mathcal{H}_{[-1,1]}$.
  Therefore, its limit, which coincides with $y_{|[-1,1]}$,
  is an element of $AC([-1,1])$. This allows us to conclude
$z=-y'+B(\cdot)y$ from \eqref{eq:zdef}, so that equations 
 \eqref{eq:L2convyn} and \eqref{eq:Hconvyn} prove our final assertion.
\hfill\end{proof}

The above results show that the abstract convergence result, Theorem
\ref{discreteconverge}, applies:
\begin{theorem}\label{thm:approxbybvp}
  Let the assumptions of Proposition \ref{prop:D3} hold.
  Then for any compact set $\mathcal{C}\subset \rho(F)\cap \Omega$ and any
  $\widehat{v}\in
  \mathcal{K}$ there is $N_0\in \mathbb{N}$ such that for all $N\ge N_0$
  and $\lambda\in\mathcal{C}$ the boundary value
  problem \eqref{eq:inhomapprox} has a unique solution
  $y_N(\lambda,\cdot)\in \mathcal{H}_N$. Furthermore, for some constant
  $C$, independent of $\widehat{v}$,
  \begin{equation}\label{eq:bvperror}
    \sup_{\lambda\in\mathcal{C}}
    \|y_N(\lambda,\cdot)-p_Ny(\lambda,\cdot)\|_{\mathcal{H}_N}\le
    C \sup_{\lambda\in
    \mathcal{C}}|R_-(\lambda)y(\lambda,x_-^N)+R_+(\lambda)y(\lambda,x_+^N)|,
  \end{equation}
  where $y(\lambda,\cdot)\in \mathcal{H}$ solves \eqref{inhom}.
\end{theorem}
From the well known decay $|y(\lambda,x_{\pm}^N)|\to 0$ as $N\to \infty$
(e.g. see the proof of \cite[Thm.3.2]{BL}) 
estimate \eqref{eq:bvperror} implies
 that the solutions $y_N(\lambda,\cdot)$ of the
finite interval problems converge uniformly in $\lambda\in
\mathcal{C}$ to the solution of the problem on the line.

We will now concentrate on the differential equation \eqref{dfens} with the coefficient of the special perturbative structure which appears in the case of traveling fronts with asymptotic hyperbolic rest states, see Example \ref{front}. 
Specifically, let us consider a first order operator of the form \eqref{1stOrderOP} where
  \begin{equation}\label{tfe}
  A(\lambda,x)=A_{\text{pc}}(\l,x)+B(x), \quad x\in\R,\end{equation} with $A_{\text{pc}}(\l,x)$ defined in \eqref{APC}. We impose the following assumptions.
  
 
  
  \begin{hypothesis}\label{hyp3}
    The differential equation \eqref{dfens} with $A(\l,x)$ from \eqref{tfe} satisfies Hypothesis
    \ref{hyp1} and Hypothesis \ref{hyp2} with the uniform
    exponential estimate, 
    \begin{equation}\label{eq:ProjConv}
      \|P_\pm(\lambda,x)-P_{\pm}(\lambda)\|\le c e^{-\alpha|x|}, \quad x\in\R_\pm,
    \end{equation}
    for all $\lambda\in\Omega$, where $\alpha$ is the exponent from \eqref{dichotomy} of the
    exponential dichotomy on $\R_\pm$ for \eqref{dfens}. The projections $P_\pm(\l)$ depend analytically
    on $\lambda\in\Omega$.
  \end{hypothesis}

  A typical situation where Hypothesis \ref{hyp3} is satisfied, 
occurs when the matrix-valued function 
  $A(\lambda,\cdot)=A_{\rm{pc}}(\l,\cdot)+B(\cdot)$ is continuous,
  assumptions (a) -- (c) in Example \ref{front} hold, and there is $c>0$
  such that for all $\lambda\in \Omega$,
  \begin{equation}\label{eq:expA}
    \|A(\lambda,x)-A_-(\lambda)\|\le c e^{-\alpha |x|},\, x\le 0,\quad
    \|A(\lambda,x)-A_+(\lambda)\|\le c e^{-\alpha |x|},\, x\ge 0,
  \end{equation}
    where $\alpha$ is the exponent of
 exponential dichotomy on $\R_\pm$ for the constant coefficient  equations
 $y'=A_\pm(\l)y$.
Then the differential equation \eqref{dfens} with $A(\l,x)$ as in \eqref{tfe} has exponential
  dichotomy on $\mathbb{R}_\pm$ for all $\lambda\in\Omega$ and the
  roughness theorem \cite[Thm.A.3]{BL} implies \eqref{eq:ProjConv}.
Thus Hypothesis \ref{hyp3} is satisfied provided \eqref{dfens} has an exponential dichotomy on $\R$ for at least one $\l\in\Omega$.

%

  Under Hypothesis \ref{hyp3}, projection boundary conditions in
  \eqref{eq:inhomapprox} are a
  suitable choice, because they always satisfy \eqref{eq:detcond} by
  construction. For convenience, we recall the definition of the
  projection boundary conditions, see \cite{B90} for more details.
  Since the limits $P_\pm(\lambda)$ depend analytically on $\lambda$,
  there are analytic bases $V_\pm^s(\lambda)$, respectively, 
  $V_\pm^u(\lambda)$ of $\mathcal{R}(P_\pm(\l))$,
  respectively, $\mathcal{N}(P_\pm(\l))$ (e.g. \cite[Sec.II.1.4]{K80}). 
  Let us split the inverse matrix composed as follows:
  \[
    \big(V_\pm^s(\lambda) \big| V_\pm^u(\lambda)
  \big)^{-1}=\begin{pmatrix}
    L_{\pm}^s(\lambda)\\L_\pm^u(\lambda)
  \end{pmatrix}, \quad L_{\pm}^s(\l)\in\C^{k,d},\, \quad L_{\pm}^u(\l)\in\C^{d-k,d}.\]
  The projection boundary conditions are then given by the boundary
  matrices
  \begin{equation}\label{eq:projbc}
    R_-(\lambda):=\begin{pmatrix}
      L_-^s(\lambda)\\0_{(d-k)\times d}
    \end{pmatrix}
    \in \mathbb{C}^{d,d},\quad
    R_+(\lambda):=\begin{pmatrix}
      0_{k\times d}\\
      L_+^u(\lambda)
    \end{pmatrix}
    \in \mathbb{C}^{d,d}.
  \end{equation}
  By construction $ \big(R_-(\lambda)V_-^s(\lambda) \big| R_+(\lambda)
    V_+^u(\lambda)\big) = I_d$, i.e. \eqref{eq:detcond} is satisfied,
and 
  \begin{equation}\label{eq:BConRange}
    R_-(\lambda)(I-P_-(\lambda))=0, \quad
    R_+(\lambda)P_+(\lambda)=0.
  \end{equation}

We now use Theorem~\ref{thm:approxbybvp} to establish our main convergence
result  for the
matrix in \eqref{defE} and the integrals  in \eqref{defB01} when 
the underlying boundary value problems are solved on a finite interval.
We denote by $\mathcal{M}_b^c=\mathcal{M}_b^c(\mathbb{R},\mathbb{C}^d)$ 
the set of finite, compactly supported, $\mathbb{C}^d$ valued Radon
measures on $\mathbb{R}$. By Riesz's Theorem, e.g.
\cite[Thm.7.17]{F99}, and Lemma~\ref{lem:A1}, $\mathcal{M}_b^c\subset
\mathcal{H}'$ for $\cH$ from \eqref{domainF}. If
$\widehat{w}\in \mathcal{H}'$ is given by $\mu\in \mathcal{M}_b^c$, we
write $\langle \widehat{w},v\rangle=\int_\mathbb{R}v^\top(x) d\mu$ for
$v\in \mathcal{H}$.  
We approximate $\widehat{w}=\mu\in \mathcal{M}_b^c\subset\mathcal{H}'$
on a finite interval $J$ by its trace  $\widehat{w}|_J=
\mu|_J$ defined through $\langle \widehat{w}|_J,v\rangle=\int_Jv^\top(x)d\mu$ for all
$v\in \mathcal{H}_J$. Obviously, 
\begin{equation}\label{eq:cpsuppwhat}
  \langle\widehat{w}|_J,p_Jv\rangle-\langle\widehat{w},v\rangle=
  \int_{\mathbb{R}\setminus J}v^\top(x) d\mu=0, \quad
\text{if} \: J\supset \supp(\mu)\; \text{and}\; v\in
  \mathcal{H}.
\end{equation}
\begin{example} \label{ex:compactf}
  Two standard examples for $\widehat{w}\in \mathcal{H}'$, given as
  $\mu\in \mathcal{M}_b^c$:
  \begin{enumerate}
    \item If $\mu=e_i \delta_{x_0}$ for some $x_0\in \mathbb{R}$,
      $i\in\{1,\dots,d\}$, then 
      $\langle\widehat{w},v\rangle=\int v^\top(x) d\mu=v(x_0)^\top e_i=v_i(x_0)$
      is the $i$'th component of $v$ evaluated at $x_0$.
    \item If $\mu$ has density $f\in
      L^1_\text{loc}(\mathbb{R},\mathbb{C}^d)$ with respect to the
      Lebesgue measure, then
      $\langle\widehat{w},v\rangle=\int_\R v^\top(x) f(x)\,dx$. \hfill$\Diamond$
  \end{enumerate}
\end{example}
After these preliminaries we define and estimate approximations
$E^N(\lambda)$ of the function $E(\lambda)$ from \eqref{defE} 
by solving finite interval boundary value problems.
\begin{theorem}\label{thm:EConv}
  Let $F$ from \eqref{1stOrderOP}
  satisfy Hypothesis \ref{hyp2}, and let condition \eqref{eq:detcond}
  hold for boundary matrices $R_\pm(\l)$. Moreover, let
  $\Gamma\subset\rho(F)\cap\Omega$ be a contour and assume
  linearly independent elements $\widehat{w}_j\in \mathcal{H}'$,
  $j=1,\dots,m$, defined by $\mu_j\in \mathcal{M}_b^c$, and linearly
  independent functions $\widehat{v}_k\in \mathcal{K}$, $k=1,\dots,\ell$ which
are bounded and have compact support.
  Then the finite interval approximation $E^N(\lambda)$ of $E(\lambda)$,
  defined by (cf. \eqref{eq:inhomapprox})
  \[
  \begin{aligned}
    F_N(\lambda)y_k^N(\lambda)&= (\widehat{v}_{k}|_{J_N},0),\, k=1,\dots,\ell,\\
    E^N(\lambda)_{jk}&=\langle\widehat{w}_j|_{J_N},
    y_k^N(\lambda)\rangle=\int_{J_N}y_k^N(\lambda)^\top d\mu_j,\, j=1,\dots,m,
  \end{aligned}
  \]
  satisfies
  \begin{equation}\label{eq:En-Eest}
    \sup_{\lambda\in\Gamma}\big\|E^N(\lambda)-E(\lambda)\big\|
    \le c e^{-\alpha\min\{-x_-^N,x_+^N\}},
  \end{equation}
  with $\alpha$ being the dichotomy exponent in $\Omega$ from
  Hypothesis~\ref{hyp2} and $c$ a uniform constant.

  If $F$ additionally satisfies Hypothesis~\ref{hyp3} and $F_N$ is given
  with the projection boundary conditions defined via
  \eqref{eq:projbc}, then
  \eqref{eq:En-Eest} improves to
  \begin{equation}\label{eq:En-Eest2}
    \sup_{\lambda\in\Gamma}\big\|E^N(\lambda)-E(\lambda)\big\|
    \le c e^{-2\alpha\min\{-x_-^N,x_+^N\}}.
  \end{equation}
\end{theorem}
As a corollary we  obtain estimates for the approximate
matrices (cf. \eqref{defB01})
\beq\lb{defBn01}D_j^N=\frac{1}{2\pi i}\int_\G \l^j E^N(\l)\,d\l,\quad
j=0,1.\enq
\begin{corollary}\label{cor:4.8}
  Under the assumptions of Theorem~\ref{thm:EConv} the following estimates hold
  \begin{equation}\label{eq:422}
    \|D_0-D_0^N\|\le c e^{-\alpha \min\{-x_-^N,x_+^N\}},\quad
    \|D_1-D_1^N\|\le c e^{-\alpha \min\{-x_-^N,x_+^N\}}.
  \end{equation}
  In the case of projection boundary conditions the constant $\alpha$
  improves to $2\alpha$.
\end{corollary}
\begin{proof}[Proof of Theorem~\ref{thm:EConv}]
  Throughout the proof, $c$ is a generic constant.
  For all $\lambda\in \Gamma\subset\rho(F)$ the differential equation \eqref{dfens} has an
  exponential dichotomy on $\mathbb{R}$ with a uniform exponent $\alpha$
  and projections $P(\l,x)=P_\pm(\l,x)$ that depend
  holomorphically on $\lambda$ and satisfy
  $\lim_{x\to\pm\infty}P(\lambda,x)=P_\pm(\lambda)$, see
  \cite[Thm.~A.5]{BL}. With these
  projections the Green's function reads
  \[G(x,\xi,\lambda)=
  \begin{cases}
    S(x,\xi,\lambda)P(\lambda,\xi),&x\ge\xi,\\
    S(x,\xi,\lambda)(P(\lambda,\xi)-I),&x<\xi,
  \end{cases}\]
  and the solution $y_k(\lambda)$ of
  $F(\lambda)y_k(\lambda)=\widehat{v}_k$ is given by
  (see \cite[Thm.A.1]{BL})
  \[
  y_k(\lambda,x)=\int_\mathbb{R}
  G(x,\xi,\lambda)\widehat{v}_k(\xi)\, d\xi.\]
  There is $N_0\in \mathbb{N}$ with $J_N\supset \supp(\mu_j)$ for all
  $N\ge N_0$ and $j=1,\dots,m$. Using \eqref{defpq} and \eqref{eq:cpsuppwhat} 
we find for $N\ge N_0$
  \begin{align}
    &\bigl|\langle\widehat{w}_j,y_k(\lambda)\rangle-
    \langle\widehat{w}_j|_{J_N},y_k^N(\lambda)\rangle\bigr|\notag\\
    &\qquad 
    \le \Big|\langle \widehat{w}_j,y_k(\lambda)\rangle-
    \langle\widehat{w}_j|_{J_N},p_Ny_k(\lambda)\rangle\Big|+
    \Big|\langle\widehat{w}_j|_{J_N},
    p_Ny_k(\lambda)-y_k^N(\lambda)\rangle\Big|\notag\\
    &\qquad \le c\Bigl\| p_Ny_k(\lambda)-y_k^N(\lambda)
    \Bigr\|_{\mathcal{H}_N}.
    \label{eq:MainEstPt1}
  \end{align}
  Since $\widehat{v}_k$ has compact support and is bounded there is a constant $c$ such that 
  $|\widehat{v}_k(\xi)|\le c e^{-2\alpha|\xi|}$ for all
  $\xi\in\mathbb{R}$ and $k=1,\dots,\ell$. This is used to bound the right
  hand side of \eqref{eq:bvperror}:
  \begin{equation} \label{eq:estintern}
    \begin{aligned}
      |R_-(\lambda)y_k(\lambda,x_-^N)|
      &\le c \|R_-(\lambda)\| 
      \int_{-\infty}^{x_-^N}e^{-\alpha|x_-^N-\xi|}|\widehat{v}_k(\xi)|\,d\xi\\
      &\qquad+ c \|R_-(\lambda)(P(\lambda,x_-^N)-I)\|
      \int_{x_-^N}^{\infty}e^{-\alpha(\xi-x_-^N)}|\widehat{v}_k(\xi)|\,d\xi\\
      &\le c
      \int_{-\infty}^{x_-^N}e^{-\alpha|x_-^N-\xi|}e^{-2\alpha|\xi|}\,d\xi\\
      &\qquad+ c\|R_-(\lambda)(P(\lambda,x_-^N)-I)\|
      e^{\alpha x_-^N}\int_{x_-^N}^{\infty}e^{-\alpha\xi}e^{-2\alpha|\xi|}\,d\xi\\
      &\le c e^{2\alpha x_-^N} + c\|R_-(\lambda)(P(\lambda,x_-^N)-I)\|
      e^{\alpha x_-^N}.
    \end{aligned}
  \end{equation}
  A similar estimate holds for $|R_+(\lambda)y_k(\lambda,x_+^N)|$.
  Since the estimates are uniform in $\lambda\in\Gamma$, we obtain
  \begin{equation}
    \sup_{\lambda\in
    \Gamma}|R_-(\lambda)y_k(\lambda,x_-^N)+R_+(\lambda)y_k(\lambda,x_+^N)|\le
    ce^{-\alpha\min\{-x_-^N,x_+^N\}},
  \end{equation}
  by Theorem~\ref{thm:approxbybvp} this proves \eqref{eq:En-Eest}.

  If Hypothesis \ref{hyp3} holds, the projections $P(\lambda,x)$ of the
  exponential dichotomy on the whole real line can be chosen to satisfy
  \eqref{eq:ProjConv}, again see \cite[Thm.~A.5]{BL}. 
  For projection boundary conditions we then find from
  \eqref{eq:ProjConv} and \eqref{eq:BConRange}
  \[
  \| R_-(\lambda)(P(\lambda,x_-^N)-I)\|
  =\|R_-(\lambda)(P(\lambda,x_-^N)-P_-(\lambda))\| \le
  c e^{\alpha x^N_-}.
  \]
  Summarizing,  we can bound the right hand side of \eqref{eq:bvperror}
  as follows
  \[ \sup_{\lambda\in
  \Gamma}|R_-(\lambda)y_k(\lambda,x_-^N)+R_+(\lambda)y_k(\lambda,x_+^N)|
  \le c e^{-2\alpha \min\{-x_-^N,x_+^N\}},
  \]
  which gives the desired improved order of convergence.
\hfill\end{proof}

\begin{remark}
  It is not difficult to weaken the assumption of compact support for
  $\widehat{v}_k$ and $\widehat{w}_j$. For example, the proof in 
\eqref{eq:estintern} shows that it is sufficient to have
$|\widehat{v}_k(\xi)|\le c e^{-2\alpha|\xi|}$ for 
  $\xi\in\mathbb{R}$ and $k=1,\dots,\ell$. \hfill$\Diamond$
\end{remark}
\subsection{Estimates of eigenvalues}
We analyze the error of the numerical algorithm from
Section~\ref{s:simpleev} when the matrices $D_0$,
$D_1$ are replaced by their approximations $D_0^N$, $D_1^N$ satisfying
the estimates \eqref{eq:422}.

As before, let $D_0$ be of rank $\varkappa$ and let $D_0=V_0\Sigma_0W_0^*$,
be the short form of its singular value decomposition (SVD), cf.
 \eqref{singval}.
In the following we consider a small perturbation $\wti{D}_0\in \C^{m,\ell}$ of
$D_0$ with the full SVD
\begin{equation}\label{eq:SVD_Tilde}
  \wti{D}_0=\vect{\wti{V}_0&\wti{V}_1}
  \vect{\wti{\Sigma}_0&0_{\k,\ell-\k}\\0_{m-\k,\k}&\wti{\Sigma}_1}\vect{\wti{W}_0^*\\\wti{W}_1^*},
\end{equation}
where
$\wti{V}_0 \in \C^{m,\varkappa}$,$\wti{V}_1\in \C^{m,m-\varkappa}$, 
$\wti{W}_0 \in \C^{\ell,\varkappa}$, $\wti{W}_1\in \C^{\ell,\ell - \varkappa}$,
 and $\wti{\Sigma}_0\in \C^{\varkappa,\varkappa}$
contains the $\varkappa$ largest singular values of $\wti{D}_0$. 
Instead of computing the eigenvalues of
\beq\label{new2.18}
D=V_0^*D_1W_0\Sigma_0^{-1}\enq in \eqref{eqfL}, we use \eqref{eq:SVD_Tilde}
and compute the eigenvalues of
\begin{equation}\label{eq:D_Tilde}
  \wti{D}=\wti{V}_0^* \wti{D}_1\wti{W}_0 \wti{\Sigma}_0^{-1},
\end{equation}
where $\wti{D}_1$ is a small perturbation of $D_1$.
The following lemma shows that the eigenvalues of $\wti{D}$
approximate those of $D$ with the order of the original perturbations.
In order to apply the perturbation theory from \cite{St} we use the Frobenius norm 
$\|D\|_{F}^2 =\mathrm{tr}(D^{*}D)$,
the spectral norm $\|D\|_2$ 
and the Hausdorff distance 
\[ \dist_H(M_1,M_2)=\max(\sup_{z\in M_1} \inf_{y \in M_2}|z-y|,
\sup_{z\in M_2} \inf_{y \in M_1}|z-y| ), \quad M_1,M_2 \subset \C.
\]

\begin{lemma}\label{thm:error_evals} 
  Let $D_0, D_1 \in \mathbb{C}^{m,\ell}$ be given such that $\rank D_0=\k$
and such that the matrix $D$ from \eqref{new2.18} has only simple eigenvalues.
  Then there exist
  $\varepsilon_0>0$ and $C>0$ such that the spectra 
 of $D$  and $\wti{D}$ from \eqref{eq:D_Tilde}
satisfy
  \begin{equation}\label{eq:evalErr}
    \dist_H\left( \sigma(D),\sigma(\wti{D}) \right)\le C 
(\|D_0-\wti{D}_0\|_F+\|D_1-\wti{D}_1\|_F),
  \end{equation}
 provided  $ \|D_0-\wti{D}_0\|_F+\|D_1-\wti{D}_1\|_F \le\varepsilon_0 $.
\end{lemma}
\begin{proof}
  In the following $C$ denotes a generic constant depending
 on $\varepsilon_0$ but not on 
   $\wti{D}_0$, $\wti{D}_1$ and $\varepsilon:=\|D_0-\wti{D}_0\|_F+\|D_1-\wti{D}_1\|_F \le \varepsilon_0$.

  Let us extend  $Y_0=V_0$ and $X_0=W_0$ to unitary matrices  
$Y=\vect{Y_0&Y_1}\in \mathbb{C}^{m,m}$ and $X=\vect{X_0&X_1} \in
  \mathbb{C}^{\ell,\ell}$ and introduce $E\in
  \mathbb{C}^{m,\ell}$ such that
  \[Y^*D_0X=\vect{\Sigma_0&0_{\k,\ell-\k}\\0_{m-\k,\k}&0_{m-\k,\ell-\k}},\quad
  E=\vect{E_{11}&E_{12}\\E_{21}&E_{22}}:=Y^*(\wti{D}_0-D_0)X\]
  with conformal partitioning. We require 
$4\varepsilon_0< \sigma_{\min}=\min_{j=1,\ldots,\varkappa}\sigma_j$ and obtain 
  \[
  \begin{aligned}
    \gamma&:=\big({\|E_{12}\|_F^2+\|E_{21}\|_F^2}\big)^{1/2}\le \|E\|_F\le
    \varepsilon\le \varepsilon_0, \\
    2 \varepsilon &\le 2 \varepsilon_0 < \frac{\sigma_{\min}}{2} \le
\sigma_{\min} -\sqrt{2}\varepsilon_0
\le\sigma_{\min} -\|E_{11}\|_{2}-\|E_{22}\|_2=:\delta.
  \end{aligned}
  \]
   Therefore \cite[Thm.~6.4]{St} applies and yields
  $Q\in \mathbb{C}^{m-\k,\k},P\in \mathbb{C}^{l-\k,\k}$ with 
  \begin{equation}\label{eq:PQEst}
    \big({\|Q\|_F^2+\|P\|_F^2}\big)^{1/2}\le 2\frac{\gamma}{\delta}\le
    \frac{4}{\sigma_{\min}} \varepsilon,
  \end{equation}
    so that the unitary matrices
  \begin{equation} \label{eq:XYperturb}
  \begin{aligned}
  \wti{Y}&=\vect{\wti{Y}_0&\wti{Y}_1}
  = Y\vect{I&-Q^*\\Q&I} \vect{(I+Q^*Q)^{-{1}/{2}}&0_{\k,m-\k}\\
  0_{m-\k,\k}&(I+QQ^*)^{-{1}/{2}}}, \\
  \wti{X}&=\vect{\wti{X}_0&\wti{X}_1}
  = X\vect{I&-P^*\\P&I} \vect{(I+P^*P)^{-{1}/{2}}&0_{\k,\ell-\k}\\
  0_{\ell-\k,\k}&(I+PP^*)^{-{1}/{2}}}
  \end{aligned}
  \end{equation}
  transform $\wti{D}_0$ into block diagonal form
  \[ \wti{Y}^*\wti{D}_0\wti{X}=\vect{A_{11}&0_{\k,\ell-\k}\\0_{m-\k,\k}&A_{22}}\in\bbC^{m,\ell}.\]
  The matrices $A_{11}$ and $A_{22}$ can be written as
  \begin{equation}\label{eq:A11A22}
    \begin{aligned}
      A_{11}=\left( I+Q^*Q \right)^{{1}/{2}}\left(
      \Sigma_0+E_{11}+E_{12}P \right)\left( I+P^*P
      \right)^{-{1}/{2}},\\
      A_{22}=\left( I+QQ^* \right)^{{1}/{2}}\left(
      E_{22}-E_{21}P^* \right)\left( I+PP^*
      \right)^{-{1}/{2}}.
    \end{aligned}
  \end{equation}
  From the estimate \eqref{eq:PQEst} we find
  \[\left\| (I+Q^*Q)^{\pm{1}/{2}}-I\right\|_{2}\le C \|Q\|_{2}^2\le C
  \varepsilon^2\]
  and a similar estimate holds for $(I+QQ^*)^{\pm{1}/{2}}$. Using
\eqref{eq:XYperturb}, \eqref{eq:A11A22}  this leads to 
  \begin{equation}\label{eq:MatErrors}
      \|\wti{Y}_0-V_0\|_F\le C\varepsilon,\;
      \|\wti{X}_0-W_0\|_F\le C\varepsilon,\;
      \|\Sigma_0-A_{11}\|_2\le C\varepsilon,\;
      \|A_{22}\|_2\le C\varepsilon.
  \end{equation}
  By \cite[II~Cor.~2.3]{GK} and the last two inequalities in \eqref{eq:MatErrors} the sets of singular values
  $\sing A_{11}$ of $A_{11}$ and $\sing A_{22}$ of $A_{22}$ satisfy
  \[
  \min(\sing A_{11})\ge \sigma_{\min}-C\varepsilon,\quad
  \max(\sing A_{22})\le C\varepsilon.\]
  Decreasing $\varepsilon_0$ further we find that $\sing A_{11}$ and 
$\sing A_{22}$ are disjoint and hence $\sing A_{11}$ contains
 the $\k$ largest singular values of $\wti{D}_0$.

  Therefore, $\mathcal{R}(\wti{X}_0)$ is the invariant subspace of
  $\wti{D}_0^*\wti{D}_0$ corresponding to the $\k$ largest eigenvalues
  of $\wti{D}_0^*\wti{D}_0$ and coincides with
  $\mathcal{R}(\wti{W}_0)$. Similarly,
  $\mathcal{R}(\wti{Y}_0)=\mathcal{R}(\wti{V}_0)$.
  Then the matrices $T_0=\wti{X}_0^*\wti{W}_0\in\mathbb{C}^{\k,\k}$ and
  $S_0=\wti{Y}_0^*\wti{V}_0\in\mathbb{C}^{\k,\k}$ are
  unitary  and satisfy 
  $\wti{X}_0T_0=\wti{W}_0$ and $\wti{Y}_0S_0=\wti{V}_0$.
  Moreover, $\wti{\Sigma}_0=S_0^*\wti{Y}_0^*\wti{D}_0\wti{X}_0 T_0=
  S_0^*A_{11}T_0$, so that the matrices
\[\wti{D}=\wti{V}^*_0\wti{D}_1\wti{W}_0\wti{\Sigma}_0^{-1}=
  S_0^*\wti{Y}_0^*\wti{D}_1\wti{X}_0 T_0T_0^*A_{11}^{-1}S_0
=S_0^*\wti{Y}_0^*\wti{D}_1\wti{X}_0 A_{11}^{-1}S_0\]
and
  \[\widehat{D}=\wti{Y}_0^*\wti{D}_1\wti{X}_0 A_{11}^{-1}\]
  are similar and have the same spectrum.
    Now estimates  \eqref{eq:MatErrors} imply
  \[\|\widehat{D}-D\|_F\le C\varepsilon.\]
  Since simple eigenvalues depend analytically on the matrix, we
  finally obtain for some $C>0$
  \begin{equation} \label{eq:comparespectra}
\dist_H\left( \sigma(D),\sigma(\wti{D}) \right)=
  \dist_H\left( \sigma(D),\sigma(\widehat{D}) \right)
  \le C \varepsilon.
  \end{equation}
  This finishes the proof. \hfill
\end{proof}
Combining this result with Corollary \ref{cor:4.8}  shows that using
boundary value problems for the computation of the point spectrum is a
robust method and leads to exponential convergence with respect to
the length of intervals.
\begin{theorem}\label{thm:EvalsConv}
  Let the assumptions of Theorem~\ref{thm:EConv} hold and let all eigenvalues of $F$
inside the contour $\Gamma$ be simple. Then there is $C>0$
such that the following holds for all intervals $[x_-^N,x_+^N]$ with
  $\min\{-x_-^N,x_+^N\}$ sufficiently large.
 The set $\sigma^N$ of the eigenvalues of the approximate pencil $F_N$ from \eqref{eq:inhomapprox}, computed by the method
  from Section~\ref{s:simpleev} using the approximations $D_0^N$, $D_1^N$
from \eqref{defBn01} instead of
  $D_0, D_1$, satisfies  the estimate
  \begin{equation}\label{eq:MainEvalError}
    \dist_H\left( \sigma^N,\sigma(D) \right)\le C
    e^{-\alpha\min\{-x_-^N,x_+^N\}}.
  \end{equation}
  In case of projection boundary conditions the constant $\alpha$
  improves to $2\alpha$.
\end{theorem}
\begin{remark}
We note that the simplicity of eigenvalues was only used in the very
last step \eqref{eq:comparespectra} of the proof of Lemma \ref{thm:error_evals}. Similar to Remark \ref{rem:jordanrobust} convergence
of spectra as $\varepsilon\rightarrow0$ still follows in the general case
from the perturbation theory in \cite{K80}. But now the rate is $\varepsilon^{\frac{1}{\mu}}$ in
\eqref{eq:comparespectra} where $\mu$ is the maximal algebraic multiplicity
of eigenvalues inside $\Gamma$. Correspondingly, the rate $\alpha$ in
\eqref{eq:MainEvalError} deteriorates to $\frac{\alpha}{\mu}$.
\end{remark}

%

\section{The Schr\"odinger operator on the line}
\label{s:appl}
We consider the eigenvalue problem for the one dimensional  Schr\"odinger operator, $H$,
\beq\lb{spSch} (H-\l)u=0,\; H=-d^2/dx^2+V(x),\;   x\in\bbR.
\enq
Here, the real valued potential $V$ satisfies $V\in L^1(\bbR)$,  the domain of $H$ is given by
\[\dom(H)=\big\{u\in L^2(\bbR): u, u' \in AC_{\loc}(\bbR),\;
-u''+Vu\in L^2(\bbR)\big\},\] 
and we assume that $\l\in\Omega=\bbC\setminus[0,\infty)$. Since $V\in L^1(\bbR)$, the essential spectrum of $H$ is equal to $[0,\infty)$, and the discrete spectrum consists of no
more than finitely many negative simple eigenvalues $0>\l_1>\l_2>\dots>\l_\k$, see, e.g.,
\cite[Sec.XIII.3]{RS78}, \cite[Sec.XVII.1.3]{CS}. The eigenvalue problem 
\eqref{spSch} can be written as the first order differential equation
\begin{equation}\label{foSch}
y'=A(\l,x)y, \,  A(\l,x)=A(\l,\infty)+B(x),\, x\in\bbR;
\end{equation}
here and below we denote
\begin{equation}\label{foSch2}
 A(\l,\infty)=\begin{pmatrix}0&1\\-\l&0\end{pmatrix},\,
B(x)=\begin{pmatrix}0&0\\V(x)&0\end{pmatrix},\, y(x)=\begin{pmatrix}u(x)\\u'(x)\end{pmatrix}.
\end{equation}
In Subsection \ref{ss:sodo} we consider the linear operator pencil of second order differential operators
$F^{(\text{II})}(\l)=H-\l I$, see \eqref{spSch}, acting from the space
$\cH=\dom(H)$ equipped with the graph norm into the space
$\cK=L^2(\bbR,\C^2)$. In Subsection \ref{ss:fodo} we consider the
nonlinear operator pencil of first order differential operators
$F^{(\text{I})}(\l)=-\partial_x+A(\l,\cdot)$, see \eqref{foSch}, \eqref{foSch2}, acting from the space $\cH$ as defined in \eqref{domainF} into the space $\cK=L^2(\bbR,\C^2)$.
Our objective is to illustrate the construction of the matrix $E$ from \eqref{defE}, \eqref{repEODE}, \eqref{cKf} and also its computation via some approximation arguments related to the boundary value problems on finite intervals, that is, to the equation
\beq\lb{BVPsch}-u''(x)+V(x)u(x)-\l u(x)=0,\; x\in[x_-^{N},x_+^{N}],\enq
equipped with appropriate boundary conditions at the endpoints $x_-^{N}$, $x_+^{N}$
satisfying $x_-^{N}\to-\infty$ and $x_+^{N}\to+\infty$ as $N\to\infty$,
and to the boundary value problems for the first order differential
equation \eqref{foSch}. 
In the current section we do not assume exponential decay of the
perturbation. Although we offer some explicit formulas for the matrix
$E$ and its approximation $E^N$ in terms of certain solutions of the
differential equations 
\eqref{spSch}, \eqref{foSch}, \eqref{foSch2}, we emphasize that they are mainly of theoretical value as our general approach in practical applications is {\em not} to use these formulas but instead  to construct $E^N$ by solving boundary value problems on finite intervals numerically.

\subsection{Second order differential operators}\label{ss:sodo}
We consider the linear operator pencil 
$F^{(\text{II})}(\l)=H-\l I$ with $H$ as in \eqref{spSch}. Our main tool will be the Jost solutions $u_\pm(\l,x)$, $x\in\R$, $\l\in\Omega=\C\setminus[0,\infty)$, of the second order Schr\"odinger differential equation
\eqref{spSch} which are uniquely determined as the solutions of the  Volterra integral equations
\begin{equation}\label{volt}u_\pm(\l,x)=e^{\pm\il x}-\int_0^{\pm\infty}\l^{-1/2}\sin(\l^{1/2}(x-\xi))V(\xi)u_\pm(\l,\xi)\,d\xi, \, x\in\R.
\end{equation} 
Here and everywhere below we choose the branch of the square root such that $\Im(\l^{1/2})>0$ for $\l\in\Omega$, in particular, $e^{\il x}\to0$ as $x\to+\infty$. It is well known that the Jost solutions  
satisfy the asymptotic boundary conditions
\begin{equation}\label{Jsas}
\lim_{x\to\pm\infty}e^{\mp\il x}u_\pm(\l,x)=1,
\end{equation}
they are holomorphic functions of $\l\in\Omega$,  for $\l<0$ they are real valued and positive for $\pm x$ sufficiently large, 
see, e.g., \cite[Chap.XVII]{CS}. The Wronskian 
\begin{equation}\label{defW}
\cW(\l)=\cW(u_-,u_+)=u_-(\l,x)u'_+(\l,x)-u'_-(\l,x)u_+(\l,x),\, x\in\R, \l\in\Omega, \end{equation} of the Jost solutions is equal to zero precisely at the points $\l_n\in\Omega$, the isolated eigenvalues of the Schr\"odinger operator $H$, where the exponentially decaying at $+\infty$ solution $u_+(\l_n,\cdot)$ is proportional to the exponentially decaying at $-\infty$ solution $u_-(\l_n,\cdot)$ with a nonzero constant $c_n$, that is, when
\begin{equation}\label{propJS}
u_+(\l_n,x)=c_nu_-(\l_n,x),\, x\in\R,\, c_n\in\C\setminus\{0\},\, n=1,\dots,\k.
\end{equation}

We refer to \cite{W87} for the general theory of Sturm-Liouville
differential  operators (see also \cite{W05} for a brief but
exceptionally readable account). 
In particular, due to $V\in L^1(\R)$
the Schr\"odinger operator $H$ is in the limit point case at
$\pm\infty$, the Jost solutions $u_\pm$ are $L^2$-solutions at
$\pm\infty$, and thus the resolvent operator
$\big(F^{(\text{II})}(\l)\big)^{-1}=(H-\l I)^{-1}$ for
$\l\in\Omega\setminus\{\l_1,\dots,\l_\k\}$ is the integral operator 
with the kernel 
\begin{equation}
R(\l,x,\xi)=\frac{1}{\cW(u_+,u_-)}\begin{cases}u_+(\l,x)u_-(\l,\xi), &-\infty<\xi\le x<+\infty,\\
u_-(\l,x)u_+(\l,\xi), &-\infty<x< \xi<+\infty.\end{cases}\end{equation}
Therefore, if $\widehat{w}_j,\widehat{v}_k\in L^2(\R)$ are chosen as indicated in Section \ref{s:abst}, that is, such that
\begin{equation*}
\rank\big(\langle \widehat{w}_j(\cdot), u_\pm(\l_n,\cdot)\rangle_{\R}\big)_{j,n=1}^{m,\k}\ge\k, \rank\big(\langle u_\pm(\l_n,\cdot), \widehat{v}_k(\cdot)\rangle_{\R}\big)_{n,k=1}^{\k,\ell}\ge\k,
\end{equation*}
then the matrix $E(\l)=\big(\langle\widehat{w}_j(\cdot), \big(\big(F^{(\text{II})}(\l)\big)^{-1}\widehat{v}_k\big)(\cdot)\rangle_\R\big)_{j,k=1}^{m,\ell}$ from \eqref{defE} is given by 
\begin{equation}\label{Ejk}
\begin{split}E_{jk}(\l)&=\frac{1}{\cW(u_+,u_-)}\int_{-\infty}^\infty \widehat{w}_j(x)u_+(\l,x)\int_{-\infty}^xu_-(\l,\xi)\widehat{v}_k(\xi)\,d\xi\,dx\\
&+\frac{1}{\cW(u_+,u_-)}\int_{-\infty}^\infty \widehat{w}_j(x)u_-(\l,x)\int^{\infty}_xu_+(\l,\xi)\widehat{v}_k(\xi)\,d\xi\,dx.
\end{split}\end{equation}

Integrating $E(\l)$ from \eqref{Ejk} over the contour $\Gamma$ from Section \ref{s:simpleev}
we thus obtain $\k=\rank D_0$ and formulas \eqref{eqfL}, \eqref{lkf}  for the eigenvalues of $H$.

We now equip equation \eqref{BVPsch}  with self-adjoint boundary conditions
\begin{equation}\label{Rbc}
u(x^N_\pm)\cos\omega_\pm-u'(x_\pm^N)\sin\omega_\pm=0, \, \text{ with some $\omega_\pm\in[0,\pi)$},
\end{equation}
and define the operator $H^N$ in $L^2([x_-^N,x_+^N])$ by
$H^N=-d^2/dx^2+V(x)$ with
\begin{equation*}
\begin{split}\dom(H^N)=\big\{u\in & L^2([x_-^N,x_+^N]):  u,  u' \in AC_{\loc}([x_-^N,x_+^N]),\; 
-u''+Vu\in L^2([x_-^N,x_+^N])\\
&\text{ and both boundary conditions \eqref{Rbc} hold}\big\},\end{split}\end{equation*}
cf.\ \cite[Sec.7]{W05}. Let $\widetilde{u}_\pm(\l,x)$ denote the $N$-dependent solutions of the Schr\"odinger equation \eqref{spSch} each of them satisfying one of the respective initial conditions
\begin{equation}
\widetilde{u}_\pm(\l,x_\pm^N)=e^{\il x_\pm^N}\sin\omega_\pm,\,
\widetilde{u}'_\pm(\l,x_\pm^N)=e^{\il x_\pm^N}\cos\omega_\pm.
\end{equation}
Since  $\widetilde{u}_+(\l,x)$, respectively, $\widetilde{u}_-(\l,x)$
satisfies the boundary condition \eqref{Rbc} at $x_+^N$, respectively, $x_-^N$ we conclude (see, e.g., \cite[p.84]{W05}) that the resolvent operator  $\big(F^{(\text{II}, N)}(\l)\big)^{-1}=(H^N-\l I)^{-1}$ is the integral operator 
with the kernel 
\begin{equation}\label{defRn}
R^N(\l,x,\xi)=\frac{1}{\cW(\widetilde{u}_+,\widetilde{u}_-)}\begin{cases}\widetilde{u}_+(\l,x)\widetilde{u}_-(\l,\xi), &x_-^N\le\xi\le x\le x_+^N,\\
\widetilde{u}_-(\l,x)\widetilde{u}_+(\l,\xi), &x_-^N\le x< \xi\le x_+^N.\end{cases}\end{equation}
In the following we consider the restriction $p_N u =u|_{[x_-^N,x_+^N]}$
(cf. \eqref{defpq}) 
as an operator from  $L^2(\R)$ into $L^2([x_-^N,x_+^N])$. Then it is known from 
\cite[Thm. 7.1]{W05} that $(F^{(II,N)}(\l))^{-1} p_N$ converges strongly
to $(F^{(II)}(\l))^{-1}$ in $L^2(\R)$. This is called generalized
strong resolvent convergence of $H^N$ to $H$ in \cite{W05} (here, one imbeds $L^2([x_-^N,x_+^N])$ into  $L^2(\R)$ by setting functions equal to zero in $\bbR\setminus[x_-^N,x_+^N]$). It follows that
the matrix $E(\l)$ from \eqref{defE} for the operator pencil
$F^{(\text{II})}(\l)$ can be written as the limit of the matrices
$E^N(\l)$ defined via the approximative operator pencils $F^{(\text{II},
N)}(\l)$.
\begin{proposition}
Assume $V\in L^1(\R)$ and let $E(\l)$ be defined as in \eqref{Ejk}. Then \[ E(\l)=\lim_{N\to\infty}E^N(\l) \text{ where }
E^N(\l)=\Big(\int_{x_-^N}^{x_+^N}\widehat{w}_j^N(x) \big((F^{(\text{II},
N)}(\l))^{-1}\widehat{v}_k^N\big)(x)\,dx\Big)_{j,k=1}^{m,\ell},\]
and we denote $\widehat{w}_j^N=p_N \widehat{w}_j$,  $\widehat{v}_k^N=p_N \widehat{v}_k$.
Similarly to \eqref{Ejk}, using \eqref{defRn} the matrix $E^N(\l)$  can be computed  by the formula
\begin{equation}\label{Enjk}
\begin{split}E_{jk}^N(\l)&=\frac{1}{\cW(\widetilde{u}_+,\widetilde{u}_-)}\int_{x_-^N}^{x_+^N} \widehat{w}_j^N(x)\widetilde{u}_+(\l,x)\int_{x_-^N}^x\widetilde{u}_-(\l,\xi)\widehat{v}_k^N(\xi)\,d\xi\,dx\\
&+\frac{1}{\cW(\widetilde{u}_+,\widetilde{u}_-)}\int_{x_-^N}^{x_+^N} \widehat{w}_j^N(x)\widetilde{u}_-(\l,x)\int^{x_+^N}_x\widetilde{u}_+(\l,\xi)\widehat{v}_k^N(\xi)\,d\xi\,dx.
\end{split}\end{equation}
\end{proposition}
We note in passing that $\lim_{N\to\infty}\cW(\widetilde{u}_+,\widetilde{u}_-)=C(\l)\cW({u}_+,{u}_-)$ where the factor  $C(\l)$ can be explicitly computed and is equal to zero precisely at the eigenvalues of the operators $H^0_+$ and $H_-^0$ defined as $H^0_\pm=-d^2/dx^2$ on $L^2((-\infty,x_+^N])$ and $L^2([x_-^N,\infty))$ with the domain determined by the respective boundary condition in \eqref{Rbc}, see \cite[Thm.3.3]{LS}.
\subsection{First order differential operators}\label{ss:fodo}
W ith $A(\l,x)$ as in \eqref{foSch},\eqref{foSch2} we consider the 
operator pencil $F^{(\text{I})}(\l)y=-y'+A(\l,x)y$ 
  and follow step-by-step the constructions in Subsections \ref{ss:defev} 
and \ref{sec3.2} culminating in formulas \eqref{repEODE} and \eqref{cKf}.

First, we need to choose the projections on the subspaces of the initial
values of the solutions of \eqref{foSch} exponentially decaying at $+\infty$ and $-\infty$ and construct their representation \eqref{reploc} and the respective Evans function \eqref{defevans1}.  The matrix $A(\l,\infty)$ for $\l\in\Omega$ has no pure imaginary eigenvalues and the differential equation 
$y'=A(\l,\infty)y$ has the exponential dichotomy on $\R$ with the dichotomy projection being the spectral projection $P(\l,\infty):=P_+(\l)=P_-(\l)$ for   $A(\l,\infty)$ corresponding to the eigenvalue $\il$. We recall that if $\l\in\Omega$ then the eigenvalues $\pm\il$ of $A(\l,\infty)$ satisfy $\Re(\il)<0<\Re(-\il)$. 

Since the perturbation $B$ in \eqref{foSch} satisfies  $\|B(\cdot)\|\in L^1(\R)$ by the general theory in e.g. \cite{BL,Cop,GLM07} there exist dichotomy projections 
\begin{equation}\label{PPM}
P_\pm(\l,x)=S(x,0,\l)P_\pm(\l,0)S(0,x,\l)
\end{equation} on $\R_\pm$ for the perturbed equation $y'=A(\l,x)y$ such
that the dichotomy subspaces $\mathcal{R}(P_+(\l,x))=\Span\{y_+(\l,x)\}$ and
$\mathcal{N}(P_-(\l,x))=\Span\{y_-(\l,x)\}$ are uniquely determined
while their direct complements $\mathcal{N}(P_+(\l,x))$ and
$\mathcal{R}(P_-(\l,x))$ are arbitrary. Here and below we use notation 
\beq\label{defypmx}
y_\pm(\l,x)=\begin{pmatrix}u_\pm(\l,x)\\u'_\pm(\l,x)\end{pmatrix},\quad x\in\bbR,\quad\l\in\Omega, \enq for the $(2\times 1)$ vector solutions of \eqref{foSch},\eqref{foSch2} which correspond to the Jost solutions $u_\pm(\l,x)$ of \eqref{spSch} defined via the Volterra equations \eqref{volt}. Also, given a vector $v=\begin{pmatrix}v_1\\v_2\end{pmatrix}\in\C^2$ we denote $v^\dag=\begin{pmatrix}-v_2\\v_1\end{pmatrix}$ and $v^\bot=(v^\dag)^\top=\begin{pmatrix} -v_2 & v_1\end{pmatrix}$ so that for any two vectors $v,w\in\C^2$ we have $\det(v\big| w)=v^\bot w=-v^\top w^\dag$; thus, if $v,w\in\C^2$ are linearly independent then the projection on $\Span\{v\}$ parallel to $\Span\{w\}$ is the matrix $(w^\bot v)^{-1}vw^\bot$.

To make a choice of the dichotomy projections $P_\pm(\l,x)$ one is tempted to let 
\begin{equation}\label{wDPP}\begin{split}
P_+(\l,0)&=(\cW(u_-,u_+))^{-1}y_+(\l,0)y_-(\l,0)^\bot,\\
I-P_-(\l,0)&=(\cW(u_+,u_-))^{-1}y_-(\l,0)y_+(\l,0)^\bot\end{split}
\end{equation} and then use \eqref{PPM} to define $P_\pm(\l,x)$ for $x\in\R_\pm$. This choice, however, is not satisfactory as these projections are meromorphic in $\Omega$ with the poles precisely at the eigenvalues $\l_n$ while the constructions in Subsection \ref{s:evfun} require holomorphy. Another choice is to normalize $u_\pm$ by letting $\widehat{u}_\pm(\l,x)=(u^2_\pm(\l,0)+u'^2_\pm(\l,0))^{-1/2}u_\pm(\l,x)$ and then replace $u_\pm$ and $y_\pm$ in \eqref{wDPP} by $\widehat{u}_\pm$ and $\widehat{y}_\pm$. This choice yields holomorphy of $P_\pm(\l,x)$ in $\Omega\setminus\Lambda$, where $\Lambda=\{\l: (u^2_+(\l,0)+u'^2_+(\l,0))(u^2_-(\l,0)+u'^2_-(\l,0))=0\}$, and admits the normalization of the type \eqref{normalrep2}, \eqref{holrep2} leading to the construction of the normalized Evans function $\cE_0$. However, the great disadvantage of this choice of $P_\pm(\l,x)$ is that these dichotomy projections are not asymptotic as $x\to\pm\infty$ to the spectral projections of $A(\l,\infty)$. 

We now construct the holomorphic dichotomy projections that are asymptotic to the spectral projections at infinity.
Our main tool will be the solutions $y_+^s(\l,\cdot)$,  $y_+^u(\l,\cdot)$ on $\R_+$ and $y_-^s(\l,\cdot), y_-^u(\l,\cdot)$ on $\R_-$  of the differential equation \eqref{foSch} satisfying the asymptotic boundary conditions
\begin{align}
\lim_{x\to+\infty}e^{-\il x}y_+^s(\l,x)&=\bv, \,
\lim_{x\to+\infty}e^{\il x}y_+^u(\l,x)=\bw,\label{LS1}\\
\lim_{x\to-\infty}e^{\il x}y_-^s(\l,x)&=\bw, \,
\lim_{x\to-\infty}e^{-\il x}y_-^u(\l,x)=\bv,\label{LS2}
\end{align}
where $\bv=\begin{pmatrix}1\\ \il\end{pmatrix}$, $\bw=\begin{pmatrix}1\\ -\il\end{pmatrix}$ are the eigenvectors of  $A(\l,\infty)=\begin{pmatrix}0&1\\-\l&0\end{pmatrix}$ such that $A(\l,\infty)\bv=\il\bv$, $A(\l,\infty)\bw=-\il\bw$.  The  spectral projections of the matrix $A(\l,\infty)$ are given by the formulas $P(\l,\infty)=(2\il)^{-1}\bv\bw^\bot$ and $I-P(\l,\infty)=-(2\il)^{-1}\bw\bv^\bot$.

Since $\|B(\cdot)\|\in L^1(\R)$, the existence of the solutions $y_+^{s,u}(\l,\cdot)$ on $\bbR_+$ satisfying \eqref{LS1} and the solutions $y_-^{s,u}(\l,\cdot)$ on $\bbR_-$ satisfying \eqref{LS2} is guaranteed by the celebrated Levinson theorem from asymptotic theory of differential equations (see, e.g., \cite[Probl.III.29]{CoL} or \cite[Thms.1.3.1,1.8.1]{E} and also \cite[Rem.7.11]{GLM07}). These solutions are obtained by solving certain inhomogeneous Fredholm type integral equations on semilines with holomorphic integral kernels and inhomogeneities.  In fact, inspecting the proof of the theorem, cf.\ \cite[Sec.1.4]{E} or \cite[Thm.8.3]{GLM07}, we observe that the solutions  $y_\pm^{s,u}(\l,\cdot)$ are holomorphic in $\l\in\Omega$ and limiting relations \eqref{LS1}, \eqref{LS2} hold uniformly in $\l$ on compact subsets of $\Omega$. 

We will now discuss how the solutions $y_+^{s,u}(\l,\cdot)$, $y_-^{s,u}(\l,\cdot)$ satisfying \eqref{LS1}, \eqref{LS2} are related to the solutions $y_\pm(\l,\cdot)$ defined in \eqref{defypmx} via the Jost solutions $u_\pm(\l,\cdot)$ of \eqref{volt}. In fact, the solution $y_+^s(\l,\cdot)$, respectively, $y_-^s(\l,\cdot)$ is uniquely determined by the first limiting relation in \eqref{LS1}, respectively, \eqref{LS2} since due to \eqref{Jsas} it is precisely the solution $y_+(\l,\cdot)$, respectively, $y_-(\l,\cdot)$ defined in \eqref{defypmx}: 
\begin{align}
y_\pm^s(\l,x)&=y_\pm(\l,x),\quad x\in\bbR_\pm, \quad\l\in\Omega.\label{LSs}\end{align}
The solution $y_+^u(\l,\cdot)$, respectively,  $y_-^u(\l,\cdot)$ is not unique and can be changed by adding  a summand proportional to $y_+(\l,\cdot)$, respectively, $y_-(\l,\cdot)$. A convenient choice of the solutions $y_\pm^u(\l,\cdot)$ is furnished by the formulas
\begin{equation}\begin{split}
y_+^u(\l,x)&=\frac{2\il}{\cW(u_-,u_+)}y_-(\l,x)-\sum_{n=1}^\k\frac{2\il \rho_n}{c_n(\l-\l_n)}y_+(\l,x),\quad x\in\bbR_+,\\
y_-^u(\l,x)&=\frac{2\il}{\cW(u_-,u_+)}y_+(\l,x)-\sum_{n=1}^\k\frac{2\il \rho_n c_n}{(\l-\l_n)} y_-(\l,x),\quad x\in\bbR_-,
\end{split}\label{LSu}
\end{equation}
where $y_\pm(\l,\cdot)$ are defined in \eqref{defypmx}, $\l_n$, $n=1,\dots,\k$, are the eigenvalues of $H$, the constants $c_n$ are taken from \eqref{propJS}, and we denote by $\rho_n$ the residue at $\l_n$ of the function $1/\cW(\l)$ for the Wronskian $\cW(\l)$ defined in  \eqref{defW}. As we will see in a moment, the solutions $y_\pm^{s}(\l,\cdot)$,
$y_\pm^{u}(\l,\cdot)$  defined in \eqref{LSs},\eqref{LSu} satisfy \eqref{LS1}, \eqref{LS2}. Using these solutions at $x=0$ we let 
\begin{equation}\label{rDPP}\begin{split}
P_+(\l,0)&=(2\il)^{-1}y_+^s(\l,0)y_+^u(\l,0)^\bot,\\
I-P_-(\l,0)&=-(2\il)^{-1}y_-^s(\l,0)y_-^u(\l,0)^\bot\end{split}
\end{equation} and then use \eqref{PPM} to define $P_\pm(\l,x)$ for $x\in\R_\pm$. By a direct calculation we also have $P_-(\l,0)=(2\il)^{-1}y_-^u(\l,0)y_-^s(\l,0)^\bot$.
\begin{lemma}\label{hPr} Assume that $V\in L^1(\R)$ and $\l\in\Omega=\C\setminus[0,\infty)$. Then  the projections 
\begin{align}\label{fDPP1}
P_+(\l,x)&=(2\il)^{-1}y_+^s(\l,x)y_+^u(\l,x)^\bot, \quad x\in\R_+,\\
P_-(\l,x)&=(2\il)^{-1}y_-^u(\l,x)y_-^s(\l,x)^\bot, \quad x\in\R_-,\label{fDPP2}
\end{align} defined via formulas \eqref{PPM}, \eqref{rDPP} are holomorphic in $\Omega$ and satisfy 
\begin{equation}\label{asF}
\lim_{x\to\pm\infty}P_\pm(\l,x)=P(\l,\infty).\end{equation} 
\end{lemma}
\begin{proof}
Formulas \eqref{LS1}, \eqref{LS2} for the solutions $y_\pm^{s}(\l,\cdot)$,
$y_\pm^{u}(\l,\cdot)$ defined in \eqref{LSs}, \eqref{LSu} follow from
\eqref{Jsas} and  the relations $\lim_{x\to\pm\infty}e^{\mp\il x}u'_\pm(\l,x)=\pm\il$ and 
 \[\lim_{x\to\pm\infty}e^{\pm\il x}u_\mp(\l,x)=\cW(\l)/(2\il),
\, \lim_{x\to\pm\infty}e^{\pm\il x}u'_\mp(\l,x)=\mp\cW(\l)/2,\]
(see, e.g., \cite[Lem.3.1]{LS}).
Also, computing the residues of the RHS of \eqref{LSu} and using \eqref{propJS} we see that the solutions $y^s_\pm(\l,\cdot)$, $y^u_\pm(\l,\cdot)$ defined via \eqref{LSs}, \eqref{LSu} are holomorphic in $\l\in\Omega$. In the remaining part of the proof we concentrate on the case of $\R_+$ as the arguments for $\R_-$ are similar. The Wronskian of the solutions $y_+^s(\l,\cdot)$ and $y_+^u(\l,\cdot)$ is $x$-independent and by \eqref{LS1}, \eqref{LS2} we infer that
\begin{align*}
y^u_+(\l,x)^\bot y_+^s(\l,x)=(e^{\il x}y^u_+(\l,x))^\bot (e^{-\il x}y_+^s(\l,x))\to \bw^\bot\bv=2\il
\end{align*}
 as $x\to+\infty$. Therefore, 
 \begin{equation}\label{xcalc}
 y^u_+(\l,x)^\bot y_+^s(\l,x)=2\il\quad\text{ for all $x\in\R_+$ }
 \end{equation} 
 and thus \eqref{fDPP1} is a holomorphic projection. But $y_+^s(\l,x)=S(x,0,\l)y_+^s(\l,0)$ by the definition of $S(x,0,\l)$ and $y_+^u(\l,x)^\bot = y_+^u(\l,0)^\bot S(x,0,\l)^{-1}$ since  $w(x)=y_+^u(\l,0)^\bot S(x,0,\l)^{-1}$ satisfies the  adjoint equation $w'=-wA(\l,x)^\top$ and therefore should be of the form $w=y^\bot$ for a solution $y$ of the equation $y'=A(\l,x)y$. Thus, \eqref{fDPP1} is in concert with \eqref{PPM} and \eqref{rDPP}. Using \eqref{LS1}, \eqref{LS2} again we have
 \[y_+^s(\l,x)y^u_+(\l,x)^\bot = (e^{-\il x}y_+^s(\l,x))(e^{\il x}y^u_+(\l,x))^\bot\to \bv\bw^\bot\]
 as $x\to+\infty$ yielding \eqref{asF}. 
\hfill\end{proof}

We are ready to identify the ingredients in the representation \eqref{reploc} of the projections $\Pi_U(\l)=P_+(\l,0)$ and $\Pi_V(\l)=I-P_-(\l,0)$ on the holomorphic families of one dimensional subspaces $U(\l)=\Span\{y_+(\l,0)\}$ and $V(\l)=\Span\{y_-(\l,0)\}$. Indeed, \eqref{LSs} and formulas \eqref{fDPP1}, \eqref{fDPP2} show that  \eqref{reploc} holds with
\begin{align}
P(\l)&=y_+^s(\l,0),\, \Phi(\l)=(2\il)^{-1}(y_+^u(\l,0))^\dag,\\
Q(\l)&=-y_-^s(\l,0),\, \Psi(\l)=(2\il)^{-1}(y_-^u(\l,0))^\dag,
\end{align}
where the normalization \eqref{normloc1} has been shown in the proof of Lemma \ref{hPr}, see \eqref{xcalc}. Therefore, due to \eqref{LSs} the Evans function \eqref{defevans1} is given by
\begin{equation}\label{defEV}
\cE(\l)=\det\big(y_+(\l,0)\big|y_-(\l,0)\big)=\cW(u_-,u_+),
\end{equation}
while the matrices $\cY_U(\l),\cY_V(\l)$ in Theorem \ref{projecthol} are given by
\begin{equation}\label{defYUV}
\cY_U(\l)=y_+(\l,0)y_-(\l,0)^\bot,\, \cY_V(\l)=y_-(\l,0)y_+(\l,0)^\bot.
\end{equation}
Furthermore, the Green's kernels \eqref{greenkernels} of the Green's operators  \eqref{greenform} are given as follows:
\begin{align*}
G_+(x,\xi,\l)&=\frac{1}{2\il}\begin{cases}y_+^s(\l,x)y_+^u(\l,\xi)^\bot,&0\le\xi\le x,\\y_+^u(\l,x)y_+^s(\l,\xi)^\bot,&0\le x<\xi,\end{cases}\\
G_-(x,\xi,\l)&=\frac{1}{2\il}\begin{cases}y_-^u(\l,x)y_-^s(\l,\xi)^\bot,&\xi\le x\le0,\\y_-^s(\l,x)y_-^u(\l,\xi)^\bot,&x<\xi\le0,\end{cases}
\end{align*}
while the function \eqref{jumpop} and the vector $[\widehat{v}]_0=(\cG_-(\l)\widehat{v})(0-)-(\cG_+(\l)\widehat{v})(0+)$ from Theorem \ref{formulaE} are computed as follows (recall that
$y_\pm^s(\l,\cdot)=y_\pm(\l,\cdot)$ by \eqref{LSs}):
\begin{align*}
G(\l,x)&=\begin{cases}y_+(\l,x)y_-(\l,0)^\bot,&x\ge0,\\
y_-(\l,x)y_+(\l,0)^\bot,&x<0,\end{cases}\\
[\widehat{v}]_0&=\frac{1}{2\il}\Big(\int_{-\infty}^0y_-^u(\l,0)y_-(\l,\xi)^\bot\widehat{v}(\xi)\,d\xi-\int_0^\infty y_+^u(\l,0)y_+(\l,\xi)^\bot\widehat{v}(\xi)\,d\xi\Big).
\end{align*}
To give a compact formula for the singular part of the RHS of \eqref{repEODE} it is convenient to introduce, for given functions $\widehat{w},\widehat{v}\in L^2(\R,\C^2)$, the following $(2\times 1)$, $(2\times 1)$ and $(2\times 2)$ matrices:
\begin{align}
\widehat{W}(\l)&=\begin{pmatrix}\int_0^\infty y_+(\l,\xi)^\top\widehat{w}(\xi)\,d\xi \\\\
\int^0_{-\infty} y_-(\l,\xi)^\top\widehat{w}(\xi)\,d\xi\end{pmatrix},\,
\widehat{V}(\l)=\begin{pmatrix}\int_0^\infty y_+(\l,\xi)^\bot\widehat{v}(\xi)\,d\xi\\ \\
\int^0_{-\infty} y_-(\l,\xi)^\bot\widehat{v}(\xi)\,d\xi\end{pmatrix},\label{DEFwv}\\
\widehat{\cY}(\l)&=\frac{1}{2\il}\begin{pmatrix}-y_-(\l,0)^\bot y_+^u(\l,0)&
y_-(\l,0)^\bot y_-^u(\l,0)\\ -y_+(\l,0)^\bot y_+^u(\l,0)&
y_+(\l,0)^\bot y_-^u(\l,0) \end{pmatrix}.\label{DEFcy}
\end{align}
Plugging  \eqref{LSu} in \eqref{DEFcy}, a short calculation reveals the 
following concretization of 
formulas \eqref{repEODE} and  \eqref{cKf} for the case of the Schr\"odinger operator.
\begin{theorem}
\label{thm:singformjost}
Assume $V\in L^1(\R)$ and use
linearly independent functions $\widehat{w}_j,\widehat{v}_k\in L^2(\R,\C^2)$, $j=1,\dots,m$, $k=1,\dots,\ell$, in \eqref{DEFwv}. Then the singular part of the RHS of \eqref{repEODE} can be expressed as follows:
\begin{equation}\label{newEs}
\frac{1}{\cE(\l)}\langle\widehat{w}_j, G(\l,\cdot)[\widehat{v}_k]_0\rangle_\R=\widehat{W}_j(\l)^\top\,\widehat{\cY}(\l)\,\widehat{V}_k(\l), \,\l\in\Omega.
\end{equation} 
As in \eqref{cKf}, using formula \eqref{newEs} the singular part of $E_{jk}(\l)$ near the eigenvalue $\l_n$, $ n=1,\dots,\k$, of the operator $H$ can be computed as follows:
\begin{equation}\label{finans}
E_{jk}^{\text{sing}}(\l)=\frac{\rho_n}{\l-\l_n}\Big(\int\limits_{-\infty}^\infty
y_-(\l_n,\xi)^\top\widehat{w}_j(\xi)\,d\xi\Big)^\top\Big(\int\limits_{-\infty}^\infty
y_+(\l_n,\xi)^\bot\widehat{v}_k(\xi)\,d\xi\Big).
\end{equation} 
\end{theorem}
We recall that at $\l=\l_n$ the solutions $y_+(\l_n,\cdot)$ and $y_-(\l_n,\cdot)$ from \eqref{defypmx} are proportional, see \eqref{propJS}, and that the residue $\rho_n$ of the function $1/\cW(\l)$ at the point $\l_n$ is given by
\begin{equation}\label{rhon}
\rho_n=\frac{1}{2\pi i}\int_{\gamma_n}\frac{d\l}{\cW\big(u_-(\l,\cdot), u_+(\l,\cdot)\big)}
\end{equation}
for a sufficiently small circle $\gamma_n$ centered at $\l_n$, see \eqref{defW}.

We now consider approximation \eqref{BVPsch} of the Schr\"odinger equation on the finite segment $[x_-^N,x_+^N]$.
 Let $u_\pm^N$ denote the Jost solutions corresponding to the truncated potential $V^N$ defined on $\R$ by $V^N(x)=V(x)$, $x\in[x_-^N,x_+^N]$ and $V^N(x)=0$ otherwise, that is, the solutions of the truncated Volterra equations
\begin{equation}\label{volttr}u_\pm^N(\l,x)=e^{\pm\il x}-\int_0^{x_\pm^N}\l^{-1/2}\sin(\l^{1/2}(x-\xi))V(\xi)u_\pm^N(\l,\xi)\,d\xi, \, x\in\R.
\end{equation}
Similarly to \eqref{defypmx}, we denote by 
\beq\label{defypmxN}
y_\pm^N(\l,x)=\begin{pmatrix}u^N_\pm(\l,x)\\ (u^N_\pm)'(\l,x)\end{pmatrix},\quad x\in\bbR,\quad\l\in\Omega,
\enq the corresponding solutions of the first order differential equation \eqref{foSch}, \eqref{foSch2}. It is easy to see from \eqref{volttr} that these solutions satisfy the following conditions:
\begin{equation}\label{cypmn}
y_+^N(\l,x_+^N)=e^{\il x_+^N}\bv,\quad
y_-^N(\l,x_-^N)=e^{-\il x_-^N}\bw,
\end{equation}
where $\bv,\bw$ are the eigenvectors of $A(\l,\infty)$ defined after equations \eqref{LS1}, \eqref{LS2}. 

We will now identify the boundary conditions as required in
\eqref{eq:inhomapprox} and \eqref{eq:projbc}. Since
$A_+(\l)=A_-(\l)=A(\l,\infty)$, we have
$P_+(\l)=P_-(\l)=(2\il)^{-1}\bv\bw^\bot$, $\mathcal{R}(P_+(\l))=\Span\{\bv\}$,
$\mathcal{N}(P_+(\l)=\Span\{\bw\}$, and thus the discussion in Section \ref{s:converge} leading to \eqref{eq:projbc} yields
\begin{equation}
R_+(\l)=\frac{1}{2\il}\begin{pmatrix}0&0\\\il&-1\end{pmatrix},\quad
R_-(\l)=\frac{1}{2\il}\begin{pmatrix}\il&1\\0&0\end{pmatrix}.
\end{equation}
In particular, due to \eqref{cypmn} the solutions $y_\pm^N(\l,\cdot)$ from \eqref{defypmxN} satisfy the boundary conditions
\begin{equation}\label{BCR}
R_-(\l)y(\l,x_-^N)+R_+(\l)y(\l,x_+^N)=0
\end{equation}
used in \eqref{eq:inhomapprox} to define the approximate operator pencil $F_N^{(\text{I})}(\l)$.

We are ready to formulate the final convergence result of this section  for the Schr\"odinger case.
Let $y_\pm^N(\l,\cdot)$ from \eqref{defypmxN} be the solutions of the equation \eqref{foSch}, \eqref{foSch2} satisfying
\eqref{volttr}, \eqref{cypmn}, and choose linearly independent functions $\widehat{w}_j,\widehat{v}_k\in L^2(\R,\C^2)$, $j=1,\dots,m$, $k=1,\dots,\ell$. Similarly to \eqref{finans}, at the eigenvalues $\l_n$, $n=1,\dots,\k$, of the operator $H$ we consider
\begin{equation}\label{finansn}
E_{jk}^{N,{\rm sing}}(\l)=\frac{\rho^N_n}{\l-\l_n}\Big(\int\limits_{-\infty}^\infty
y_-^N(\l_n,\xi)^\top\widehat{w}_j(\xi)\,d\xi\Big)^\top\Big(\int\limits_{-\infty}^\infty
y_+^N(\l_n,\xi)^\bot\widehat{v}_k(\xi)\,d\xi\Big),
\end{equation} 
where $\rho_n^N$ is defined by \eqref{rhon} with the Jost solutions $u_\pm(\l,\cdot)$ replaced by $u_\pm^N(\l,\cdot)$.
\begin{theorem}\label{EntoE} Assume $V\in L^1(\R)$.  Then the
  singular part \eqref{cKf} of the matrix $E(\l)$ associated with the
  operator pencil $F^{(\text{I})}(\l)$ on the whole line is the limit as
  $N\to\infty$ of the matrices $E^{N,{\rm sing}}(\l)$ in \eqref{finansn} associated with the operator pencil $F_N^{(\text{I})}(\l)$ on $[x_-^N,x_+^N]$ with the boundary conditions \eqref{BCR}.
\end{theorem}
\begin{proof}
 Since the solutions $y_+^s(\l,\cdot)$ and $y_+^u(\l,\cdot)$ from \eqref{LS1} are linearly independent solutions of \eqref{foSch}, \eqref{foSch2} and $y_+^N(\l,\cdot)$ is a solution of the same equation, we may choose constants $\alpha^N_+$, $\beta^N_+$ such that $y_+^N(\l,x)=\alpha_+^Ny_+^s(\l,x)+\beta_+^Ny_+^u(\l,x)$ for all $x\in\bbR_+$. Solving the last equation for $\alpha_+^N$, $\beta_+^N$ at $x=x_+^N$ and using \eqref{cypmn}, \eqref{xcalc} yields
\begin{equation}\label{abp}\begin{split}
&\alpha_+^N=\frac{\det\big(y_+^N(\l,x_+^N)\big| y_+^u(\l,x_+^N)\big)}{-2\il}=
\frac{\det\big(\bv\big|e^{\il x_+^N} y_+^u(\l,x_+^N)\big)}{-2\il}\to1,\\
&e^{-2\il x_+^N}\beta_+^N=\frac{\det\big(e^{-\il x_+^N}y_+^s(\l,x_+^N)\big|\bv\big)}{-2\il}\to0
\end{split}\end{equation}
as $N\to\infty$ due to \eqref{LS1} (we recall that $\Re(\il)<0$). A similar argument shows that if
$y_-^N(\l,x)=\alpha_-^Ny_-^s(\l,x)+\beta_-^Ny_-^u(\l,x)$ then
\begin{equation}\label{abm}
\alpha_-^N\to 1,\quad e^{2\il x_-^N}\beta_-^N\to0\quad\text{as $N\to\infty$}.
\end{equation}
Computing Wronskians and using \eqref{LSs}, \eqref{abp}, \eqref{abm} we conclude
that 
\[
\cW(y_-^N(\l,\cdot),y_+^N(\l,\cdot))=\cW\big(\alpha_-^Ny_-^s(\l,\cdot)+\beta_-^Ny_-^u(\l,\cdot), \alpha_+^Ny_+^s(\l,\cdot)+\beta_+^Ny_+^u(\l,+)\big)
\]
converges to $\cW(y_-(\l,\cdot),y_+(\l,\cdot))$ as $N\to\infty$ (and even uniformly in $\lambda$ on compacta in $\Omega$ because the convergence in \eqref{LS1}, \eqref{LS2} is uniform on compacta). It follows that  $\rho^N_n\to\rho_n$ as $N\to\infty$.

It remains to show that the integral terms in \eqref{finansn} converge to the respective integral terms in \eqref{finans}.  The latter fact follows from the  assertions
\begin{equation}\label{llm}
\|y_\pm(\l,\cdot)-y_\pm^N(\l,\cdot)\|_{L^2(\R_\pm\cap[x_-^N,x_+^N])}\to0\quad\text{as $N\to\infty$}
\end{equation}
using the  Cauchy-Schwartz inequality.
To prove \eqref{llm} for $\R_+$ (the argument for  $\R_-$ is similar), we recall that $y^s_+(\l,\cdot)=y_+(\l,\cdot)$ by \eqref{LSs}, and then use  the representation $y_+^N(\l,\cdot)=\alpha_+^Ny_+(\l,\cdot)+\beta_+^Ny_+^u(\l,\cdot)$ to estimate
\begin{align*}
\|y_+(\l,\cdot)&-y_+^N(\l,\cdot)\|_{L^2([0,x_+^N])}\le\big|1-\alpha_+^N\big|
\,\|y_+(\l,\cdot)\|_{L^2([0,x_+^N])}\\ &
\qquad+|\beta_N|\big(\int_0^{x_+^N}e^{-2\Re(\il)x}\cdot e^{2\Re(\il)x}|y_+^u(\l,x)|^2\,dx \big)^{1/2}\\
&\le\big|1-\alpha_+^N\big|
\,\|y_+(\l,\cdot)\|_{L^2(\R_+)}+ c|\beta_N|\big(e^{-2\Re(\il)x_+^N}-1\big)^{1/2}
\end{align*}
since the function $e^{\il x}y_+^u(\l,x)$ is bounded on $\R_+$ due to \eqref{LS1}. Now \eqref{abp} implies \eqref{llm} finishing the proof of the 
theorem.
\hfill\end{proof}



\section{Numerical experiments for the FitzHugh-Nagumo equation}
\label{sec:num}
We apply the contour-method to investigate
spectral stability of traveling waves in the FitzHugh-Nagumo system
(FHN). A traveling wave in FHN is a solution of
\begin{equation}\label{eq:fhn01}
  \begin{aligned}
    0&=u''+cu'+u-\tfrac{1}{3}u^3-v,\\
    0&=cv'+\Phi(u+a-bv).
  \end{aligned}
\end{equation}
For the standard parameter values $a=0.7$, $b=0.8$, $\Phi=0.08$
(see \cite{M81},\cite[\S~5]{BL})
one finds both, a stable pulse with speed $c \approx -0.812$ and an
unstable pulse with speed $c\approx-0.514$. We choose the latter one
and denote it by $(\bar{u},\bar{v})$.

Linearization about this pulse leads to the linear operator
\[
\cL \begin{pmatrix}u\\v
\end{pmatrix}
=\begin{pmatrix} u''+cu'+u-\bar{u}^2u-v\\
  cv'+\Phi u-\Phi b v
\end{pmatrix}.\]
For spectral stability one has to analyze the location of the spectrum
of this operator. The eigenvalue problem reads
\[(\lambda I-\cL )\begin{pmatrix} u\\v
\end{pmatrix} = \begin{pmatrix} 0\\0
\end{pmatrix}\;\text{in}\; L^2(\R,\C^2).\]
From the dispersion relation one can show (e.g. \cite[\S~5]{BL} or
\cite[\S~7]{RM12}) that there is
no essential spectrum in $\Omega=\{\Re \lambda > -0.064\}$. 
Thus, we may use the circle $\Gamma=\{\lambda\in
\mathbb{C}:|\lambda-1|=1.05\}$ for the contour method.
Equation \eqref{yeq} becomes
\begin{equation}\label{eq:yeqFHN2ndOrd}
  (\lambda I-\cL )\begin{pmatrix} y_k^1(\lambda)\\y_k^2(\lambda) \end{pmatrix} 
    = \widehat{v}_k \;\text{in}\; L^2(\R,\C^2).
\end{equation}
For the computation we choose
$\widehat{v}_k$ as a random linear combination from the
$2M$-dimensional space
\[\widetilde{\mathcal{K}}_M=
\Span\Bigl\{ \begin{pmatrix} \varphi_j \\ 0 \end{pmatrix},
             \begin{pmatrix} 0 \\ \varphi_j  \end{pmatrix} :
\,j=0,\ldots,M-1\Bigr\} \; \subset L^2(\mathbb{R},\mathbb{C})^2,
\]
\[ \varphi_j(x)=\max\{0,1-|x-x_j|\},
\quad x_j=-5+\frac{10j}{M-1},\,j=0,\dots,M-1.
\]
More precisely, $ \widehat{v}_k=\sum_{j=0}^{M-1}\bigl(\xi^j_k
(\varphi_j,0)^\top +\zeta_k^j (0,\varphi_j)^\top\bigr)$, with
$\xi_k^j$ and $\zeta_k^j$ independent and normally distributed random
variables.

With $z =(y_k^1,(y_k^1)',y_k^2)^\top$ we rewrite \eqref{eq:yeqFHN2ndOrd}
as a first order system (see Section \ref{sec3.2}),
\begin{equation}\label{eq:yeqFHN1stOrd}
\begin{pmatrix}z_1\\z_2\\z_3
\end{pmatrix}'
=\begin{pmatrix}
  z_2\\ 
  \lambda z_1-cz_2-z_1+\bar{u}^2 z_1+z_3\\
  -\tfrac{1}{c}\Phi z_1+\tfrac{1}{c}(\lambda+\Phi b)
  z_3
\end{pmatrix}+
\begin{pmatrix}
  0\\-\widehat{v}_k^1\\-\tfrac{1}{c}\widehat{v}_k^2
\end{pmatrix}
\;\text{in}\; L^2(\R,\C^3).
\end{equation}
For the approximation \eqref{eq:inhomapprox} of \eqref{eq:yeqFHN1stOrd}
on a bounded interval we choose projection and periodic boundary
conditions, which both satisfy \eqref{eq:detcond}. For these 
boundary conditions, we showed in
Theorem~\ref{thm:EConv} exponential error estimates (with a twice as
good rate for the projection boundary conditions), when the
$\widehat{v}_k$ are compactly supported, a property shared by the linear
combinations of hat functions  
$\widehat{v}_k\in\widetilde{\mathcal{K}}_M$.
Furthermore, for the contour method we choose 
\begin{itemize}
  \item $l=10$ right hand sides from $\widetilde{\mathcal{K}}_{40}$,
  \item $m=401$ functionals $\widehat{w}_j=\delta_{x_j}\in
    \mathcal{H}'$, where $\delta_{x_j}\in \mathcal{M}_b^c$ is the
    Dirac measure at $x_j=-2+\frac{j}{100}$, $j=0,\dots,400$,
  \item symmetric finite intervals $J=[-\frac{L}{2},\frac{L}{2}]$ of length $L$,
  \item the number $\varkappa$ for the rank test \eqref{vkrank} such
that the singular values of $D_0^N$ satisfy
 $\sigma_1 \ge \ldots \ge \sigma_{\varkappa} \ge \theta \sigma_1 > \sigma_{\varkappa+1}$ for some given $\theta >0$.
\end{itemize}
There are always two eigenvalues inside the circle, the zero and the unstable
eigenvalue, see Figure \ref{fig:ChoiceOfKappa} (a). We take the unstable
eigenvalue for tests of accuracy. A highly accurate reference eigenfunction
 is computed by applying Newton's 
method to  the discrete  boundary eigenvalue problem (trapezoidal method on a
large interval $[-50,50]$ at small step-size $\Delta x=0.01$).
The same step-size $\Delta x =0.01$ is used when solving 
\eqref{eq:inhomapprox} on equidistant grids for the contour method.

\subsection{Dependence on the interval size}
In our first experiment we vary the length of the interval for the
finite boundary value problems \eqref{eq:inhomapprox}. The other data is
fixed. In particular, we use a large number 
 of quadrature points ($M=100$)  on the contour and  $\theta=10^{-8}$ for determining
 $\varkappa$ as above to keep the influence of the rank test small.
\begin{figure}[tb!]
  \psfrag{Interval length L}{\small Intervall length $L$}
  \psfrag{log10(|la-la(L)|)}{\small $\log_{10}(|\lambda_u-\lambda_u(L)|)$}
  \psfrag{-10}{\small $-10$}
  \psfrag{-8}{\small $-8$}
  \psfrag{-6}{\small $-6$}
  \psfrag{-4}{\small $-4$}
  \psfrag{-2}{\small $-2$}
  \psfrag{0}{\small $0$}
  \psfrag{5}{\small $5$}
  \psfrag{10}{\small $10$}
  \psfrag{15}{\small $15$}
  \psfrag{20}{\small $20$}
  \psfrag{25}{\small $25$}
  \psfrag{30}{\small $30$}
  \begin{center}
    \includegraphics[scale=0.7]{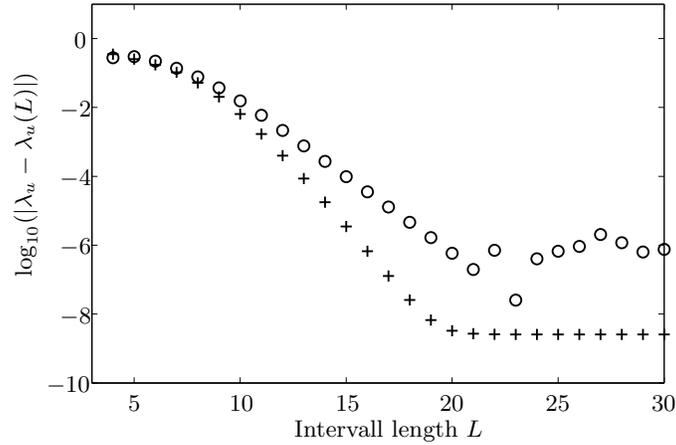}
  \end{center}
  \caption{Convergence of the approximate unstable eigenvalue for
  projection bc's ($+$) and periodic bc's ($\circ$).}
  \label{fig:EvalConv}
\end{figure}

In Figure~\ref{fig:EvalConv} we plot the distance of the approximate
unstable eigenvalue $\lambda_u(L)$ obtained by the contour method to the
reference value $\lambda_u$. For the contour method
\eqref{eq:inhomapprox} is solved on $[-\frac{L}{2},\frac{L}{2}]$ with
periodic ($\circ$) and projection boundary conditions ($+$).  For both
 boundary conditions one finds
 an exponential rate of convergence with a significantly better rate
for the second one, as predicted by Theorem~\ref{thm:EConv}.
\begin{figure}[tb!]
  \footnotesize
  \psfrag{Interval length L}{ Intervall length $L$}
  \psfrag{log10(angle(V,V(L)))}{ $\log_{10}(\angle(y_u,y_u(L)))$}
  \psfrag{-10}[cc]{ $-10$}
  \psfrag{-8}[cc]{ $-8$}
  \psfrag{-6}[cc]{ $-6$}
  \psfrag{-4}[cc]{ $-4$}
  \psfrag{-2}[cc]{ $-2$}
  \psfrag{0}[cc]{ $0$}
  \psfrag{5}[cc]{ $5$}
  \psfrag{10}[cc]{ $10$}
  \psfrag{15}[cc]{ $15$}
  \psfrag{20}[cc]{ $20$}
  \psfrag{25}[cc]{ $25$}
  \psfrag{30}[cc]{ $30$}
  \begin{center}
    \includegraphics[scale=0.7]{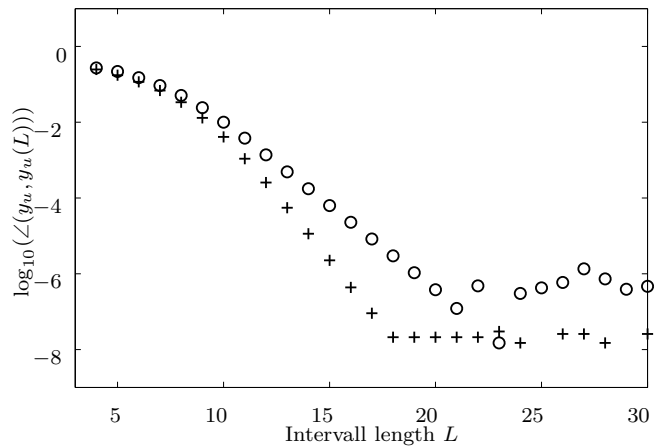}
  \end{center}
  \caption{Convergence of the approximate eigenfunction for projection
  boundary conditions ($+$) and periodic bc's ($\circ$).}
  \label{fig:EvectConv}
\end{figure}

The same observation is true for the convergence of the
eigenfunction. This is shown in Figure~\ref{fig:EvectConv}, where we
compare the angle between the approximate eigenfunction $y_u(L)$ and the reference
eigenfunction $y_u$. For the approximate eigenfunction, we use
\eqref{biorth}--\eqref{vapprox2} with  hat functions
$\widehat{u}_k(x)=\max\{0,1-|k-200-100x|\}$, $k=0,1,\dots,400$, which, together
with  $\widehat{w}_k$, form a biorthogonal system.

In Figure~\ref{fig:EvectPlot} we plot the $u$-component of the reference
eigenfunction (which is actually an approximation on $[-50,50]$) and
compare it to the approximate eigenfunctions obtained for the different
interval lengths $L=5,10,15$ with projection bc's. The 
approximate eigenfunctions are only shown on  $[-2,2]$ which is contained
in all domains of definition.
\begin{figure}[tb!]
  \footnotesize
  \psfrag{-6}{$-6$}
  \psfrag{-4}{$-4$}
  \psfrag{-2}{$-2$}
  \psfrag{0}{$0$}
  \psfrag{2}{$2$}
  \psfrag{4}{$4$}
  \psfrag{6}{$6$}
  \psfrag{u(x)}{}
  \psfrag{x}{$x$}
  \psfrag{Global1}{$y_u$}
  \psfrag{l=5}{$L=5$}
  \psfrag{l=10}{$L=10$}
  \psfrag{l=15}{$L=15$}
  \begin{center}
    \includegraphics[scale=0.7]{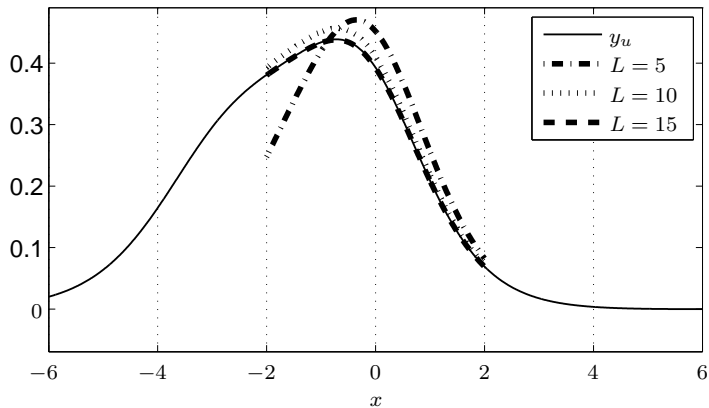}
  \end{center}
  \caption{First component of $y_u$ and its approximation by the contour
  method for different interval sizes.}
  \label{fig:EvectPlot}
\end{figure}

\subsection{Dependence on the number of quadrature points}
In our second experiment we fix the interval to $[-50,50]$ with
projection boundary conditions and vary the number of quadrature
points. As functionals $\widehat{w}_j$ we choose the
point evaluation at all grid points. We determine the rank as before and
take $\theta=10^{-10}$ to keep the influence of the rank test small.
All other data are the same as in the first experiment.
The results are shown in Figure~\ref{fig:quadConv}.
\begin{figure}[tb!]
  \begin{center}
    \footnotesize
    \psfrag{log10(|la-la(h)|)}{$\log_{10}(|\lambda_u-\lambda_u(M)|)$}
    \psfrag{log10(angle(V,V(h)))}{$\log_{10}(\angle(y_u,y_u(M)))$}
    \psfrag{Number of quadrature points}{Number of quadrature points}
    \subfigure[Eigenvalue]{
    \includegraphics[width=0.45\textwidth]{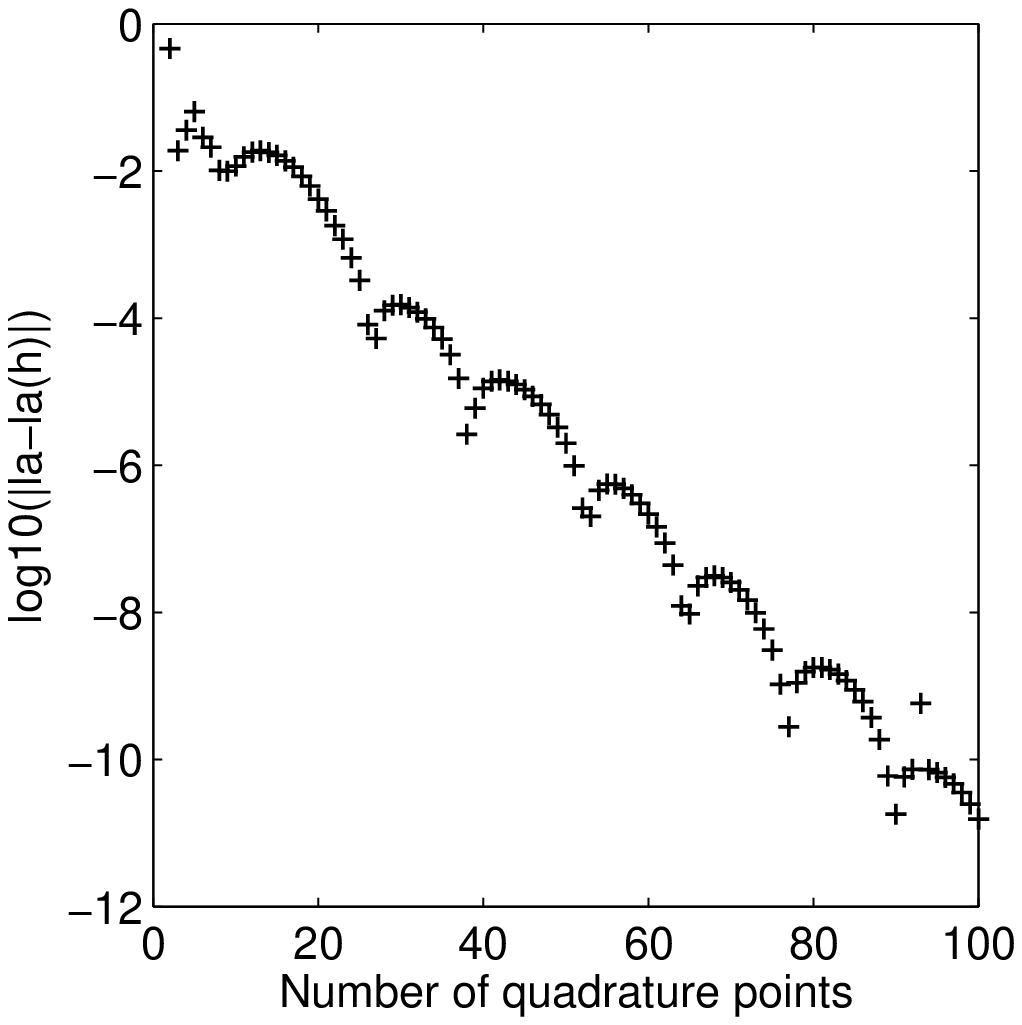}
    }~
    \subfigure[Eigenfunction]{
    \includegraphics[width=0.45\textwidth]{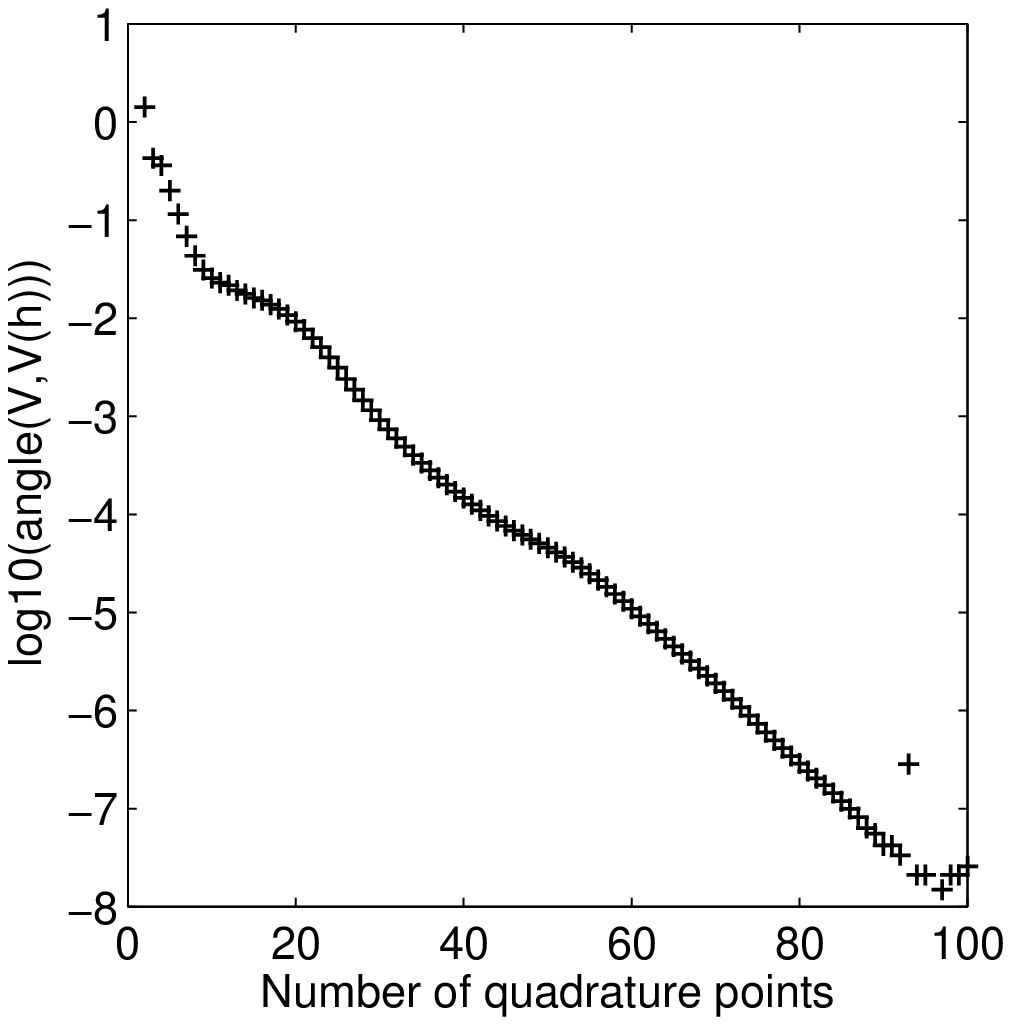}
    }
  \end{center}
  \caption{Convergence of eigenvalue and eigenfunction with
  increasing number of quadrature points.}
  \label{fig:quadConv}
\end{figure}

One observes exponential convergence rate with respect to the number
of quadrature points, see \cite{B12} for a proof. For the
eigenvalue errors there are some apparent resonances which have not yet been
inverstigated further. It turns out that quadrature errors 
dominate in this case and that error plots are almost identical for
periodic boundary conditions.

\subsection{Dependence on the rank test}
We now keep all data fixed but vary the rank test by prescribing
the value of $\varkappa$.

All other data are the same as in the previous experiment except
that we choose $45$ quadrature points on the contour.
In Table (c) from Figure~\ref{fig:ChoiceOfKappa} we list all singular
values of the numerical approximation $D_0^N$ in this case.
In Figures~\ref{fig:ChoiceOfKappa} (a) and (b) we plot the approximate
eigenvalues for $\varkappa=2$ and $\varkappa=10$, respectively.

It turns out that the two eigenvalues inside the circle are nearly
independent of $\varkappa$ for $\varkappa\ge 2$, while the  eigenvalues
outside heavily depend on $\varkappa$.

It is shown in \cite{B12} that values outside but close to the contour
still represent good approximations of eigenvalues. In our example,
however, these eigenvalues are generated by the essential spectrum
of the continuous problem which lies very close to the contour.
This case is not covered by the analysis in \cite{B12} and requires
further investigations.
\begin{figure}[tb!]
  \begin{center}
    \footnotesize
    \psfrag{R}[r]{$\mathbb{R}$}
    \psfrag{iR}[cc]{$i\mathbb{R}$}
    \psfrag{1}[cc]{$1$}
    \psfrag{0.5}[cc]{$0.5$}
    \psfrag{-0.5}[cc]{$-0.5$}
    \psfrag{-1}[cc]{$-1$}
    \psfrag{0}[cc]{$0$}
    \psfrag{1.5}[cc]{$1.5$}
    \psfrag{2}[cc]{$2$}
    \subfigure[$\varkappa=2$]{
    \includegraphics[width=0.45\textwidth]{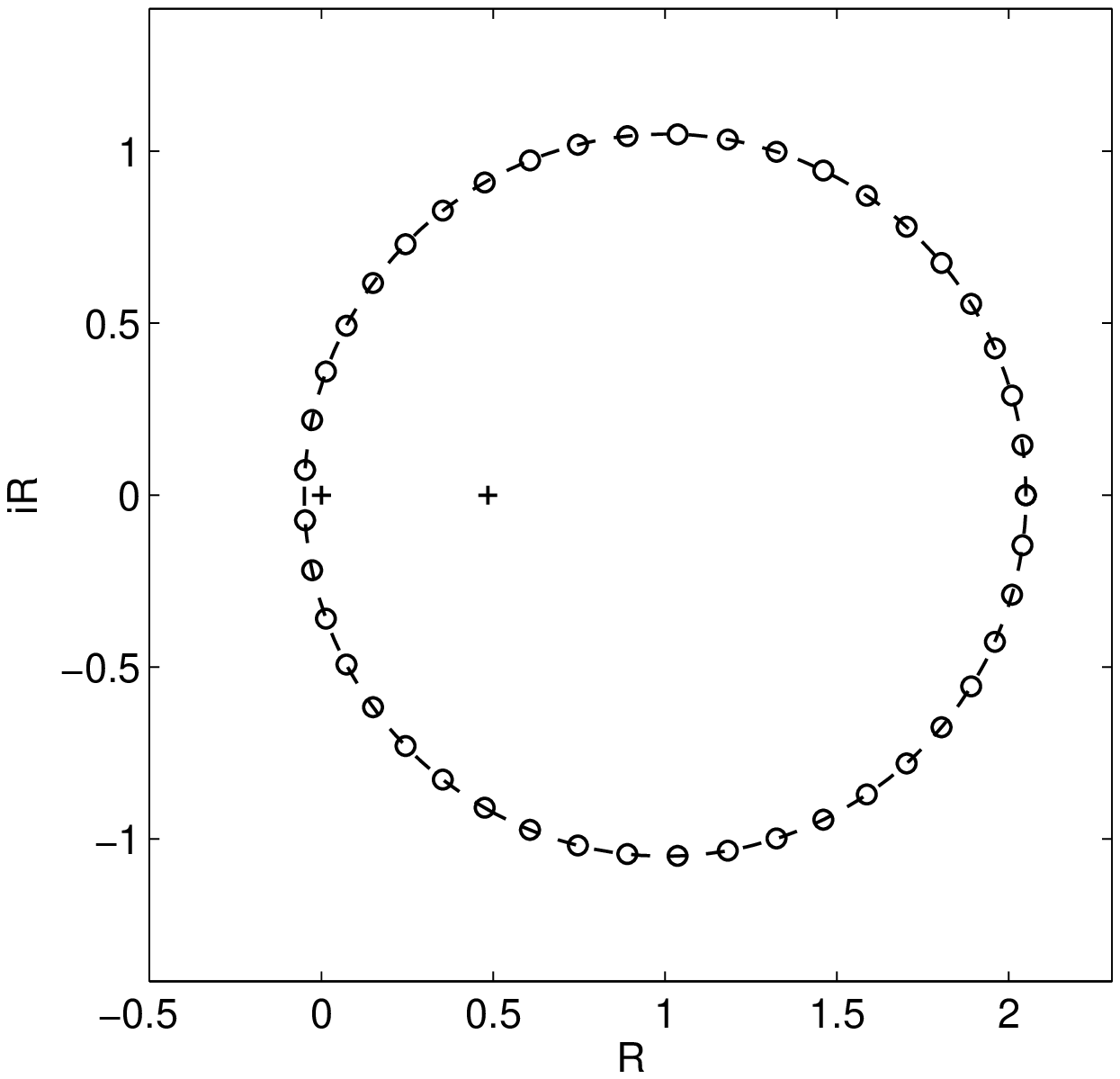}
    \label{fig:2svs}}~
    \subfigure[$\varkappa=10$]{
    \includegraphics[width=0.45\textwidth]{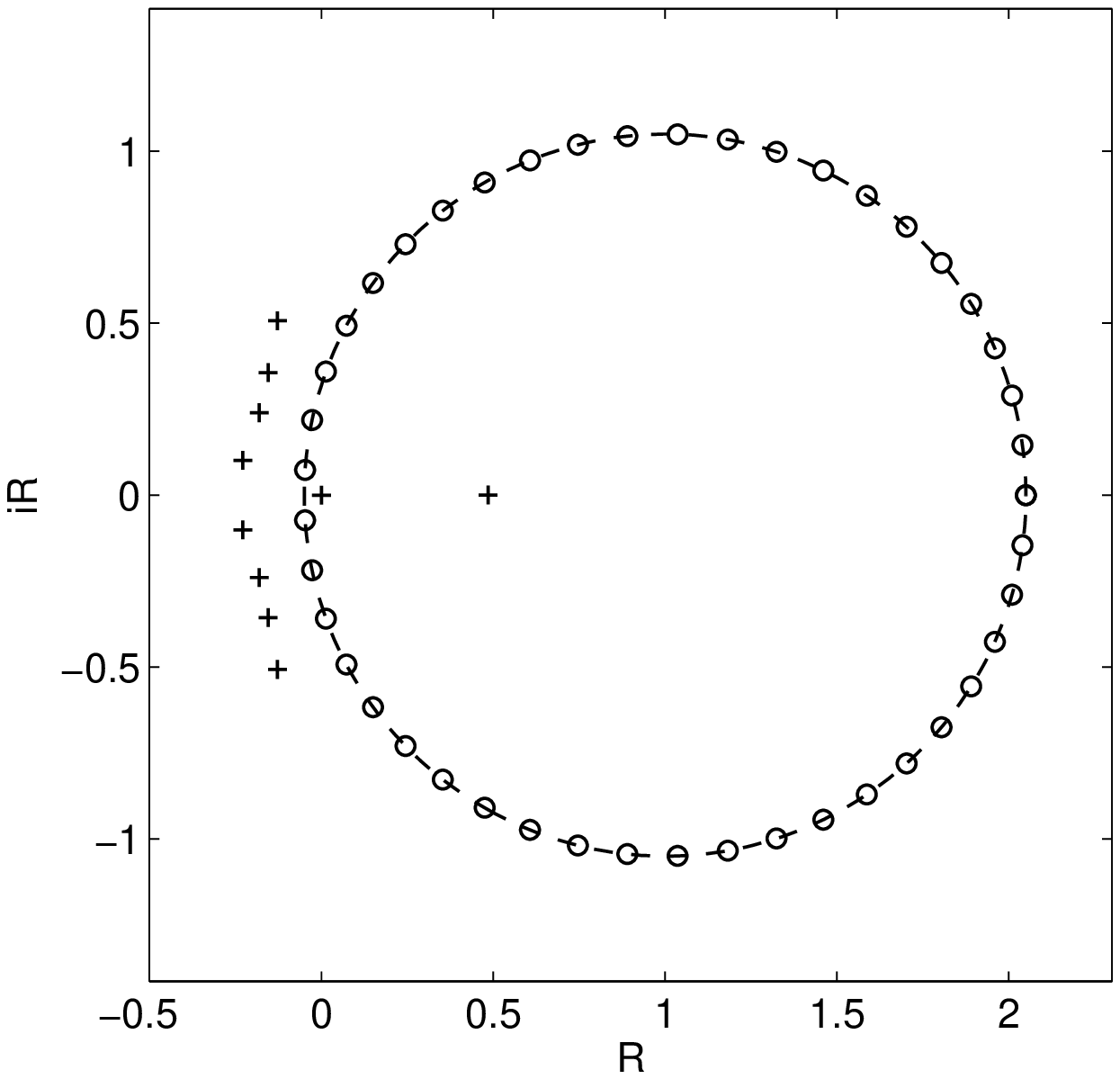}
    \label{fig:10svs}}\\
    \subfigure[Singular values of $D_0^N$.]{
    \footnotesize{
    \begin{tabular}[b]{rrrrrrrrrr}
      1.&2.&3.&4.&5.&6.&7.&8.&9.&10.\\
      $103$&$89$&$1.25$&$0.276$&$0.133$&$0.079$&$0.034$&$0.0066$&$0.0033$&$0.00077$
    \end{tabular}}}\label{tab:svs}
  \end{center}
  \caption{Influence of the rank test on the spectrum for two
  different values of $\varkappa$ (see \eqref{vkrank}). $\circ$: Quadrature
  points, $+$: approximate eigenvalues.}
  \label{fig:ChoiceOfKappa}
\end{figure}

\appendix 

\section{Embedding results for the function spaces $\mathcal{H}$ and $\mathcal{H}_J$}
The aim of this appendix is to show several embedding properties for the
function spaces used in Sections \ref{sec3.2} and \ref{s:converge}.
Recall the Banach spaces \eqref{domainF}, \eqref{eq:HJdef}, i.e.
\[\mathcal{H}=\big\{y\in L^2(\mathbb{R},\mathbb{C}^d):y\in AC_{\loc},
-y'+By\in L^2(\mathbb{R},\mathbb{C}^d)\big\}.\]
with  norm
$\|y\|_{\mathcal{H}}^2=\|y\|_{L^2}^2+\|-y'+By\|_{L^2}^2$ and
\[\mathcal{H}_J=\big\{y\in L^2(J,\mathbb{C}^d):y\in AC(J),
-y'+By\in L^2(J,\mathbb{C}^d)\big\},\]
with norm
$\|y\|_{\mathcal{H}_J}^2=\|y\|_{L^2(J)}^2+\|-y'+By\|_{L^2(J)}^2$.

In the following we assume $B\in L^1(\R,\C^{d,d})$ and let $T(x,x_0)$, $x,x_0\in\R$, denote the solution operator of $-y'+B(x)y$ (that is, the propagator of the differential equation $y'=B(x)y$ on $\R$) in the sense of Carath\'eodory, i.e.\
$T(\cdot,x_0)$ is a mild solution of the initial value problem
$Y'=B(x)Y$, $Y(x_0)=I_d$
in the sense of Carath\'eodory, or equivalently
\begin{equation}\label{eq:A1}
  T(x,x_0)=I_d+\int_{x_0}^x
  B(\xi)T(\xi,x_0)\,d\xi,\; x,x_0\in
  \mathbb{R}.
\end{equation}
Then $T\in C(\mathbb{R}^2,\mathbb{C}^{d,d})$ and
$T(\cdot,x_0)\in AC_{\loc}(\mathbb{R})$. Gronwall's inequality yields
for all $x,x_0\in \mathbb{R}$:
\begin{equation}\label{eq:A2}
  \|T(x,x_0)\|\le
  \exp\big(\int_{x_0}^x\|B(\xi)\|\,d\xi\big)
  \le \exp\big(\|B\|_{L^1}\big)=:K.
\end{equation}
Every $y\in\mathcal{H}_J$ for $J\subseteq\R$ satisfies the equation  $-y'+By=z$ on $J$ for some $z\in L^2(J,\C^d)$
and hence for all $x,x_0\in J\subseteq\R$ we have
\begin{equation}\label{eq:A3}
  y(x)=T(x,x_0)y(x_0)-\int_{x_0}^x T(x,\xi)z(\xi)\,d\xi.
\end{equation}
This implies for all $x,x_0 \in J\subseteq\R$
the estimate
\begin{equation}\label{eq:A4}
  |y(x)-y(x_0)|\le 
  K \|B\|_{L^1([x,x_0])}|y(x_0)|
  + \sqrt{|x-x_0|} K \|y\|_{\mathcal{H}_J},
\end{equation}
where we use $[x,x_0]$ to denote the interval $[x_0,x]$ in case $x_0<x$.
Indeed, \eqref{eq:A4} follows from
\begin{equation}\label{eq:A5}
  \big\|T(x,x_0)-I_d\big\|\le K
  \|B\|_{L^1([x,x_0])},
\end{equation}
and
\begin{equation}\label{eq:A6}
  \big\|\int_{x_0}^x T(x,\xi)z(\xi)\,d\xi\big\|
  \le \sqrt{|x-x_0|} K \|z\|_{L^2},
\end{equation}
which are easily obtained from \eqref{eq:A1} and \eqref{eq:A2}.

Our first lemma shows asymptotic decay of elements in
$\mathcal{H}$, that is, embedding of $\mathcal{H}$ in the space $C_0(\R,\C^d)=\{ y \in C(\R,\C^{d}): \lim_{|x|\rightarrow \infty}y(x)=0 \}$. We supply a short proof and refer to \cite[Lem.3.16]{CL} for a more general case.
\begin{lemma}\label{lem:A1} Assume $B\in L^1(\R,\C^{d,d})$. Then
  $\mathcal{H}\subset  C_0(\R,\C^d)$ and $\cH_J\subset C(J,\C^d)$.
\end{lemma}
\begin{proof}
Inclusions  $\mathcal{H}\subset C(\mathbb{R},\C^d)$ and  $\cH_J\subset C(J,\C^d)$ follow from \eqref{eq:A3}. Assume that for some
   $y\in \mathcal{H}$ there is a sequence
  $(x_N)_{N\in \mathbb{N}}\subset\R$ with
  $|x_N|\to\infty$, so that $|y(x_N)|\ge \nu >0$. Without loss of
  generality we may assume $x_{N+1}\ge x_N+1$. Now 
let $\delta_0=\min\{\frac{1}{2},\nu^2(3K\|y\|_{\mathcal{H}})^{-2}\}$.
Since $B\in L^1(\mathbb{R},\mathbb{C}^{d,d})$
  we can choose $0<\delta_1\le \delta_0$, such that
$K \int_\mathcal{M} \|B(x)\|\,dx\le \frac{1}{3}$ for all measurable
 $\mathcal{M}\subset  \mathbb{R}$ with $\meas(\mathcal{M})\le 2\delta_1$.
   For $|x-x_N|\le \delta_1$ inequality \eqref{eq:A4} then implies
  \[|y(x)|\ge
  |y(x_N)|-\frac{1}{3}|y(x_N)|-\frac{\nu}{3}\ge\frac{\nu}{3}.\]
  This leads to a contradiction via the estimate
  \[\|y\|_{\mathcal{H}}^2\ge \int_{\mathbb{R}}|y(x)|^2\,dx\ge
  \sum_{N=0}^\infty \int_{x_N-\delta_1}^{x_N+\delta_1}|y(x)|^2\, dx
  =\infty.\]
 \hfill \end{proof}
Next, we show that embedding of $\mathcal{H}_J$ in $L^\infty(J,\C^d)$ is uniform for all $J\subseteq\R$.
\begin{lemma}\label{lem:A2}
There exists a constant $C>0$ such that for all intervals
  $J=[a,b]\subset \mathbb{R}$ with $b-a\ge1$ and for $J=\R$,
  \begin{equation}\label{eq:A7}
    \|y\|_{L^\infty(J)}\le C\|y\|_{\mathcal{H}_J},\; \text{ for all}\; y\in
    \mathcal{H}_J.
  \end{equation}
\end{lemma}
\begin{proof}
  Let $y\in \mathcal{H}_J$, then
  $\|y\|_{L^\infty}<\infty$ and there is
  $\bar{x}\in J$ with $|y(\bar{x})|=\|y\|_{L^\infty}$
  by Lemma \ref{lem:A1}.
  Let $\delta\in(0, \min(1,K^{-2}))$ be so small that 
  $K \int_\mathcal{M} |B(x)|\, dx \le \frac{1}{2}$ 
  holds for all measurable $\mathcal{M}\subset \mathbb{R}$ with 
  $\meas(\mathcal{M})\le \delta$. Let
  $C=2\sqrt{(2+2\delta)/{\delta}}$. 
  If $\|y\|_{L^\infty}^2\le 8 \|y\|_{\cH_J}^2$ then \eqref{eq:A7} holds
  since $C>\sqrt{8}$. If
  $|y(\bar{x})|^2> 8 \|y\|_{\cH_J}^2$ then,
  by \eqref{eq:A4}, for all $x\in J$
  with $|x-\bar{x}|\le \delta$ we have
  \[|y(x)|\ge |y(\bar{x})|-\frac{1}{2}|y(\bar{x})|-\sqrt{\delta_1}
  K \|y\|_{\mathcal{H}_J}\ge
  \frac{1}{2}|y(\bar{x})|-\|y\|_{\mathcal{H}_J}>0.\]
  From this  we obtain
  \begin{multline*}
    \|y\|_{\mathcal{H}_J}^2\ge \int_J |y(x)|^2\,dx\ge \int_{J\cap
    [\bar{x}-\delta,\bar{x}+\delta]} \big(\tfrac{1}{2}|y(\bar{x})|-
    \|y\|_{\mathcal{H}_J}\big)^2\, dx\\
    \ge \delta \left(
    \frac{\|y\|_{L^\infty}^2}{4}-\|y\|_{L^\infty}\|y\|_{\mathcal{H}_J}
    +\|y\|_{\mathcal{H}_J}^2\right)
    \ge \delta \left(
    \frac{\|y\|_{L^\infty}^2}{8}-\|y\|_{\mathcal{H}_J}^2\right),
  \end{multline*}
and \eqref{eq:A7} follows.
\hfill\end{proof}



\end{document}